\documentclass{preprint}

\usepackage{amsfonts,amsmath, times}
\newtheorem{theorem}{Theorem}[section]
\newtheorem{lemma}[theorem]{Lemma}
\newtheorem{proposition}[theorem]{Proposition} 
\newtheorem{corollary}[theorem]{Corollary}

\newtheorem{remark}[theorem]{Remark}
\newtheorem{definition}{Definition}[section]

\newtheorem{assumption}{Assumption}[section]
\numberwithin{equation}{section}

\begin{document}

\def\paral{/\kern-0.55ex/}
\def\parals_#1{/\kern-0.55ex/_{\!#1}}
\def\bparals_#1{\breve{/\kern-0.55ex/_{\!#1}}}
\def\n#1{|\kern-0.24em|\kern-0.24em|#1|\kern-0.24em|\kern-0.24em|}
\newenvironment{proof}{
 \noindent\textbf{Proof}\ }{\hspace*{\fill}
  \begin{math}\Box\end{math}\medskip}

\newcommand{\A}{{\bf \mathcal A}}
\newcommand{\B}{{\bf \mathcal B}}
\def\C{\mathbb C}
\newcommand{\D}{{\rm I \! D}}
\newcommand{\dom}{{\mathcal D}om}
\newcommand{\pathR}{{\mathcal{\rm I\!R}}}
\newcommand{\Nabla}{{\bf \nabla}}
\newcommand{\E}{{\mathbf E}}
\newcommand{\Epsilon}{{\mathcal E}}
\newcommand{\F}{{\mathcal F}}
\newcommand{\G}{{\mathcal G}}
\def\g{{\mathfrak g}}
\newcommand{\HH}{{\mathcal H}}
\def\h{{\mathfrak h}}
\def\k{{\mathfrak k}}
\newcommand{\I}{{\mathcal I}}
\def\m{{\mathfrak m}}
\newcommand{\K}{{\mathcal K}}
\newcommand{\p}{{\mathbb P}}
\newcommand{\R}{{\mathbb R}}
\newcommand{\Rc}{{\mathcal R}}
\def\T{{\mathcal T}}
\def\M{{\mathcal M}}
\def\N{{\mathcal N}}
\newcommand{\pnabla}{{\nabla\!\!\!\!\!\!\nabla}}
\def\X{{\mathbb X}}
\def\Y{{\mathbb Y}}
\def\L{{\mathcal L}}
\def\1{{\mathbf 1}}
\def\half{{ \frac{1}{2} }}
\def\vol{{\mathop {\rm vol}}}

\def\ad{{\mathop {\rm ad}}}
\def\Conj{{\mathop {\rm Ad}}}
\def\Ad{{\mathop {\rm Ad}}}
\newcommand{\const}{\rm {const.}}
\newcommand{\eg}{\textit{e.g. }}
\newcommand{\as}{\textit{a.s. }}
\newcommand{\ie}{\textit{i.e. }}
\def\s.t.{\mathop {\rm s.t.}}
\def\esssup{\mathop{\rm ess\; sup}}
\def\Ric{\mathop{\rm Ric}}
\def\div{\mathop{\rm div}}
\def\ker{\mathop{\rm ker}}
\def\Hess{\mathop{\rm Hess}}
\def\Image{\mathop{\rm Image}}
\def\Dom{\mathop{\rm Dom}}
\def\id{\mathop {\rm Id}}
\def\Image{\mathop{\rm Image}}
\def\Cyl{\mathop {\rm Cyl}}
\def\Conj{\mathop {\rm Conj}}
\def\Span{\mathop {\rm Span}}
\def\trace{\mathop{\rm trace}}
\def\ev{\mathop {\rm ev}}
\def\Ent{\mathop {\rm Ent}}
\def\tr{\mathop {\rm tr}}
\def\graph{\mathop {\rm graph}}
\def\loc{\mathop{\rm loc}}
\def\so{{\mathfrak {so}}}
\def\su{{\mathfrak {su}}}
\def\u{{\mathfrak {u}}}
\def\o{{\mathfrak {o}}}
\def\pp{{\mathfrak p}}
\def\gl{{\mathfrak gl}}
\def\hol{{\mathfrak hol}}
\def\z{{\mathfrak z}}
\def\t{{\mathfrak t}}
\def\<{\langle}
\def\>{\rangle}
\def\span{{\mathop{\rm span}}}
\def\range{{\mathop{\rm range}}}
\def\dim{{\mathop{\rm dim}}}
\def\coker{{\mathop {\rm Coker}}}
\def\index{\mathop {{\rm index}}}
\def\diam{\mathrm {diam}}
\def\inj{\mathrm {inj}}
\def\Lip{\mathrm {Lip}}
\def\Iso{\mathrm {Iso}}
\def\Osc{\mathop{\rm Osc}}
\def\P{\mathbb P}
\renewcommand{\thefootnote}{}
\def\d{\; d}

\author{Xue-Mei Li}
\institute{Mathematics Institute, The University of Warwick, Coventry CV4 7AL, U.K.\\
Email address: xue-mei.li@warwick.ac.uk}
\title{ Limits of Random Differential Equations on Manifolds}
\date{}
\maketitle

\begin{abstract}
Consider a family of random ordinary differential equations on a manifold driven by vector fields of the form $\sum_kY_k\alpha_k(z_t^\epsilon(\omega))$ where $Y_k$ are vector fields, $\epsilon$ is a positive number, $z_t^\epsilon$ is a  ${1\over \epsilon} \L_0$ diffusion process taking values in possibly a different manifold, $\alpha_k$ are annihilators of $\ker (\L_0^*)$. Under H\"ormander type conditions on $\L_0$ we prove that, as $\epsilon $ approaches zero,
the stochastic processes $y_{t\over \epsilon}^\epsilon$ converge weakly and in the Wasserstein topologies. We describe this limit and give an  upper bound for the rate of the convergence. 
 \end{abstract}
\medskip


\vspace{1em}
\hrule
\begin{center} 
Note added in Proof 
\end{center}

\noindent In the published version of this paper,  some important assumptions were
omitted in Lemmas~10.1 and 10.2. I removed these lemmas here.  They should be replaced by conditional Doob's inequality for 
stochastic integrals of functions of stochastic flows. 
\vspace{1em}
\hrule

\section{Introduction}

 Let $M$ and $G$ be finite dimensional smooth manifolds. Let $Y_k$, $k=1,\dots, m$, be $C^6$  vector fields on $M$, $\alpha_k$ real valued $C^r$ functions on $G$, $\epsilon$ a positive number, and
$(z_t^\epsilon)$  diffusions on a filtered probability space $(\Omega, \F, \F_t, \P)$ with values in $G$ and
 infinitesimal generator $\L_0^\epsilon={1\over \epsilon} \L_0$ which will be made precise later. 
 The aim of this paper is to study limit theorems associated to the system of  ordinary differential equations on $M$, 
 \begin{equation}
\label{1}
\dot y_t^\epsilon(\omega)=\sum_{k=1}^m Y_k\left(y_t^\epsilon(\omega)\right)\alpha_k(z_t^\epsilon(\omega))
\end{equation}
under the assumption that $\alpha_k$ `averages' to zero. The `average' is with respect to the unique invariant probability measure of $\L_0$, in case $\L_0$ satisfies strong H\"ormander's condition, and more generally the `average' is the projection to a suitable function space. We prove that $y_{t\over \epsilon}^\epsilon$ converges as $\epsilon\to 0$ to a Markov process whose Markov generator has an explicit expression.

 This study is motivated by problems arising from stochastic homogenization.   It turned out  that in the study
 of randomly perturbed systems with a conserved quantity, which does not necessarily take value in a linear space, the reduced equations for the slow variables can sometimes be transformed into (\ref{1}).   
Below, in section \ref{section-Examples} we illustrate this by 4 examples including one on the orthonormal frame bundle over a Riemannian manifold. Of these examples, the first is from \cite{Li-geodesic} where we did not  
know how to obtain a rate of convergence, and the last three from \cite{Li-homogeneous}
where a family of interpolation equations on homogeneous manifolds are introduced. An additional example can be found in \cite{LI-OM-1}.

\subsection{Outline of the Paper}
In all the examples, which we described in \S\ref{section-Examples} below, the scalar functions average to $0$ with respect to a suitable probability measure on $G$.
Bearing in mind that if a Hamiltonian system
 approximates a physical system with error $\epsilon$ on a compact time interval, over a time interval of size ${1\over \epsilon}$ the physical orbits deviate visibly from that of the Hamiltonian system unless the error is reduced by oscillations, it is natural and a classical problem to study ODEs whose right hand sides are random and whose averages in time are zero. 

The objectives of the present article are: (1)  to prove that, as $\epsilon$ tends to zero, the law of $(y_{s\over \epsilon}^\epsilon, s \le t)$ converges weakly to a probability measure
$\bar \mu$ on the path space over $M$ and to describe the properties of the limiting Markov 
semigroups; 
(2) to estimate the rate of convergence, especially in the Wasserstein distance. 
 For simplicity we assume that all the equations are complete. 
In sections \ref{section-formula}, \ref{section-uniform-estimates},
\ref{section-weak} and \ref{section-rate} we assume that $\L_0$ is  a regularity improving Fredholm operator
on a compact manifold $G$, see Definition \ref{def-Fredholm}. In Theorem \ref{thm-weak} we assume,
 in addition, that $\L_0$ has Fredholm index $0$.  But strong H\"ormander's condition can be used to replace
 the condition `regularity improving Fredholm operator of index $0$'.

For simplicity, throughout the introduction, $\alpha_k$ are bounded and belong to 
$ N^\perp$ where $N$ is the kernel of $\L_0^*$, the adjoint of the unbounded operator $\L_0$ in $L^2(G)$ with respect to the volume measure. In case $\L_0$ is not elliptic we assume in addition that $r\ge 3$ or $r\ge \max{\{3, {n\over 2} +1\}}$, depending on the result. 
The growth conditions on $Y_k$ are  in terms of a  control function $V$ and a controlled function space $B_{V,r}$ where $r$  indicates the order of the derivatives to be controlled, see (\ref{BVr}). For simplicity we assume both $M$ and $G$ are compact. 
  %Depending on the assumptions, we have two methods for the convergence which we describe below. 

%Some aspects of the results and methodology appear to be new when $M=\R^n$ or when $\L_0$ is the Laplace-Beltrami operator.  An elliptic diffusion operator is, naturally,  regularity improving in the terminology of 
%this paper.  It has Fredholm index $0$ if it is self adjoint. 
% On a compact manifold it has a unique invariant probability measure.

In Section \ref{section-estimates} we present two elementary lemmas, Lemma \ref{lemma3}
and Lemma \ref{lemma3-2}, assuming $\L_0$ mixes exponentially in a weighted total 
variation norm with weight $W: G\to \R$.
  In  Section \ref{section-formula}, for $\L_0$ a regularity improving Fredholm operator and $f$ a $C^2$ function, 
  we deduce a formula for $f(y_{t\over \epsilon}^\epsilon)$ where the transmission of the  randomness from the fast motion $(z_t^\epsilon)$ to the slow motion $(y_t^\epsilon)$ is manifested in a martingale. This provides a platform
  for the uniform estimates over large time intervals,  weak convergences, and the study of rate of convergence in later sections.
  
In Section \ref{section-uniform-estimates}, we obtain uniform estimates
 in $\epsilon$ for functionals of $y_{t}^\epsilon$ over  $[0,{1\over \epsilon}]$.
 Let $\L_0$ be a regularity improving Fredholm operator,  $y_0^\epsilon=y_0$, and $V$ a $C^2$ function such that $\sum_{j=1}^m|L_{Y_j}V|\le c+KV$, $\sum_{i,j=1}^m|L_{Y_i}L_{Y_j}V|  \le c+KV$ for some numbers $c$ and $K$.
Then, Theorem \ref{uniform-estimates},  for every  numbers $p\ge 1$ there exists a positive number $\epsilon_0$ such that
  $\sup_{0<\epsilon\le \epsilon_0}\E \sup_{0\le u\le t} V^p(y_{u\over \epsilon}^\epsilon)$ is finite and belongs
  to $B_{V,0}$ as a function of $y_0$.
This leads to convergence in the Wasserstein distance and will be used later to prove a key lemma on averaging functions along the paths of $(y_t^\epsilon, z_t^\epsilon)$.

In Section \ref{section-weak}, $\L_0$ is an operator on a compact manifold $G$ satisfying 
H\"ormander's condition and with Fredholm index $0$; $M$ has positive injectivity radius and other geometric restrictions. In particular we do not make any assumption on the ergodicity of $\L_0$. 
%One key ingredient for this is formula (\ref{Ito-tight}), the other is sub-elliptic estimates.  
Let
$\overline{\alpha_i\beta_j}$ denote $\sum_{l} u_l\<\alpha_i\beta_j, \pi_l\>$ where $\{u_l\}$ is a basis of the kernel of
$\L_0$ and $\{\pi_l\}$ the dual basis in the kernel of $\L_0^*$.
   Theorem \ref{thm-weak} states that, given bounds on $Y_k$ and its derivatives and for
 $\alpha_k\in C^r$ where $r\ge  \max{\{3, {n\over 2}+1\}}$, $(y_{s\over \epsilon}^\epsilon, s \le t)$ converges weakly, as $\epsilon \to 0$,  to the Markov process with Markov generator  $\bar \L=\sum_{i,j=1}^m \overline{\alpha_i\beta_j} L_{Y_i}L_{Y_j}$. This follows from a tightness result, Proposition \ref{tightness} where no assumption on the Fredholm index of $\L_0$ is made, and  a law of large numbers for sub-elliptic operators on compact manifolds, Lemma \ref{lln}. Convergences of $\{(y_{t\over \epsilon}^\epsilon, 0\le t\le T)\}$  in the Wasserstein $p$-distance are also obtained.

In Section \ref{section-sde} we study the solution flows of SDEs and their associated Kolmogorov equations, 
to be applied to the limiting operator $\bar \L$ in Section \ref{section-rate}. Otherwise this section is independent of the rest of the paper. Let $Y_k, Y_0$ be $C^6$ and $C^5$ vector fields respectively. 
 If $M$ is compact, or more generally if $Y_k$ are $BC^5$ vector fields, the conclusions in this section holds, trivially. Denote $B_{V,4}$ the set of functions whose derivatives up to order $r$ are controlled by a function $V$, c.f.(\ref{BVr}).
Let $\Phi_t(y)$ be the solution flow to 
$$dy_t=\sum_{k=1}^m Y_k(y_t)\circ dB_t^k+Y_0(y_t)dt.$$
Let $P_t f(y)=\E f(\Phi_t(y))$ and
$Z={1\over 2}\sum_{k=1}^m \nabla_{Y_k} Y_k+Y_0$. Let $V\in C^2(M, \R_+)$ and 
  $\sup_{s\le t}\E V^q(\phi_s(y))\in B_{V,0}$ for every $q\ge 1$. This assumption on $V$ is implied by the following conditions:
$|L_{Y_i}L_{Y_j} V| \le c+KV$, $|L_{Y_j}V|\le c+KV$, where $C, K$ are constants.
Let $\tilde V=1+\ln (1+|V|)$. We assume, in addition,  for some number $c$ the following hold:
\begin{equation}
\label{control-conditions}
{\begin{split}
&\sum_{k=1}^m\sum_{\alpha=0}^5 | \nabla ^{(\alpha)} Y_k| \in B_{V,0}, \;
&\sum_{\alpha=0}^4 |\nabla ^{(\alpha)} Y_0| \in B_{V,0},\; \\
&\sum_{k=1}^m| \nabla Y_k|^2 \le c\tilde V,\; &\sup_{|u|=1}\<\nabla_u Z, u\>\le c\tilde V.
\end{split}}
\end{equation}
Then there is a global smooth solution flow $\Phi_t(y)$, Theorem \ref{limit-thm}.  Furthermore for $f\in B_{V,4}$,
 $\L f\in B_{V,2}$, $\L^2 f\in B_{V,0}$, and $ P_tf \in B_{V,4}$. 

For $M=\R^n$, an example of the control pair is:  $V(x)=C(1+|x|^2)$ and $\tilde V(x)=\ln(1+|x|^2)$. Our conditions are weaker than those commonly used
in the probability literature for $d(P_tf)$, in two ways. Firstly we allow non-bounded first order derivative, secondly we allow one sided conditions on the drift and its first order derivatives. In this regard, we extend a theorem of W. Kohler, G. C. Papanicolaou \cite{Kohler-Papanicolaou74} where they used estimations from O. A. Oleinik- E. V. Radkevi{\v{c}} \cite{Oleinik-Radkevic-book}.
 The estimates on the derivative flows, obtained in this section, are often assumptions in applications of Malliavin calculus to the study of stochastic differential equations. Results in this section might be of independent interests.

Let $P_t$ be the Markov semigroup generated by $\bar \L$. In  Section \ref{section-rate}, we prove the following estimate:
$|\E f(\Phi_t^\epsilon(y_0))- P_tf(y_0)|\le  C(t)\gamma(y_0)\epsilon \sqrt{|\log \epsilon|}$
where $C(t)$ is a constant, $\gamma$ is a function in $B_{V,0}$ and $\Phi_t^\epsilon(y_0)$ the solution to (\ref{1}) with initial value $y_0$. 
%The weak convergence of $\Phi_\cdot^\epsilon(y_0)$ follows. 
The conditions on the vector fields $Y_k$ are similar to (\ref{control-conditions}), we also assume the conditions of Theorem \ref{uniform-estimates} and that $\L_0$ satisfies strong H\"ormander's condition. We incorporated traditional techniques on time averaging with techniques from homogenization. The homogenization techniques was developed from  \cite{Li-averaging} which was inspired by the study  in M. Hairer and G. Pavliotis \cite{Hairer-Pavliotis}.  
For the rate of convergence we were particularly influenced by the following papers:
 W. Kohler and G. C. Papanicolaou \cite{Kohler-Papanicolaou74, Kohler-Papanicolaou75}
 and G. C. Papanicolaou and S.R.S. Varadhan \cite{Papanicolaou-Varadhan73}.
 Denote $\hat P_{y^\epsilon_{t\over \epsilon}}$ the probability distributions of the random variables $y^\epsilon_{t\over \epsilon}$ and $\bar \mu_t$ the probability measure determined by $P_t$. The under suitable conditions,  $W_1(\hat P_{y^\epsilon_{t\over \epsilon}}, \bar \mu_t) \le C\epsilon^r$,
 where $r$ is  any positive number less or equal to ${1\over 4}$ and $W_1$ denotes the Wasserstein 1-distance,
 see \S \ref{Wasserstein}.

\subsection{Main Theorems}
\label{discussion}
We contrive to impose as little as possible on the vector fields $\{Y_k\}$, hence a few sets of assumptions are used. For the examples we have in mind, $G$ is
a compact Lie group acting on a manifold $M$, and so for simplicity $G$ is assumed to be compact throughout the article, with few exceptions. In a future study, it would be nice to provide some interesting examples in which $G$ is not compact.

If $M$ is also compact, only the following two conditions are needed:
(a)   $\L_0$ satisfies strong H\"ormander's condition; 
  (b) $\{\alpha_k\} \subset C^r\cap N^\perp$ where $N$ is the annihilator of  the kernel of $\L_0^*$ and $r$ is a sufficiently large number.
If $\L_0$ is elliptic,  `$C^r$' can be replaced by `bounded measurable'. For the convergence condition (a) can be replaced by 
`$\L_0$ satisfies H\"ormander's condition and has Fredholm index $0$'. If $\L_0$ has a unique invariant probability measure, no condition is needed on the Fredhom index of $\L_0$.
  
 {\bf  Theorem \ref{thm-weak} and Corollary \ref{convergence-wasserstein-p}. } Under the conditions of Proposition \ref{tightness}
and Assumption \ref{assumption-Hormander}, $(y_{t\over\epsilon}^\epsilon)$ converges weakly to  the Markov process determined by
$$\bar \L =-\sum_{i,j=1}^m \overline{ \alpha_i \L_0^{-1} \alpha_j }
L_{Y_i}L_{Y_j}, \quad  \overline{ \alpha_i \L_0^{-1} \alpha_j }=\sum_{b=1}^{n_0} u_b \< \alpha_i \L_0^{-1} \alpha_j ,\pi_b\>,$$
where $n_0$ is the dimension of the kernel of $\L_0$ which, by the assumption
that $\L_0$ has Fredholm index $0$,  equals the dimension of the kernel of $\L_0^*$. The set of functions $\{u_b\}$ is a basis of $\ker (\L_0)$ and $\{\pi_b\}\subset \ker (\L_0^*)$ its dual basis. In case $\L_0$ satisfies strong H\"ormander's condition, then there is a unique invariant measure and $\overline{ \alpha_i \L_0^{-1} \alpha_j }$ is simply the average of $\alpha_i \L_0^{-1} \alpha_j$ with respect
to the unique invariant measure.
Let $p\ge 1$ be a number and $V$ a Lyapunov type function
such that $\rho^p\in B_{V, 0}$, a function space controlled by $V$. If furthermore Assumption \ref{assumption2-Y} holds,  $(y_{\cdot\over \epsilon}^\epsilon)$ converges, on $[0,t]$,  in the Wasserstein $p$-distance. 
  \\[1em]

  {\bf Theorem \ref{rate}.}
 Denote $\Phi_t^\epsilon(\cdot)$ the solution flow to (\ref{1})  and $P_t$ the semigroup for $\bar \L$. If Assumption \ref{assumption-on-rate-result} holds then for $f \in B_{V,4}$,
$$\left|\E f\left(\Phi^\epsilon_{T\over \epsilon}(y_0)\right) -P_Tf(y_0)\right|\le \epsilon |\log \epsilon|^{1\over 2}C(T)\gamma_1(y_0),$$
where  $\gamma_1\in B_{V,0}$ and $C(T)$ is a constant increasing in $T$. Similarly,  if $f\in BC^4$,
\begin{equation}
\label{rate-summary}
\left|\E f\left(\Phi^\epsilon_{T\over \epsilon}(y_0)\right) -P_Tf(y_0)\right|
 \le \epsilon  |\log \epsilon|^{1\over 2}\,C(T)\gamma_2(y_0) \left(1+ |f|_{4, \infty}\right).
\end{equation}
where $\gamma_2$ is a function in $B_{V,0}$ independent of $f$ and $C$ are 
increasing functions. 
\\[1em] 

A complete connected Riemannian manifold is said to have bounded geometry if it has strictly positive injectivity radius,
and if the Riemannian curvature tensor and its covariant derivatives are bounded.

{\bf Proposition \ref{proposition-rate}.} 
 Suppose that $M$ has bounded geometry, $\rho_o^2 \in B_{V,0}$, and Assumption \ref{assumption-on-rate-result} holds.
Let $\bar \mu$ be the limit measure and  $\bar \mu_t=(ev_t)_*\bar \mu$. 
 Then for every $r<{1\over 4}$ 
there exists $C(T)\in B_{V,0}$ and $\epsilon_0>0$ s.t. for all $\epsilon\le\epsilon_0$ and $t\le T$,
$$d_W({\mathrm {Law}} ({y^\epsilon_{t\over \epsilon}}), \bar \mu_t)\le C(T)\epsilon^{r}.$$

 Besides the fact that we work on manifolds, where there is the inherited non-linearity and the problem with cut locus, the following aspects of the paper are perhaps new.
  (a) We do not assume there exists a unique invariant probability measure on the noise and  the effective processes are obtained by a suitable projection, accommodating one type of degeneracy. Furthermore the noise 
  takes value in another manifold, accommodating `removable' degeneracy. For example the stochastic processes in question lives in a Lie group, while the noise are entirely in the directions of a sub-group. (b) 
We used Lyapunov functions to control the growth of the vector fields and their derivatives,  leading to estimates uniform in $\epsilon$ and to the conclusion on the convergence in the Wasserstein topologies. 
A key step for the convergence is  a law of large numbers, with rates, for sub-elliptic operators (i.e. operators satisfying H\"ormander's sub-elliptic estimates). (c)  Instead of working with iterated time averages we use a solution to Poisson equations to reveal the effective operator. Functionals of the processes $y_{t\over \epsilon}^\epsilon$ splits naturally into the sum of a fast martingale, a finite variation term involving a second order differential operator in H\"ormander form, and a term of order $\epsilon$. From this we obtain the effective diffusion, in explicit H\"ormander form. It is perhaps also new to have an estimate for the rate of the convergence in the Wasserstein distance. Finally we improved known theorems on the existence of global smooth solutions for SDEs in \cite{Li-flow}, c.f.  Theorem \ref{limit-thm} below where a criterion is given in terms of a pair of Lyapunov functions. New estimates on the moments of higher order covariant derivatives of the derivative flows are also given.

\subsection{Classical Theorems} 
We review, briefly, basic ideas from existing literature on random ordinary differential equations with fast oscillating vector fields. Let $F(x,t,\omega, \epsilon):=F^{(0)}(x,t,\omega)+\epsilon F^{(1)}(x,t,\omega)$, where $F^{(i)}(x,t,\cdot)$ are measurable functions, for which
a Birkhoff ergodic theorem holds whose limit is denoted by $\bar F$. The solutions to the equations $\dot y_t^\epsilon = F(y_t^\epsilon,{t\over \epsilon},\omega, \epsilon)$
 over a time interval $[0, t]$, can be approximated by the solution to the averaged equation driven by 
 $\bar F$. 
  If $\bar F^{(0)}=0$,  we should observe the solutions in the next time scale and study $\dot x_t^\epsilon ={1\over \epsilon} F(x_t^\epsilon,{t\over \epsilon^2},\omega, \epsilon)$. See R. L. Stratonovich  \cite{Stratonovich-rhs, Stratonovich61}.   
  Suppose for some functions $\bar a_{j,k}$ and $\bar b_j$ the following 
estimates hold uniformly:
 \begin{equation}\label{condition-rate}
 {\begin{split}
 &\left| {1\over \epsilon^3} \int_s^{s+\epsilon}
 \int_{s}^{r_1}
 \E F_j^{(0)}(x,  {r_2\over \epsilon^2}) F^{(0)}_k(x,  {r_1\over \epsilon^2})\, dr_2\; dr_1-\bar a_{j,k}(s,x)\right|\,
\le o(\epsilon),\\
%%% modification after final version remove  dr_2\; dr_1
 &\left| {1\over \epsilon^3} \int_s^{s+\epsilon}
 \int_{s}^{r_1}
\sum_{k=1}^d \E {\partial F_j^{(0)}\over \partial x_k}( x,  {r_2\over \epsilon^2}) F_k^{(0)}( x,  {r_1\over \epsilon^2})\, dr_2\; dr_1
 +{1\over \epsilon} \int_s^{s+\epsilon} \E F_j^{(1)} ( x,  {r\over \epsilon^2}) \,dr-\bar b_j(x, s)\right|\\
& \le o(\epsilon).
 \end{split}}
 \end{equation}
Then under  a `strong mixing' condition with suitable mixing rate,  the solutions  of the equations $\dot x_t^\epsilon ={1\over \epsilon} F(x_t^\epsilon,{t\over \epsilon^2},\omega, \epsilon)$
converge weakly on any compact interval to a Markov process. 
This is a theorem of R. L. Stratonovich \cite{Stratonovich61} and R. Z. Khasminskii\cite{Khasminskii-rhs},
further refined and explored  in Khasminskii \cite{Khasminskii66} and   A. N. Borodin \cite{Borodin77}.   
These theorems lay foundation for investigation beyond ordinary differential equations with a fast oscillating right hand side.

In our case, noise comes into the system via a $\L_0$-diffusion satisfying H\"ormander's conditions, and 
hence we could by pass these assumptions and also obtain convergences in the Wasserstein distances.
For manifold valued stochastic processes, 
some difficulties are caused by the inherited non-linearity. For example, integrating a vector field along a path makes sense only after they are parallel translated back. 
Parallel transports of a vector field along a path, from time $t$ to time $0$, involves the whole path up to time $t$ and introduces extra difficulties;  
this is still an unexplored territory wanting further investigations. For the proof of tightness, the non-linearity
causes particular difficulty if the Riemannian distance function is not smooth. The advantage of working on a manifold setting is that for some specific physical models, the noise can be untwisted and becomes easy to deal with. 

 Our estimates for the rate of convergence, section \ref{section-rate} and \ref{Wasserstein}, can be considered as an extension to that in W. Kohler and G. C. Papanicolaou \cite{Kohler-Papanicolaou74, Kohler-Papanicolaou75}, 
 which were in turn developed from the following sequence of remarkable papers:  
 R. Coghurn and R. Hersh \cite{Cogburn-Hersh},  J.B. Keller and 
G. C. Papanicolaou \cite{Papanicolaou-Keller}, R. Hersh and M. Pinsky  \cite{Hersh-Pinsky72}, R. Hersh and G. C. Papanicolaou \cite{Hersh-Papanicolaou72} and  G. C. Papanicolaou and S.R.S. Varadhan \cite{Papanicolaou-Varadhan73}.
 See also T. Kurtz \cite{Kurtz70} and \cite{Papanicolaou-Stroock-Varadhan77} by
  D. Stroock and S. R. S. Varadhan.
  
The condition $\bar F=0$ needs not hold for this type of scaling and convergence.
 If $F(x, t, \omega, \epsilon)=F^{(0)} (x, \zeta_t(\omega))$, where $\zeta_t$ is a stationary process with values in $\R^m$, and $\bar F^{(0)}=X_H$, the Hamiltonian vector field associated to a function $H\in BC^3( \R^2; \R)$ whose level sets are closed connected curves without intersections, then $H(y_{t\over\epsilon}^\epsilon)$ converge to a Markov process, under suitable mixing and technical assumptions.
See A. N. Borodin and M. Freidlin \cite{Borodin-Freidlin}, also M. Freidlin and M. Weber \cite{Freidlin-Weber} where a first integral replaces the Hamiltonian, and also X.-M. Li \cite{Li-geodesic}  where the value of a map from a manifold to another is preserved by the unperturbed system.
%studied Ginzburg-Landau system of SDEs

In M. Freidlin and  A. D.Wentzell  \cite{Freidlin-Wentzell}, the following type of central limit theorem is proved:
${1\over \sqrt \epsilon} \left(H(x_s^\epsilon)-H(\bar x_s)\right)$
 converges to a Markov diffusion.  This formulation is not suitable when the conserved quantity  takes value in a non-linear space.

For the interested reader, we  also refer to the following articles on limit theorems, averaging and Homogenization for stochastic equations on manifolds:  
N. Enriquez, J. Franchi, Y. LeJan \cite{Enriquez-Franchi-LeJan}, I.  Gargate, P. Ruffino \cite{Gargate-Ruffino}, 
 N. Ikeda, Y. Ochi \cite{Ikeda-Ochi}, Y. Kifer \cite{Kifer92},
M. Liao and L. Wang \cite{Liao-Wang-05}, S. Manade, Y. Ochi \cite{Manabe-Ochi},
Y. Ogura \cite{Ogura01}, M. Pinsky \cite{Pinsy-parallel}, and R. Sowers \cite{Sowers01}.

\subsection{Further Question.}
(1)  I am grateful to the associate editor for pointing out the paper by C. Liverani and S. Olla \cite{Liverani-Olla}, where random perturbed Harmiltonian system, in the context of weak interacting particle systems, is studied. 
Their system is somewhat related to the completely integrable equation studied in \cite{Li-averaging} leading to a new problem which we now state.  Denote $X_f$ the Hamiltonian vector field on a symplectic manifold corresponding to a function $f$. 
If the symplectic manifold is $\R^{2n}$ with the canonical symplectic form, $X_f$ is the skew gradient of $f$.
Suppose that $\{H_1, \dots, H_n\}$ is a completely integrable system,
i.e. they are poisson commuting at every point and their Hamiltonian vector fields are linearly independent at almost all points.
Following \cite{Li-averaging} we consider a completely integrable SDE perturbed by a transversal Hamiltonian vector field:
$$dy_t^\epsilon = \sum_{i=1}^n X_{H_i}(y_t^\epsilon)\circ dW_t^i+X_{H_0}(y_t^\epsilon)dt+\epsilon X_K(y_t^\epsilon)dt.$$ 
 Suppose that $X_{H_0}$ commutes with $X_{H_k}$ for $k=1,\dots, n$, then each $H_i$ is a first integral of the unperturbed system.
Then, \cite[Th 4.1]{Li-averaging}, within the action angle coordinates of a regular value of the energy function $H=(H_1, \dots, H_n)$,
the energies $\{ H_1(y^\epsilon_{t\over \epsilon^2}), \dots, H_n(y^\epsilon_{t\over \epsilon^2})\}$ converge weakly to a Markov process. When restricted to the level sets of the energies, the fast motions are ellipitic.
It would be desirable to remove the `complete integrability' in favour of Hormander's type conditions.
There is a non-standard symplectic form  on $(\R^4)^N$ with respect to which the vector fields in \cite{Liverani-Olla} are Hamiltonian vector fields and when restricted to level sets of the energies the unperturbed system in \cite{Liverani-Olla}  
satisfies H\"ormander's condition, see \cite[section 5]{Liverani-Olla}, and therefore provides a motivating example for further studies. Finally note that the driving vector fields in (\ref{1}) are in a special form,  results here would not apply to the systems in \cite{Li-averaging} nor that in \cite{Liverani-Olla}, and hence it would be interesting to formulate
 and develop limit theorems for more general random ODEs to include these two cases. 
 
 (2) It should be interesting to develop a theory for the ODEs below \begin{equation}
\label{1}
\dot y_t^\epsilon(\omega)=\sum_{k=1}^m Y_k\left(y_t^\epsilon(\omega)\right)\alpha_k(z_t^\epsilon(\omega), y_t^\epsilon))
\end{equation}
where $\alpha_k$ depends also on the $y^\epsilon$ process.

(3) It would be nice to extend the theory to allow the noise to live in a  non-compact manifold, in which $\L_0$ should be an Ornstein-Uhlenbeck type operator whose drift term would provide for a deformed volume measure.

\medskip

{\it Notation.} 
Throughout this paper $\B_b(M;N)$, $C_K^r(M;N)$, and $BC^r(M;N)$ denote
the set of functions from $M$ to $N$ that are respectively bounded measurable, $C^r$ with compact supports, and bounded $C^r$ with bounded first $r$ derivatives. If $N=\R$ the letter $N$ will be suppressed. Also ${\mathbb L}(V_1;V_2)$ denotes the space of bounded linear maps; $C^r(\Gamma TM)$ denotes $C^r$ vector fields on a manifold $M$.

\section{Examples}
\label{section-Examples} Let  $\{W_t^k\}$ be independent real valued Brownian motions on a given filtered probability space, $\circ$ denote Stratonovich integrals. In the following $H_0$  and  $ A_k$ are smooth vector fields, and $\{A_1, \dots, A_k\}$ is an orthonormal basis at each point of the vertical tangent spaces. To be brief, we do not specify the properties of the vector fields,  instead refer the interested reader to 
\cite{Li-geodesic} for details. For any $\epsilon>0$, the stochastic differential equations
$$du_t^\epsilon=H_0(u_t^\epsilon)dt+{1\over \sqrt \epsilon}\sum_{k=1}^{n(n-1)\over 2} A_k(u_t^\epsilon)\circ dW_t^k$$
 are degenerate and they interpolate between the geodesic equation ($\epsilon=\infty$) and Brownian motions on the fibres ($\epsilon=0$). 
The fast random motion is transmitted to the horizontal direction by the action of the Lie bracket $[H_0, A_k]$.
If $ H_0=0$, there is a conserved quantity to the system which is the projection from the orthonormal bundle to the base manifold.  This allows us to separate the slow variable $(y_t^\epsilon)$ and the fast variable $(z_t^\epsilon)$. The reduced equation
for $(y_t^\epsilon)$, once suitable `coordinate maps' are chosen, can be written in the form of  (\ref{1}). 
In \cite{Li-geodesic} we proved that $(y^\epsilon_{t\over \epsilon})$ converges weakly to a rescaled horizontal Brownian motion. Recently J. Angst, I. Bailleul and C. Tardif  gave this a beautiful treatment, \cite{Angst-Bailleul-Tardif}, using rough path analysis. 
By theorems in this article, the above model can be generalised to include random perturbation by hypoelliptic diffusions, i.e. $\{A_1, \dots, A_k\}$ generates all vertical directions. In \cite{Li-geodesic} we did not know how to obtain a rate for the convergence.    Theorem \ref{rate}, in this article, will apply  and indeed we have an upper bound for the rate of convergence.

As a second example, we consider, on the special orthogonal group $SO(n)$,  the following equations:
\begin{equation}
\label{interpolation}
dg_t^\epsilon={1\over  \sqrt \epsilon}\sum_{k=1}^{n(n-1)\over 2} g_t^\epsilon E_{k}\circ dW_t^k + g_t^\epsilon Y_0dt,
\end{equation}
where $\{E_k\}$ is an orthonormal basis of $\so(n-1)$, as a subspace of $\so(n)$, and $Y_0$ is a skew symmetric matrix orthogonal to $\so(n-1)$. The above equation is closely related to the following set of equations:
$$dg_t=\gamma\sum_{k=1}^{n(n-1)\over 2} g_t E_{k}\circ dW_t^k +\delta  g_t Y_0dt,$$
where $\gamma, \delta$ are two positive numbers. If $\delta=0$ and $\gamma=1$, the solutions are
Brownian motions on $SO(n-1)$. If $\delta={1\over |Y_0|}$ and $\gamma=0$, the solutions are
unit speed geodesics on $SO(n)$. These equations interpolate between a Brownian motion on the sub-group $SO(n-1)$ and a one parameter family of subgroup on $SO(n)$. See  \cite{Li-homogeneous}.
Take $\delta=1$ and let $\gamma={1\over \sqrt \epsilon}\to \infty$,  what could be the `effective limit' if it exists?
The slow components of the solutions, which we denote by $(u_t^\epsilon)$,
satisfy equations of the form (\ref{1}). They are `horizontal lifts' of the projections of the solutions to $S^n$.
 If $\m$ is the orthogonal complement of $\so(n-1)$ in  $\so(n)$ then $\m$ is $\Ad_H$-irreducible and $\Ad_H$-invariant, noise is transmitted from $\h$ to every direction in $\m$, and this in the uniform way. It is therefore plausible that $u_{t\over \epsilon}^\epsilon$ can be approximated  by a diffusion $\bar u_t$ of constant rank. The projection of $u_t$ to $S^n$ is a scaled Brownian motion with scale $\lambda$.   The scale $\lambda$ is a function of the dimension $n$, but is independent of $Y_0$ and is associated to an eigenvalue of the Laplacian on $SO(n-1)$, indicating the speed of propagation.
%We study the limit as $n\to \infty$ and $\epsilon \to 0$ in a forthcoming article.

As a third example we consider the Hopf fibration $\pi: S^3\to S^2$. Let $\{X_1, X_2, X_3\}$ be the Pauli matrices, they form an orthonormal basis of $\su(2)$ with respect to the canonical bi-invariant Riemannian metric. $$X_1=\left(\begin{matrix} i &0\\0&-i
\end{matrix}\right), \quad X_2=\left(\begin{matrix} 0 &1\\-1&0
\end{matrix}\right),
 \quad X_3=\left(\begin{matrix} 0 &i\\i&0
\end{matrix}\right).
$$
Denote $X^*$ the left invariant vector field generated by $X\in \su(2)$. 
By declaring $\{{1\over \sqrt \epsilon}X_1^*, X_2^*, X_3^*\}$ an orthonormal frame, we obtain a family of
left invariant Riemannian metrics $m^\epsilon$ on $S^3$. The Berger's spheres,  $(S^3, m^\epsilon)$,  converge in measured Gromov-Hausdorff
topology to the lower dimensional sphere $S^2({1\over 2})$. For further discussions see K. Fukaya \cite{Fukaya} and
J. Cheeger and M. Gromov \cite{Cheeger-Gromov}.
Let $W_t$ be a one dimensional Brownian motion and take $Y$ from $\m:=\<X_2, X_3\>$.
 The infinitesimal generator of the equation $
dg_t^\epsilon={1\over \sqrt \epsilon} X_1^*(g_t^\epsilon)\circ dW_t+Y^*(g_t^\epsilon) \;dt$
satisfies weak H\"ormander's conditions. The `slow motions', suitably sacled, converge
to a `horizontal' Brownian motion whose generator is ${1\over 2}c\trace_{\m}\nabla d$, where
the trace is taken in $\m$.  A slightly different, ad hoc, example on the Hopf fibration is discussed in  \cite{LI-OM-1}. 
An analogous equations can be considered on $SU(n)$ where the diffusion coefficients come from a maximal torus.

Finally we give an example where the noise $(z_t^\epsilon)$ in the reduced equation is not elliptic. 
Let $M=SO(4)$ and let $E_{i,j}$ be the elementary $4\times 4$ matrices and $A_{i,j}={1\over \sqrt 2} (E_{ij}-E_{ji})$.
For $k=1,2$ and $3$, we consider the equations
$$dg_t^\epsilon={1\over \sqrt\epsilon}  A_{1,2}^*(g_t^\epsilon)\circ db_t^1+
{1\over \sqrt\epsilon}  A_{1,3}^*(g_t^\epsilon)\circ db_t^2+A_{k4}^*(g_t^\epsilon)dt.$$
The slow components of the solutions of these equations again satisfy an equation of the form (\ref{1}).

 \section{Preliminary Estimates}
\label{section-estimates}

Let  $\L_0$ be a diffusion operator on a manifold $G$ and $Q_t$ its transition semigroup and transition probabilities.  
Let $\|\cdot \|_{TV}$ denote the total variation norm of a measure, normalized so that the total variation norm between two probability measures is less or equal to $2$. By the duality formulation for the total variation norm,
$$\|\mu\|_{TV}=\sup_{|f| \le 1, f\in \B_b(G;\R)} \left|\int_G fd\mu \right|.$$
For $W\in \B( G; [1,\infty))$ denote $\|f\|_W$ the weighted supremum norm
and $\|\mu\|_{TV,W}$ the weighted total variation norm:
$$\|f\|_W=\sup_{x\in G}{|f(x)|\over W(x)}, \quad 
\|\mu\|_{TV, W}= \sup_{ \{ \|f\|_W\le 1\}} \left| \int_G f d\mu\right| .
$$
\begin{assumption}
\label{assumption1}
There is an invariant probability measure $\pi$ for $\L_0$, a real valued function $W\in L^1(G, \pi)$ with $W\ge 1$,  numbers $\delta>0$ and $a>0$ such that 
$$\sup_{x\in G}{\|Q_t(x,\cdot)-\pi\|_{TV, W}\over W(x)} \le ae^{-\delta t}.$$
\end{assumption}
 If $G$ is compact we take $W\equiv 1$.
 
  In the following lemma we collect some elementary estimates, which will be used to prove Lemma \ref{lemma3}
  and \ref{lemma3-2}, for completeness their proofs are given in the appendix. Write $\bar W=\int_G W d\pi$.
\begin{lemma}
\label{lemma2}
Assume Assumption \ref{assumption1}. 
Let $f,g:G\to \R$ be bounded measurable functions and let $c_\infty=|f|_\infty\|g\|_W$.
Then the following statements hold  for all $s, t\ge 0$.
\begin{enumerate}
\item [(1)] Let $(z_t)$ be an $\L_0$ diffusion. If $ \int_G gd\pi=0$,
$${\begin{split} &\left|{1\over t-s} \int_s^{t} \int_s^{s_1} \left( \E\left\{ f(z_{s_2}) g(z_{s_1}) \Big| \F_s\right\}
-\int_G fQ_{s_1-s_2}g d\pi\right) ds_2ds_1\right|\\
&\le {a^2c_\infty\over (t-s)\delta^2}W(z_s).\end{split}}$$

\item [(2)]  Let $(z_t)$ be an $\L_0$ diffusion. If  $\int_G gd \pi=0$ then 
$${ \begin{split}
&\left|{1\over t-s} \int_s^{t} \int_s^{s_1}  \E\left\{ f(z_{s_2}) g(z_{s_1}) \Big| \F_s\right\} \d s_2 \d s_1
-  \int_G\int_0^\infty f Q_rg \d r \d\pi
 \right|\\
&\le {c_\infty  \over (t-s) \delta^2} (a^2 W(z_s)+a \bar W )+{c_\infty a\over \delta} \bar W.
\end{split} } $$

\item [(3)]
Suppose that either $\int_G f\d\pi =0$ or $\int_G g\d\pi =0$. Let
$$C_1={a\over \delta^2}
(aW+\bar W)|f|_\infty\|g\|_W, \quad 
C_2={2a\over \delta}|f|_\infty\|g\|_W\bar W
+{a \over \delta}|\bar g| \; \|f\|_W W.$$
 Let $(z_t^\epsilon)$ be an $\L_0^\epsilon$ diffusion.  Then  for every $\epsilon>0$,
 $$\left| \int_{s\over \epsilon}^{t\over \epsilon} \int_{s\over \epsilon}^{s_1} 
  \E\left\{ f(z^\epsilon_{s_2}) g(z^\epsilon_{s_1}) \Big| \F_{s\over \epsilon}\right\}\d s_2\d s_1 \right|
 \le C_1(z_{s\over \epsilon}^\epsilon) \epsilon^2+C_2(z_{s\over \epsilon}^\epsilon) (t-s).$$
\end{enumerate}
\end{lemma}

To put Assumption \ref{assumption1} into context, we consider H\"ormander type operators.
Let $L_{X}$ denote Lie differentiation in the direction of a vector field 
 $X$ and $[X, Y]$  the Lie bracket of two vector fields $X$ and $Y$.
Let $\{X_i, i=0, 1,\dots, m'\}$ be a family of smooth vector fields on a compact smooth manifold 
$G$ and  $\L_0={1\over 2}\sum_{i=1}^{m'}L_{X_i}L_{X_i}+L_{X_0}$.
If $\{X_i, i=1, \dots, m'\}$ and their Lie brackets generate the tangent space $T_xG$ at each point $x$ we say that the operator $\L_0$ satisfies the strong H\"ormander's condition. 
\begin{lemma}
\label{lemma1}
Suppose that $\L_0$ satisfies the strong H\"ormander condition on a compact manifold $G$ and
let $Q_t(x,\cdot)$ be its family of  transition probabilities. Then Assumption \ref{assumption1} holds with $W$ identically $1$. Furthermore the invariant probability measure $\pi$ 
has a strictly positive smooth density w.r.t. the Lebesgue measure and 
$$\|Q_t(x,\cdot)-\pi(\cdot)\|_{TV} \le Ce^{-\delta t}$$
for all $x$ in $G$ and for all $t>0$.
\end{lemma}
\begin{proof}
By H\"ormander's theorem there are smooth functions $q_t(x,y)$ such that
$Q_t(x,dy)=q_t(x,y)dy$. Furthermore  $q_t(x,y)$ is strictly positive, see J.-M. Bony \cite{Bony69} and A. Sanchez-Calle \cite{Sanchez-Calle78}. Let
$a=\inf_{x,y\in M}q_t(x,y)>0$. Thus D\"oeblin's condition holds: if
$\vol(A)$ denotes the volume of a Borel set $A$,
$Q_t(x, A)\ge a \,\vol(A)$.
\end{proof}

We say that $W$ is a $C^3$ Lyapunov function for the ergodicity problem if there are constant $c\not =0$ and $C>0$ s.t. 
$\L_0 W\le C-c^2 W$.   If such a function exists, the $\L_0^\epsilon$ diffusions are conservative.
Suppose that the Lyapunov function $V$ satisfies in addition the following conditions:
there exists a number $\alpha\in (0,1)$ and $t_0>0$ s.t. for every $R>0$, 
$$\sup_{\{(x,y): V(x)+V(y)\le R\}} \|Q_{t_0}(x, \cdot)-Q_{t_0}(y, \cdot)\|_{TV} \le 2(1-\alpha),$$
Then there exists a unique invariant measure $\pi$ such that Assumption \ref{assumption1} holds,
see e.g.  M. Hairer and J. Mattingly \cite{Hairer-Mattingly08}. 
We mention the following standard estimates which helps to understand the estimates in Lemma \ref{lyapunov-ergodicity}.
\begin{lemma}
\label{lyapunov-ergodicity}
Let $W$ be a $C^3$ Lyapunov function for the ergodicity problem of $\L_0$,
 $ \E W(z_0^\epsilon)$ 
is uniformly bounded in $\epsilon$ for $\epsilon$ sufficiently small.  Then there exist  numbers $\epsilon_0>0$ and  $c$ s.t.
for all $t>0$,
$$\sup_{s\le t}\sup_{\epsilon \le \epsilon_0}\E W(z_{s\over \epsilon}^\epsilon) \le c.$$
\end{lemma}
\begin{proof}
By localizing $(z_t^\epsilon)$ if necessary, we see that 
$W(z_t^\epsilon)-W(z_0^\epsilon)-{1\over \epsilon} \int_0^t \L_0 W (z^\epsilon_r) dr$
is a martingale. Let $c\not =0$ and $C>0$ be constant s.t. 
$\L_0 W\le C-c^2 W$.
Then
$ \E W(z_{s\over \epsilon}^\epsilon) \le \left (\E W(z_0^\epsilon)+{1\over \epsilon}Ct\right) e^{-{c^2\over \epsilon} t}$.
\end{proof}

As an application we see that, under the assumption of Lemma \ref{lyapunov-ergodicity}, the functions $C_i$ in part (3) of Lemma \ref{lemma2}
satisfy that $\sup_{\epsilon \le \epsilon_0}\E C_i(z_{s\over \epsilon}^\epsilon) <\infty$.

  \begin{definition}\label{definition-complete}
We say that a stochastic differential equation (SDE) on $M$ is {\it complete} or {\it conservative} if for each initial point $y\in M$ any solution with initial value $y$ exists for all  $t\ge 0$. Let $\Phi_t(x)$ be its solution starting from $x$. 
 The SDE is {\it strongly complete} if it has a unique strong solution and that $(t,x)\mapsto \Phi_t(x, \omega)$ is continuous for a.s. $\omega$.
\end{definition}

From now on, by a solution we always mean a globally defined solution.  
   For $\epsilon\in (0,1)$ we define $\L_0^\epsilon={1\over \epsilon } \L_0$. Let  $Q_t^\epsilon$ denote their transition semigroups and transition probabilities.  
 For each $\epsilon>0$, let $(z_t^\epsilon)$ be an  $\L_0^\epsilon$ diffusion. 
Let  $\alpha_k\in \B_b(G;\R)$ and $(y_t^\epsilon)$ solutions to the equations
 \begin{equation}
\label{sde-3}
\dot y_t^\epsilon=\sum_{k=1}^m Y_k(y_t^\epsilon)\alpha_k(z_t^\epsilon).
\end{equation}
Let $\Phi_{s,t}^\epsilon$ be the solution flow to (\ref{sde-3}) with  $\Phi_{s,s}^\epsilon(y)=y$.
We denote by $\bar g$ the average of an integrable function $g: G\to \R$ with respect to $\pi$. Let 
\begin{equation}
\label{cpsi-0}
c_0(a, \delta)={a^2+a\over \delta^2}+{3 a\over\delta}, \quad c_{W}=c(a,\delta)(W+\bar W).
\end{equation}
  
\begin{lemma}
\label{lemma3}
Suppose that  Assumption \ref{assumption1} holds.
 Let $f,g \in \B_b(G; \R)$ and $\bar g=0$.  Suppose that $\alpha_k$ are bounded. Then for any
 $F\in C^1(M; \R)$,  $0\le s \le t$ and $0<\epsilon<1$,
$$\left|\epsilon \int_{s\over \epsilon}^{t\over \epsilon} \int_{s\over \epsilon}^{s_1} 
  \E\left\{ F(y^\epsilon_{s_2}) g(z^\epsilon_{s_2})f(z^\epsilon_{s_1})
 \big| \F_{s\over \epsilon}\right\}  \d s_2 \d s_1\right|\\
\le 2 \gamma_\epsilon |g|_\infty |f|_\infty ( \epsilon^2+ (t-s)^2 ).$$
Here  $$\gamma_\epsilon
= \left( |F(y_{s\over\epsilon}^\epsilon)| \;c_W  ( z_{s\over \epsilon}^\epsilon)
+ \sum_{l=1}^m |\alpha_l|_\infty  
{\epsilon\over t-s}  \int_{s\over \epsilon}^{t\over \epsilon} 
  \E  \left\{  \left|(L_{Y_l} F)(y^\epsilon_r)\right|c_W (z^\epsilon_r)
  \; \big|\;\F_{s\over \epsilon} \right\}  dr
 \right).$$

\end{lemma}

\begin{proof}
We first expand $F(y_{s_2}^\epsilon)$ at ${s \over \epsilon}$: 
$${ \begin{split}
&\epsilon \int_{s\over \epsilon}^{t\over \epsilon} \int_{s\over \epsilon}^{s_1} 
  \E\left\{ F(y^\epsilon_{s_2}) g(z^\epsilon_{s_2})f(z^\epsilon_{s_1})
 \big| \F_{s\over \epsilon}\right\}  \d s_2 \d s_1
 =\epsilon F(y^\epsilon_{s\over \epsilon}) \int_{s\over \epsilon}^{t\over \epsilon} \int_{s\over \epsilon}^{s_1} 
  \E\left\{  g(z^\epsilon_{s_2})f(z^\epsilon_{s_1})
 \big| \F_{s\over \epsilon}\right\}  \d s_2 \d s_1\\
 &+\sum_{l=1}^m\epsilon \int_{s\over \epsilon}^{t\over \epsilon} \int_{s\over\epsilon}^{s_1}\int_{s\over \epsilon}^{s_2} 
  \E\left\{ (dF)(Y_l(y^\epsilon_{s_3})) \alpha_l(z_{s_3}^\epsilon) g(z^\epsilon_{s_2})f(z^\epsilon_{s_1})
 \big| \F_{s\over \epsilon}\right\} \d s_3 \d s_2 \d s_1
\end{split} } $$
By  part (3) of lemma \ref{lemma2}
$$ \left|\epsilon  F(y^\epsilon_{s\over \epsilon}) \int_{s\over \epsilon}^{t\over \epsilon} \int_{s\over \epsilon}^{s_1} 
  \E\left\{  g(z^\epsilon_{s_2})f(z^\epsilon_{s_1})
 \big| \F_{s\over \epsilon}\right\}  \d s_2 \d s_1\right|\le   |F(y_{s\over \epsilon}^\epsilon)||f|_\infty|g|_\infty
 c_W (z^\epsilon_{s\over \epsilon}) \left(\epsilon^3+(t-s)\epsilon\right).$$
 It remain to estimate the summands in the second term, whose absolute value is bounded by the following
$${ \begin{split}
&A_l:= \left|\epsilon\int_{s\over \epsilon}^{t\over \epsilon} \int_{s\over\epsilon}^{s_1}\int_{s\over \epsilon}^{s_2}
\E \left\{ (dF)(Y_l(y^\epsilon_{s_3})) \alpha_l(z_{s_3}^\epsilon) g(z^\epsilon_{s_2})f(z^\epsilon_{s_1})
 \big| \F_{s\over \epsilon}\right\} \d s_3 \d s_2 \d s_1\right|
\\
&=\left| \epsilon \int_{s\over \epsilon}^{t\over \epsilon} 
\E \left\{(dF)(Y_l(y^\epsilon_{s_3})) \alpha_l(z_{s_3}^\epsilon)
\int_{s_3}^{t\over\epsilon}\int_{s_2}^{t\over \epsilon}
  \E\left\{  g(z^\epsilon_{s_2})f(z^\epsilon_{s_1})
 \big| \F_{s_3}\right\}  \d s_1 \d s_2 \Big|\F_{s\over \epsilon} \right\} ds_3 \right|.
\end{split} } $$
For $s_3\in [{s\over \epsilon}, {t\over \epsilon}]$, we apply part (3) of lemma \ref{lemma2}
to bound the inner iterated integral,
$${ \begin{split}
 &\left| \int_{s_3}^{t\over\epsilon}\int_{s_2}^{t\over \epsilon}
  \E\left\{  g(z^\epsilon_{s_2})f(z^\epsilon_{s_1})
 \big| \F_{s_3}\right\}  \d s_1 \d s_2\right|
 =\left |  \int_{s_3}^{t\over\epsilon}\int_{s_3}^{s_1}
  \E\left\{  g(z^\epsilon_{s_2})f(z^\epsilon_{s_1})
 \big| \F_{s_3}\right\}  \d s_2 \d s_1\right|\\
& \le \left(   \epsilon^2+ t-\epsilon s_3 \right)  c_W (z^\epsilon_{s_3})  |f|_\infty |g|_\infty.
\end{split} } $$
We bring this back to the previous line,  the notation $L_{Y_l}F=dF(Y_l)$ will be used,
$${ \begin{split}
&A_l\le \epsilon
  \int_{s\over \epsilon}^{t\over \epsilon} 
  \E  \left\{ \left|  (dF)(Y_l(y^\epsilon_{s_3})) c_W (z^\epsilon_{s_3}) \alpha_l(z_{s_3}^\epsilon) \Big|\F_{s\over \epsilon} \right\} \right| 
\left(  \epsilon^2+ (t-\epsilon s_3) \right)  |f|_\infty |g|_\infty \d s_3\\
&\le   |f|_\infty |g|_\infty |\alpha_l|_\infty (t-s)  (\epsilon^2+(t-s))
{\epsilon\over t-s}  \int_{s\over \epsilon}^{t\over \epsilon} 
  \E  \left\{  \left|(L_{Y_l} F)(y^\epsilon_{s_3})\right|c_W (z^\epsilon_{s_3})\Big|\F_{s\over \epsilon} \right\}  ds_3.
\end{split} } $$
Putting everything together we see that, for $\gamma_\epsilon$ given in the Lemma, $\epsilon<1$,
 $${ \begin{split}
&\left|
\epsilon \int_{s\over \epsilon}^{t\over \epsilon} \int_{s\over \epsilon}^{s_1} 
  \E\left\{ F(y^\epsilon_{s_2}) g(z^\epsilon_{s_2})f(z^\epsilon_{s_1})
 \big| \F_{s\over \epsilon}\right\}  \d s_2 \d s_1\right|
 \le 2\gamma_\epsilon    |g|_\infty |f|_\infty \left( \epsilon^2+
 (t-s)^2\right).
\end{split} } $$
The proof is complete.
\end{proof}

 In Section \ref{section-uniform-estimates} we will estimate $\gamma_\epsilon$ and give uniform, in $\epsilon$,  moment estimates of functionals of $(y_t^\epsilon)$ on $[0, {T \over \epsilon}]$.  

\begin{lemma}
\label{lemma3-2}
Assume that $(z_t^\epsilon)$ satisfies Assumption \ref{assumption1} and $\alpha_j$ are bounded.
 If $F\in C^2(M;\R)$ and $f\in \B_b(G;\R)$, then  for all $s\le t$,
$${\begin{split}
&\left| {\epsilon \over t-s} \int_{s\over \epsilon }^{t\over \epsilon} \E \left\{ F (y_r^\epsilon) f(z_{r}^\epsilon) \big| \F_{s\over \epsilon}\right\} dr
-\bar f\;   F(y_{s\over \epsilon}^\epsilon) \right| \\
& \le {2a\over \delta} |f|_\infty \left( W(z_{s\over \epsilon}^\epsilon)|F|(y_{s\over \epsilon}^\epsilon)
+\sum_{j=1}^\m \gamma^j_\epsilon  |\alpha_j|_\infty  \right)\left ({\epsilon^2\over t-s} +(t-s)\right)
\end{split}}
$$
where 
$$\gamma^j_{\epsilon}(y)=c_W (z_{s\over \epsilon}^\epsilon) 
 \;|L_{Y_j}F (y^\epsilon_{s\over\epsilon})|+ \sum_{l=1}^m |\alpha_l|_\infty  
{\epsilon\over t-s}  \int_{s\over \epsilon}^{t\over \epsilon} 
  \E  \left\{  \left|L_{Y_l} L_{Y_j}F(y^\epsilon_r)\right|c_W (z^\epsilon_r)
  \; \big|\;\F_{s\over \epsilon} \right\}  dr.
  $$
\end{lemma}

\begin{proof}
We note that,
$${\begin{split}
 {\epsilon \over t-s}\int_{s\over \epsilon}^{t\over \epsilon }F (y_r^\epsilon) f(z_{r}^\epsilon)dr
=&F(y_{s\over \epsilon}^\epsilon)  {\epsilon \over t-s}\int_{s\over \epsilon }^{t\over \epsilon} f(z_{r}^\epsilon)dr\\
&+ \sum_{j=1}^m  {\epsilon \over t-s}\int_{s\over \epsilon }^{t\over \epsilon} \int_{s\over \epsilon}^{s_1}
dF (Y_j(y_{s_2}^\epsilon)) \alpha_j (z_{s_2}^\epsilon) f(z_{s_1}^\epsilon)ds_2ds_1.
\end{split}}$$

Letting $\psi(r)=ae^{-\delta r}$, it is clear that for $k\ge 2$, 
$${\begin{split}
&\left| \E \left\{   \left( F(y_{s\over \epsilon}^\epsilon)
{\epsilon \over t-s}  \int_{s\over \epsilon }^{t\over \epsilon} f(z_{r}^\epsilon)dr
 -\bar f\;F(y_{s\over \epsilon}) \right)\big |\;\F_{s\over \epsilon}\right\}\right|\\
&\le  \|f\|_W W(z^\epsilon_{s\over \epsilon})   \left|F(y_{s\over \epsilon}^\epsilon)\right| 
{\epsilon^2\over t-s} \int_{s\over  \epsilon^2}^{t\over  \epsilon^2} \psi\left ( r-{s\over \epsilon^2}\right) dr 
\le {a\over \delta} \|f\|_W W(z^\epsilon_{s\over \epsilon})
 \left|F(y_{s\over \epsilon}^\epsilon)\right|  {\epsilon^2 \over t-s}.
 \end{split}}$$
To the second term we apply Lemma \ref{lemma3} and obtain the bound
$${\begin{split}&
 \left| \E \left\{ \sum_{j=1}^m {\epsilon\over t-s}\int_{s\over \epsilon }^{t\over \epsilon} \int_{s\over \epsilon}^{s_1}
dF (Y_j(y_{s_2}^\epsilon)) \alpha_j (z_{s_2}^\epsilon) f(z_{s_1}^\epsilon)ds_2ds_1\big |\;\F_{s\over \epsilon}\right\}\right|\\
&\le 2 \sum_{j=1}^m \tilde\gamma_\epsilon^j |\alpha_j|_\infty |f|_\infty 
 \left ({\epsilon^2\over t-s} +(t-s)\right)
\end{split}}$$
 where
$$\gamma_\epsilon^j
=  |L_{Y_j}F|(y_{s\over\epsilon}^\epsilon)| \;c_W  ( z_{s\over \epsilon}^\epsilon)
+ \sum_{l=1}^m |\alpha_l|_\infty  
{\epsilon\over t-s}  \int_{s\over \epsilon}^{t\over \epsilon} 
  \E  \left\{  \left|(L_{Y_l} L_{Y_j}F)(y^\epsilon_r)\right|c_W (z^\epsilon_r)
  \; \big|\;\F_{s\over \epsilon} \right\}  dr.
  $$
 Adding the error estimates together  we conclude the proof.
 \end{proof}
 
 It is worth noticing that if $\phi:\R\to \R$ is a concave function $\phi(W)$ is again a Lyapunov function.
Thus by using $\log W$ if necessary, we may assume uniform bounds on $\E W^p(z_{s\over \epsilon}^\epsilon) $
and further estimates on the conditional expectation in the error term 
are expected from  Cauchy-Schwartz inequality.
 If $G$ is compact then  $c_W $ is bounded.
In Corollary \ref{corollary-to-lemma3-2}, we will give uniform estimates on moments of $\gamma^j_\epsilon$.

\section{A Reduction}
\label{section-formula}

Let $G$  be a smooth manifold of dimension $n$ with volume measure $dx$.
Let $H^s\equiv H^s(G)$ denote the Sobolev space of real valued functions over a  manifold $G$ and
$\|-\|_s$ the Sobolev norm.
The norm $(\|u\|_s)^2 :=(2\pi)^{-n}\int |\hat u(\xi)|^2 (1+|\xi|^2)^s d\xi$ 
 extends from domains in $\R^n$ 
 to compact manifolds, e.g. by taking supremum over $\|u\|_s$ on charts.
 If $s\in \N$,  $H^s$ is the completion of $C^\infty(M)$
with the norm $\|u\|_{s} =\sum_{j=0}^s \int (|\nabla^j u|)^2 dx )^{1\over 2}$
where $\nabla$ is usually taken as the Levi-Civita connection;
when the manifold is compact this is independent of the Riemannian metric. And 
$u\in H^s$ if and only if for any function $\phi\in C_K^\infty$, $\phi u$ in any chart 
belongs to $H^s$. 

Let $\{X_i, i=0, 1,\dots, m'\}$ be a family of smooth vector fields on $G$ and let us consider the H\"ormander form operator  $\L_0={1\over 2}\sum_{i=1}^{m'}L_{X_i}L_{X_i}+L_{X_0}$. 
 Let $$\Lambda:= \{X_{i_1}, [X_{i_1}, X_{i_2}], [X_{i_1}, [X_{i_2},X_{i_3}]], i_j =0,1, \dots, m'\}.$$
If  the vector fields  in $\Lambda$ generate $T_xG$ at each $x\in G$, we say that H\"ormander's condition is satisfied.  
 By the proof in a theorem of H\"ormander\cite[Theorem1.1]{Hormander-hypo-acta},   if $\L_0$ satisfies the H\"ormander condition then  $u$ is a $C^\infty$ function in every open set where  $\L_0 u$ is a $C^\infty $ function. 
There is a number $\delta>0$ such that there is an $\delta$ improvement in the Sobolev regularity:
if $u$ is a distribution such that $\L_0 u\in H^s_{\loc}$, then $u\in H^{s+\delta}_{\loc}$.

Suppose that $G$ is compact. Then $\|u\|_\delta \le C(\|u\|_{L^2}+\|\L_0 u\|_{L^2})$, the resolvents $(\L_0+\lambda I)^{-1}$ as operators from $L^2(G;dx)$ to $L^2(G;dx)$ are compact,  and $\L_0$ is  Fredholm on $L^2(dx)$, by which we mean that $\L_0$ is a bounded linear operator from $\Dom(\L_0)$ to $L^2(dx)$ and has the Fredholm property :   its range is closed and of finite co-dimension, the dimension of its kernel, $\ker(\L_0)$ is finite. The domain of $\L_0$ is endowed with the norm $|u|_{\Dom(\L_0)}=|u|_{L_2}+|\L_0 u|_{L_2}$.
 Let $\L_0^*$ denote the adjoint of $\L_0$. Then the kernel $N$ of $\L_0^*$ is finite dimensional and its elements are measures on $M$ with smooth densities in $L^2(dx)$. Denote $N^\perp$ the annihilator of $N$,  $g\in L^2(dx)$ is in $ N^\perp$ if and only if $\<g, \pi\>=0$ for all $\pi\in \ker (\L_0^*)$. Since $\L_0$ has closed range, $(\ker(\L_0^*))^\perp$ can be identified with the range of $\L_0$, and the set of $g$ such that the Poisson equation $\L_0 u=g$ is solvable is exactly $N^\perp$. We denote by $\L_0^{-1}g$ a solution. Furthermore $\L_0^{-1}g$ is $C^r$ whenever $g$ is $C^r$. Denote by $\index(\L_0)$, $\dim \ker \L_0-\dim \,\coker \L_0$, the index of a Fredholm operator $\L_0$, where $\coker=L^2(dx)/\range (\L_0)$. If $\L_0$ is self-adjoint, $\index (\L_0)=0$.

 \begin{definition}
 \label{def-Fredholm}
 We say that $\L_0$ is a regularity improving Fredholm operator, if it is a Fredholm operator and 
  $\L_0^{-1}\alpha$ is $C^r$ whenever  $\alpha \in C^r \cap N^\perp$.
 \end{definition}

Let $\{W_t^k, k=1,\dots, m'\}$ be a family of independent real valued Brownian motions. 
We may and will often represent  $\L_0^\epsilon$-diffusions $(z_t^\epsilon)$ as solutions to the following stochastic differential equations, in Stratonovich form,
$$dz_t^\epsilon ={1\over \sqrt \epsilon}\sum_{k=1}^{m'} X_k(z_t^\epsilon) \circ dW_t^k
+{1\over  \epsilon} X_0(z_t^\epsilon)dt.$$

\begin{lemma}
\label{lemma5}
Let $\L_0$ be a regularity improving Fredholm operator on a compact manifold $G$,
$\alpha_k\in C^3\cap N^\perp$, and $\beta_j=\L_0^{-1}\alpha_j$. Let $(y_r^\epsilon)$ be global solutions of (\ref{sde-3}) on $M$. Then for all $0\le s<t$, $\epsilon>0$  and  $f\in C^2(M; \R)$,  
\begin{equation}\label{Ito-tight}
{\begin{split} 
f(y^\epsilon_{t\over \epsilon})
=&f(y^\epsilon_{s\over \epsilon}) 
+ \epsilon \sum_{j=1}^m \left( df(Y_j(y^\epsilon_{t\over \epsilon} ) )\beta_j(z^\epsilon_{t\over \epsilon})
 -df(Y_j(y^\epsilon_{s\over \epsilon} ))\beta_j( z^\epsilon_{s\over \epsilon})\right)\\
&-\epsilon \sum_{i,j=1}^m\int_{s\over \epsilon}^{t\over \epsilon} L_{Y_i}L_{Y_j} f(y^\epsilon_r))
\alpha_i(z^\epsilon_r)\;\beta_j(z^\epsilon_r) \;dr \\
&-  \sqrt \epsilon  \sum_{j=1}^m\sum_{k=1}^{m'}
\int_{s\over \epsilon}^{t\over \epsilon}  df( Y_j(y^\epsilon_r)) \;
d\beta_j( X_k(z^\epsilon_r)) \;dW_r^k.
\end{split}}\end{equation}
Suppose that, furthermore, for each $\epsilon>0$,    $j,k=1,\dots, m$,
$\int_{s\over \epsilon}^{t\over \epsilon} 
\E |df( Y_j(y^\epsilon_r))|^2 |(d\beta_j( X_k)(z^\epsilon_r)|^2 \;dr$ is finite.
 Then, 
 \begin{equation}\label{Ito-tight-2}
{\begin{split} 
&\E\left\{ f(y^\epsilon_{t\over \epsilon})\big | \F_{s\over \epsilon}\right\} -f(y^\epsilon_{s\over \epsilon}) 
= \epsilon \sum_{j=1}^m \left( \E\left\{  df(Y_j(y^\epsilon_{t\over \epsilon} ) )\beta_j(z^\epsilon_{t\over \epsilon})\big | \F_{s\over \epsilon}\right\} 
 -df(Y_j(y^\epsilon_{s\over \epsilon} ))\beta_j( z^\epsilon_{s\over \epsilon})\right)\\
&-\epsilon \sum_{i,j=1}^m\int_{s\over \epsilon}^{t\over \epsilon} \E\left\{  L_{Y_i}L_{Y_j} f(y^\epsilon_r))
\alpha_i(z^\epsilon_r)\;\beta_j(z^\epsilon_r) \big | \F_{s\over \epsilon}\right\} \;dr.
\end{split} } \end{equation}
\end{lemma}
\begin{proof}
Firstly, for any $C^2$ function  $f:M\to R$,
$$f(y_{t\over \epsilon}^\epsilon)-f(y_{s\over \epsilon}^\epsilon)= \sum_{j=1}^{m}\int_{s\over \epsilon}^{t\over \epsilon} df (Y_j(y_{s_1}^\epsilon )) \alpha_j (z_{s_1}) ds_1.$$
Since the $\alpha_j$'s  are $C^2$ so are $\beta_j$, following from the regularity improving property of $\L_0$. 
We apply It\^o's formula to the functions $(L_{Y_j}f)\beta_j: M\times G\to \R$. To avoid extra regularity conditions,  
we apply It\^o's formula to the function $df(Y_j)$, which is $C^1$, and the $C^3$ functions $\beta_j$ separately and follow it with the product rule. This gives:
\begin{equation*} 
\label{product}
{\begin{split} 
df(Y_j(y^\epsilon_{t\over \epsilon} ) )\beta_j(z^\epsilon_{t\over \epsilon})
&= df(Y_j(y^\epsilon_{s\over \epsilon} ))\beta_j( z^\epsilon_{s\over \epsilon})
+\sum_{j=1}^m \int_{s\over \epsilon}^{t\over \epsilon} L_{Y_i}L_{Y_j} f(y^\epsilon_r)
\,\alpha_i(z^\epsilon_r)\;\beta_j( z^\epsilon_r)\; \;dr \\
&+ {1\over \sqrt \epsilon} \sum_{k=1}^{m'} \int_{s\over \epsilon}^{t\over \epsilon}  
L_{Y_j} f( y^\epsilon_r) \, d\beta_j\left( X_k (z^\epsilon_r)\right)dW_r^k
+{1\over \epsilon} \int_{s\over \epsilon}^{t\over \epsilon}  L_{Y_j} f(y^\epsilon_r) \,\L_0 \beta_j(z_r^\epsilon)dr .
\end{split} }
\end{equation*}
Substitute this into the earlier equation, we obtain (\ref{Ito-tight}). Part (\ref{Ito-tight-2}) is obvious, as
we note that
 $${\begin{split}
 &\E \left(\sum_{k=1}^{m'}
\int_{s\over \epsilon}^{t\over \epsilon}  df( Y_j(y^\epsilon_r)) (d\beta_j)
\left( X_k(z^\epsilon_r)\right) \;dW_r^k\right)^2
\le 
 \sum_{k=1}^{m'}\E\int_{s\over \epsilon}^{t\over \epsilon}   df( Y_j(y^\epsilon_r))  |^2 |d\beta(X_k(z_r^\epsilon))|^2 | \d r<\infty
 \end{split}}$$
and the stochastic integrals are  $L^2$-martingales, so (\ref{Ito-tight-2}) follows.
 \end{proof} 
 
 When $G$ is compact, $d\beta(X_k)$ is bounded and the condition becomes:
 $\E\int_{s\over \epsilon}^{t\over \epsilon}   df( Y_j(y^\epsilon_r))  |^2 \d r$ is finite,
which we discuss below. Otherwise, assumptions on
$\E |d\beta(X_k(z_r^\epsilon))|^{2+} $ is needed.

\section{Uniform Estimates}
 \label{section-uniform-estimates}

 If $V: M\to \R_+$ is a locally bounded function such that $\lim_{y\to \infty} V(y)=\infty$
we say that $V$ is a pre-Lyapunov function. 
Let  $\alpha_k\in \B_b(G;\R)$. Let $\{Y_k\}$ be $C^1$ smooth vector fields on $M$
 such that:  either (a) each $Y_k$ grows at most linearly; or (b) 
there exist a pre-Lyapunov function $V\in C^1(M;\R_+)$, positive constants $c$ and $K$ such that 
$\sum_{k=1}^m |L_{Y_k} V| \le c+KV$ then  the equations (\ref{sde-3})
are complete.  
In case (a) let $o\in M$ and $a$ be a constant such that  $|Y_k(x)|\le a(1+\rho(o,x))$
where $\rho$ denotes the Riemannian distance function on $M$. 
 For $x$ fixed, denote $\rho_x=\rho(x, \cdot)$.
By the definition of the Riemannian distance function,
$$
\rho(y_t^\epsilon, y_0)\le \int_0^t |\dot y_s^\epsilon|ds =\sum_{k=1}^m  \int_0^t |Y_k(y_s^\epsilon) \alpha_k(z_s^\epsilon) |ds
\le\sum_{k=1}^m |\alpha_k|_\infty  \int_0^t |Y_k(y_s^\epsilon)| ds.$$
This together with the inequality  $\rho(y_t^\epsilon, o)\le \rho(y_t^\epsilon, y_0)+\rho(o, y_0^\epsilon)$ implies that
for all $p\ge 1$, there exist constants $C_1, C_2$ depending on $p$ such that
$$ \sup_{s\le t}\rho^p(y_s^\epsilon, o)\le \left(C_1 \rho^p(o, y_0^\epsilon)+C_2 t\right)  e^{C_2t^p}$$
where $C_2=a^pC_1^2( \sum_{k=1}^m |\alpha_k|_\infty)^p$.  When restricted to $\{t<\tau^\epsilon\}$, the first time $y_t^\epsilon$ reaches the cut locus, the bounded is simple $Ce^{Ct}$.
In case (b),  for any $q\ge 1$, 
$$  \sup_{s\le t} \left(V(y_s^\epsilon)\right)^q\le \left( V^q(y_0^\epsilon)+ctq\sum_{k=1}^m|\alpha_k|_\infty  \right)
\exp{\left( q\sum_{k=1}^m|\alpha_k|_\infty(K+c)t\right)},$$
which followed easily from the bound $$|dV^q(\alpha_k Y_k)|=|qV^{q-1}dV(\alpha_kY_k)|\le q|\alpha_k|_\infty(c+(c+K) V^q).$$

For the convenience of comparing the above estimates, which are standard and expected,
 with the uniform estimates of $(y_t^\epsilon)$ in Theorem
\ref{uniform-estimates} below in the time scale ${1\over \epsilon}$, we record this in the following Lemma.

 \begin{lemma}
\label{standard-estimates}
Let $\alpha_k \in \B_b(G;\R)$. Let $\epsilon\in (0,1)$, $0\le s\le t$, $\omega\in \Omega$.
\begin{enumerate}
\item If $\{Y_k\}$ grow at most linearly then  (\ref{sde-3}) is complete  and  there exists $C,C(t)$ s.t.
$$  \sup_{0\le s\le t}\rho^p(y_s^\epsilon(\omega), o)\le \left(C \rho^p(o, y_0^\epsilon(\omega))+C(t)\right) e^{C(t)}.$$
\item If there exist a pre-Lyapunov function $V\in C^1(M;\R_+)$, positive constants $c$ and $K$ such that 
$\sum_{j=1}^m |L_{Y_j} V| \le c+KV$, then (\ref{sde-3}) is complete. 
\item If (\ref{sde-3}) is complete and there exists  $V\in C^1(M;\R_+)$ such that 
$\sum_{j=1}^m |L_{Y_j} V| \le c+KV$ then there exists a constant $C$,  s.t.
$$ \sup_{0\le s\le t} \left(V(y_s^\epsilon(\omega))\right)^q\le \left(   \left( V(y_0^\epsilon(\omega))\right)^q+Ct \right) e^{Ct}.$$

\end{enumerate}

\end{lemma}
 
 If $V\in \B(M;\R)$ is a positive function, denote by $B_{V,r}$ the following classes of functions:
\begin{equation}
\label{BVr}
B_{V,r}=\left\{ f\in C^r(M;\R): \sum_{j=0}^r |d^{j}f|\le c+c V^q \hbox{ for some numbers $c$, $q$ }\right\}.
\end{equation}
In particular, $B_{V,0}$ is the class of continuous functions bounded by a function of the form $ c+c V^q$. In $\R^n$, the constant functions and the function $V(x)=1+|x|^2$ are potential `control' functions. 
 
\begin{assumption}
\label{assumption2-Y}
Assume that (i)  (\ref{sde-3}) are complete for every $\epsilon\in (0,1)$, (ii) $\sup_\epsilon\E\left( V(y_0^\epsilon)\right)^q$ is finite for every $q\ge 1$;
 and (iii)  there exist a function $V\in C^2(M;\R_+)$, positive constants $c$ and $K$ such that 
$$\sum_{j=1}^m |L_{Y_j} V| \le c+KV, \quad \sum_{i,j=1}^m |L_{Y_i}L_{Y_j} V|  \le c+KV.$$
\end{assumption}

Below we assume that $\L_0$ satisfies H\"ormander's condition. We do not make any assumption on the  mixing rate.
  Let $\beta_j=\L_0^{-1}\alpha_j$, $a_1= \sum_{j=1}^m |\beta_j|_\infty$,  $a_2=\sum_{i,j=1}^m |\alpha_i|_\infty|\beta_j|_\infty$,
  $a_3=\sum_{j=1}^m |d\beta_j|_\infty$, and $a_4=\sum_{k=1}^m |X_k|^2_\infty$.
  We recall that if $\alpha_k$ and $\L_0$ satisfy  Assumption \ref{assumption-Hormander} then
  $\L_0$ is a regularity improving Fredholm operator.
  
\begin{theorem}
\label{uniform-estimates}
Let  $\L_0$ be a regularity improving Fredholm operator on a compact manifold $G$,
and $\alpha_k\in C^3(G;\R)\cap N^\perp$.
Assume  that $Y_k$ satisfy Assumption \ref{assumption2-Y}. Then for all $p\ge 1$, there exists  a constant $C=C(c, K, a_i,p)$
s.t. for any $0\le s\le t$ and all $\epsilon\le\epsilon_0$,
\begin{equation}
\label{uniform-estimate-inequality}
 \E \left\{ \sup_{s\le u\le t}  \left( V(y^\epsilon_{u\over \epsilon}) \right)^{2p} \; \big| \; \F_{s\over \epsilon} \right\}
\le  \left( 4 V^{2p}(y^\epsilon_{s\over \epsilon}) +C(t-s)^2+C \right) 
e^{ C(t-s+1)t}.
\end{equation}
Here  $\epsilon_0\le 1$ depends on $c,K, p, a_1$ and $V, Y_i, Y_j$.
\end{theorem}

\begin{proof}
Let $0\le s\le t$.  We apply (\ref{Ito-tight}) to $f=V^p$:
\begin{equation*}
{\begin{split} 
V^p(y^\epsilon_{t\over \epsilon})
=& V^p(y^\epsilon_{s\over \epsilon}) 
+ \epsilon \sum_{j=1}^m dV^p\left(Y_j(y^\epsilon_{t\over \epsilon} )\right )\beta_j(z^\epsilon_{t\over \epsilon}) 
 -\epsilon \sum_{j=1}^m dV^p\left (Y_j(y^\epsilon_{s\over \epsilon} )\right)\beta_j( z^\epsilon_{s\over \epsilon})\\
&-\epsilon  \sum_{i,j=1}^m  \int_{s\over \epsilon}^{t\over \epsilon} L_{Y_i}L_{Y_j} V^p\left(y^\epsilon_r\right)
\alpha_i(z^\epsilon_r)\;\beta_j(z^\epsilon_r) \d r\\
&-\sqrt\epsilon \sum_{k=1}^p \int_{s\over \epsilon}^{t\over \epsilon}
\sum_{j=1}^m dV^p (Y_j(y_r^\epsilon))(d\beta_j) (X_k(z_r^\epsilon)) \;dW_r^k.
\end{split}}\end{equation*}

In the following estimates, we may first assume that $\sum_{j=1}^m |L_{Y_j}V|$ and $\sum_{i,j=1}^m |L_{Y_j}L_{Y_i}V|$ are bounded.  We may then replace $t$ by $t\wedge \tau_n$ where $\tau_n$ is the first time that either quantity is greater or equal to $n$. We take this point of view for proofs of inequalities and may not repeat it each time. 

We take the supremum over $[s,t]$ followed by conditional expectation of both sides of the inequality:
\begin{equation*}
{\begin{split} 
& \E \left\{ \sup_{s\le u\le  t} V^p(y^\epsilon_{u\over \epsilon}) \; \big| \; \F_{s\over \epsilon} \right\}
\le V^p(y^\epsilon_{s\over \epsilon}) 
+ \epsilon  \E \left\{ \sup_{s\le u\le  t} 
  \sum_{j=1}^m  dV^p\left(Y_j(y^\epsilon_{u\over \epsilon} )\right )\beta_j(z^\epsilon_{u\over \epsilon}) \; \big| \; \F_{s\over \epsilon} \right\}\\
 & \qquad - \sum_{j=1}^m dV^p\left (Y_j(y^\epsilon_{s\over \epsilon} )\right)\beta_j( z^\epsilon_{s\over \epsilon})\\
&\qquad +\epsilon  \E\left\{ \sup_{s\le u\le  t} \left|  \int_{s\over \epsilon}^{u\over \epsilon}  \sum_{i,j=1}^mL_{Y_i}L_{Y_j} V^p\left(y^\epsilon_r\right)
\alpha_i(z^\epsilon_r)\;\beta_j(z^\epsilon_r)  \;dr \right|\; \big| \; \F_{s\over \epsilon} \right\}\\
&\qquad +\sqrt\epsilon\E \left\{ \sup_{s\le u\le  t} \left| \sum_{k=1}^{m'} \int_{s\over \epsilon}^{u\over \epsilon}
\sum_{j=1}^m dV^p (Y_j(y_r^\epsilon))(d\beta_j) (X_k(z_r^\epsilon)) dW_r^k\right|  \; \big| \; \F_{s\over \epsilon} \right\}.
\end{split}}\end{equation*}
By the conditional Jensen inequality, Doob's inequality and the flow property,
 there exists a universal constant $\tilde C$ 
depending only on $p$ s.t.,

\begin{equation*}
{\begin{split} 
& \E\left\{  \sup_{s\le u\le t} V^{2p}(y^\epsilon_{u\over \epsilon}) \; \big| \; \F_{s\over \epsilon} \right\}\\
&\le 4  V^{2p}(y^\epsilon_{s\over \epsilon}) 
+ 4\epsilon^2  \E\left( \left\{   \sum_{j=1}^m   |\beta_j|_\infty \sup_{s\le u \le t} 
 \left| dV^p(Y_j(y^\epsilon_{u\over \epsilon} ))\right|  \; \big| \; \F_{s\over \epsilon} \right\}\right)^2\\
&+ 4\epsilon^2 \left( \sum_{j=1}^m   |\beta_j|_\infty \left| dV^p(Y_j(y^\epsilon_{s\over \epsilon} ))\right| \right)^2 \\
& +8\epsilon(t-s) \E\left\{\left( \int_{s\over \epsilon}^{t\over \epsilon} 
  \sum_{i,j=1}^m  |\alpha_i|_\infty|\beta_j|_\infty \left| L_{Y_i}L_{Y_j} V^p\left(y^\epsilon_r\right)\right|\;dr\right)^2  \; \big| \; \F_{s\over \epsilon} \right\}\\
&+\tilde C  \sum_{k=1}^p\E  \left\{
\epsilon \int_{s\over \epsilon}^{t\over \epsilon}  \left|  \sum_{j=1}^m   dV^p( Y_j(y^\epsilon_r)) (d\beta_j)  \left( X_k(z^\epsilon_r)\right) 
\right|^2 \d r  \; \big| \; \F_{s\over \epsilon} \right\}.
\end{split}}
\end{equation*}

Since $\sum_j|L_{Y_j}V|\le c+KV$ and $\sum_{i,j=1}^p |L_{Y_i}L_{Y_j}V|\le c+KV$,
there are constants $c_1$ and $K_1$ such that
$\max_{j=1, \dots, m}|L_{Y_j}V^p |\le c_1+K_1V^p$ and  $\max_{i,j=1, \dots, m} |L_{Y_i}L_{Y_j}V^p|\le c_1+K_1V^p$.
This leads to the following estimate:

\begin{equation*}
{\begin{split} 
&\E\left\{  \sup_{s\le u\le t} V^{2p}(y^\epsilon_{u\over \epsilon}) \; \big| \; \F_{s\over \epsilon} \right\}\\
\le&
 4   V^{2p}(y^\epsilon_{s\over \epsilon}) +
8\epsilon^2(a_1)^2  \left(  2(c_1)^2+ (K_1)^2 \E\left \{\sup_{s\le u \le t}   V^{2p}(y^\epsilon_{u\over \epsilon} )   \; \big| \; \F_{s\over \epsilon}\right\}
+(K_1)^2 V^{2p}(y^\epsilon_{s\over \epsilon} )\right) \\
&+16 (a_2)^2(t-s) \epsilon  \int_{s\over \epsilon}^{t\over \epsilon}  \left( (c_1)^2+ (K_1)^2\E\left\{  V^{2p}(y^\epsilon_r )  \; \big| \; \F_{s\over \epsilon}\right\}\right)dr \\
&+ \tilde C (a_3a_4)^2  \epsilon
\int_{s\over \epsilon}^{t\over \epsilon}  \E \left\{ \left (c_1+K_1 V^p((y^\epsilon_r))\right)^2  \; \big| \; \F_{s\over \epsilon} \right\}\d r.
\end{split}}\end{equation*}
Let  $\epsilon_0=\min\{{1\over 8a_1K_1}, 1\}$. For $\epsilon\le\epsilon_0$,
\begin{equation*}
{\begin{split} 
&{1\over 2}  \E\left\{  \sup_{s\le u\le t} V^{2p}(y^\epsilon_{u\over \epsilon}) \; \big| \; \F_{s\over \epsilon} \right\}\\
\le & 4  V^{2p}(y^\epsilon_{s\over \epsilon}) 
+16\epsilon^2 ( a_1c_1)^2  +16(t-s)^2 (a_2c_1)^2+ 4\tilde C (a_3a_4c_1)^2(t-s)\\
&+ \left( 16 (a_2K_1)^2(t-s) +4\tilde C (a_3a_4K_1)^2  \right)  \epsilon 
 \int_{s\over \epsilon}^{t\over \epsilon}   \E \left\{ V^{2p}(y^\epsilon_r ) \; \big| \; \F_{s\over \epsilon} \right\} \;dr.  
\end{split}}\end{equation*}
It follows that there exists a constant $ C$ such that  for $\epsilon\le \epsilon_0$,
\begin{equation*}
\E\left\{  \sup_{s\le u\le t} V^{2p}(y^\epsilon_{u\over \epsilon}) \; \big| \; \F_{s\over \epsilon} \right\}
\le  \left( 4 V^{2p}(y^\epsilon_{s\over \epsilon}) +C(t-s)^2+C \right) 
e^{ C(t-s+1)t}.
\end{equation*}

\end{proof}
 
 {\it Remark.}
 If $M=\R^n$, $Y_i$ are vector fields with bounded first order derivatives, then $\rho_0^2$ is a 
 pre-Lyapunov function satisfying  the conditions of Theorem \ref{uniform-estimates}, hence Theorem \ref{uniform-estimates}
 holds. 
 Let us recall that $B_{V,r}$ is defined in (\ref{BVr}).

We return to Lemma \ref{lemma3-2} in Section \ref{section-estimates} to obtain a key estimation
 for the estimation in Section \ref{section-rate}.
Let us recall that $B_{V,r}$ is defined in (\ref{BVr}).

\begin{corollary}
\label{corollary-to-lemma3-2}
Assume  (\ref{sde-3}) is complete, for every $\epsilon \in (0,1)$, and
conditions of Assumption \ref{assumption1}.
 Let $V\in \B(M;\R_+)$ be a locally bounded function and  $\epsilon_0$
a positive number s.t. for all $q\ge 1$ and $T>0$,  there exists a locally bounded function $C_q: \R_+\to \R_+$, a 
real valued polynomial $\lambda_q$ such that for $0\le s\le t\le T$ and for all $\epsilon\le\epsilon_0$
\begin{equation}
\label{moment-assumption-20}
 \sup_{s\le u \le t}  \E\left\{ V^q(y_{u\over \epsilon}^\epsilon) \; \big| \F_{s\over \epsilon} \right\} 
 \le C_q(t)+C_q(t) \lambda_q(  V(y_{s\over \epsilon}^\epsilon)), 
 \quad \sup_{0<\epsilon \le \epsilon_0}\E (V^q(y_0^\epsilon))<\infty.
\end{equation}
 Let $h\in \B_b( G; \R)$. If $f\in B_{V,0}$ is a function s.t.  $L_{Y_j}f\in B_{V,0}$ and $ L_{Y_l}L_{Y_j}f\in B_{V,0}$
 for all $j, l=1, \dots, m$,  then
 for all $0\le s\le t$,
 $$\left| {\epsilon \over t-s} \int_{s\over \epsilon }^{t\over \epsilon} \E \left\{ f (y_r^\epsilon) h(z_{r}^\epsilon) \big| \F_{s\over \epsilon}\right\} dr
-\bar h\;   f(y_{s\over \epsilon}^\epsilon) \right|  \le \tilde c |h|_\infty  \gamma_\epsilon(y_{s\over \epsilon}^\epsilon) \left ({\epsilon^2\over t-s}+(t-s)\right).
$$
Here $\tilde c$ is a constant, see (\ref{cpsi}) below, and
$$\gamma_\epsilon=|f| + \sum_{j=1}^m |L_{Y_j}f|+\sum_{j,l=1}^m
{\epsilon\over t-s}  \int_{s\over \epsilon}^{t\over \epsilon} 
  \E  \left\{  \left|L_{Y_l} L_{Y_j}f(y^\epsilon_r)\right | \; \big|\;\F_{s\over \epsilon} \right\}  dr.$$
For all $s<t$ and $p\ge 1$, $$\sup_{s\le u\le t }\sup_{\epsilon \le \epsilon_0}
\E\left( \gamma_\epsilon(y_{u\over \epsilon}^\epsilon)\right)^p<\infty.$$
More explicitly, if $\sum_{j=1}^m\sum_{l=1}^m|L_{Y_l} L_{Y_j}f|\le K+KV^q$ where $K, q$ are constants,
then there exists a constant $C(t)$ depending only on the differential equation (\ref{sde-3}) s.t.
   $$\gamma_\epsilon\le |f|+ \sum_{j=1}^m  |L_{Y_j}f|+K+C(t)V^q.$$
\end{corollary}
\begin{proof}
By Lemma \ref{lemma3-2}, 
$${\begin{split}
&\left| {\epsilon \over t-s} \int_{s\over \epsilon }^{t\over \epsilon} \E \left\{ f (y_r^\epsilon) h(z_{r}^\epsilon) \big| \F_{s\over \epsilon}\right\} dr
-\bar h\;   f(y_{s\over \epsilon}^\epsilon) \right| \\
& \le {2a\over \delta} |h|_\infty \left( W(z_{s\over \epsilon}^\epsilon)|f(y_{s\over \epsilon}^\epsilon)|
+\sum_{j=1}^\m \gamma^j_\epsilon  |\alpha_j|_\infty \right)\left ({\epsilon^2\over t-s} +(t-s)\right),\\
\hbox{ where }
\gamma^j_{\epsilon}(y)&=c_W (z_{s\over \epsilon}^\epsilon) 
 \;|L_{Y_j}f (y^\epsilon_{s\over\epsilon})|+ \sum_{l=1}^m |\alpha_l|_\infty  
{\epsilon\over t-s}  \int_{s\over \epsilon}^{t\over \epsilon} 
  \E  \left\{  \left|L_{Y_l} L_{Y_j}f(y^\epsilon_r)\right|c_W (z^\epsilon_r)
  \; \big|\;\F_{s\over \epsilon} \right\}  dr.
\end{split}}
$$
Since $W$ is bounded so is $c_W $, which is bounded by $2c(a,\delta) |W|_\infty$.
Furthermore
$$  \E  \left\{  \left|L_{Y_l} L_{Y_j}f(y^\epsilon_r)\right|c_W (z^\epsilon_r)
  \; \big|\;\F_{s\over \epsilon} \right\}  dr
\le   2c(a,\delta) |W|_\infty\E  \left\{  \left|L_{Y_l} L_{Y_j}f(y^\epsilon_r)\right|
  \; \big|\;\F_{s\over \epsilon} \right\}  dr.$$
We gather all constant together,
\begin{equation}
\tilde c= {2a\over \delta}|W|_\infty+2c(a,\delta)|W|_\infty\sum_{j,l=1}^m |\alpha_j|_\infty+2\left(\sum_{j=1}^m |\alpha_j|_\infty\right)^2.
\label{cpsi}
\end{equation}
It is clear that, $$\left| {\epsilon \over t-s} \int_{s\over \epsilon }^{t\over \epsilon} \E \left\{ f (y_r^\epsilon) h(z_{r}^\epsilon) \big| \F_{s\over \epsilon}\right\} dr
-\bar h\;   f(y_{s\over \epsilon}^\epsilon) \right| 
\le  \tilde c\, \gamma_\epsilon |h|_\infty\left ({\epsilon^2\over t-s} +(t-s)\right).$$
 Since $f$, $L_{Y_j}$ and $L_{Y_l}L_{Y_j}f\in B_{V,0}$,  by (\ref{moment-assumption-20}), the following quantities
 are finite for all $p\ge 1$:
$$\sup_{\epsilon \le \epsilon_0} \sup_{s\le u \le t}  \E   \left|  (L_{Y_l}L_{Y_j}f)(y_{u\over \epsilon}^\epsilon)\right| ^p, \quad
 \sup_{\epsilon \le \epsilon_0}    \sup_{s\le u \le t} \E   \left|  L_{Y_j}f(y_{u\over \epsilon}^\epsilon)\right| ^p, \quad 
 \sup_{\epsilon \le \epsilon_0}  \sup_{s\le u \le t}  \E   \left| f(y_{u\over \epsilon}^\epsilon)\right|^p.
$$ 
Furthermore since $\sum_{j=1}^m\sum_{l=1}^m|L_{Y_l} L_{Y_j}f|\le K+KV^q$,
 \begin{equation*}
 {\begin{split}
\sum_{j=1}^m\sum_{l=1}^m {\epsilon\over t-s}  \int_{s\over \epsilon}^{t\over \epsilon} 
   \E  \left\{  \left|L_{Y_l} L_{Y_j}f(y^\epsilon_r)\right|
  \; \big|\;\F_{s\over \epsilon} \right\}  dr
  \le K+C(t) V^q(y_{s\over \epsilon}^\epsilon).
 \end{split}}
 \end{equation*}
Consequently,  $\gamma_\epsilon \le |f|+ \sum_{j=1}^m  |L_{Y_j}f|+K+C(t)V^q$, completing the proof.
 
\end{proof}

\section{Convergence under H\"ormander's Conditions}
\label{section-weak}
Below $\inj(M)$ denotes the injectivity radius of $M$ and
 $\rho_y=\rho(y, \cdot)$ is the Riemannian distance function on $M$ from a point $y$.    Let $o$ denote a point in $M$.
 The following proposition applies to an operator $\L_0$, on a compact manifold,
  satisfying H\"ormander's condition.
    \begin{proposition}
\label{tightness} 
Let $M$ be a manifold with positive injectivity radius and $\epsilon_0>0$. Suppose  conditions
(1-5) below or conditions (1-3), (4') and (5).
 \begin{enumerate}
\item [(1)]  $\L_0$ is a regularity improving Fredholm operator on $L^2(G)$ 
for a compact manifold $G$;
\item [(2)]  $\{\alpha_k\} \subset C^3\cap N^\perp$;
\item [(3)] Suppose that for $\epsilon \in (0, \epsilon_0)$, (\ref{sde-3}) is complete and $\sup_{\epsilon \le \epsilon_0}\E \rho(y_0^\epsilon, o)<\infty$;
\item [(4)] Suppose that there exists a  locally bounded function $V$ 
s.t.  for all $\epsilon\le\epsilon_0$ and for any $0\le s\le u\le t$, and for all $p\ge 1$, 
$$\E V^p(y^\epsilon_0) \le c_0,  
\quad  \sup_{s\le u\le t}\E \left\{   \left( V(y^\epsilon_{u\over \epsilon}) \right)^p \; \big| \; \F_{s\over \epsilon} \right\}
\le  K +K V^{p'}(y^\epsilon_{s\over \epsilon})$$
where $c_0=c_0(p)$, $K=K(p,t)$, and $ p'=p'(p,t)$ is a natural number; $K, p'$
 are  locally bounded in $t$.
 \item[(4')]  There exist a function $V\in C^2(M;\R_+)$, positive constants $c$ and $K$ such that 
$$\sum_{j=1}^m |L_{Y_j} V| \le c+KV, \quad \sum_{i,j=1}^m |L_{Y_i}L_{Y_j} V|  \le c+KV.$$
\item [(5)] For $V$ in part (4) or in part (4'),  suppose that for some number $\delta>0$,
$$|Y_j|\in B_{V,0} \quad  \sup_{\rho(y,\cdot) \le \delta}| L_{Y_i}L_{Y_j}\rho_y(\cdot) |\in B_{V,0}.$$
\end{enumerate} 
 Then there exists a distance function $\tilde \rho$ on $M$ that is compatible with the topology of $M$ 
 and there exists a number $\alpha>0$ such that
$$\sup_{\epsilon\le \epsilon_0} \E \sup_{s\not =t} \left({\tilde \rho\left( y_{s\over \epsilon}^\epsilon,  y_{t\over \epsilon}^\epsilon\right) \over
|t-s|^\alpha }\right)<\infty,$$
and for any $T>0$, $\{ (y_{t\over\epsilon}^\epsilon, t\le T), 0<\epsilon\le 1\}$ is tight.
\end{proposition}

\begin{proof}
By Theorem \ref{uniform-estimates}, conditions (1-3) and (4') imply condition (4).
(a) Let  ${\delta}<\min (1, {1\over 2}  \inj(M))$. 
Let $f: \R_+\to \R_+$ be a smooth convex function such that $f(r)=r$ when $r\le{\delta\over 2}$ and $f(r)=1$ when $r\ge \delta$.
Then $\tilde \rho(x,y)=f\circ \rho$ is a distance function with $\tilde \rho \le 1$. Its open sets generate the same topology on $M$ as that by $\rho$.
 Let $\beta_j$ be a solution to
$\L_0\beta_j=\alpha_j$. For any $y_0\in M$,
 $|L_{Y_j}\tilde \rho^2(y_0,\cdot)|\le 2|Y_j(\cdot)|$. Since $|Y_j|\in B_{V,0}$,
$\int_0^{t\over \epsilon} \E |L_{Y_j} \tilde \rho|(y_r^\epsilon)|^2 dr<\infty$.
 We may apply (\ref{Ito-tight-2}) in 
Lemma \ref{lemma5},
$${ \begin{split}
&\E\left\{ \tilde \rho^2\left(y^\epsilon_{s\over \epsilon}, y_{t\over \epsilon}^\epsilon\right)
\;\big |\; \F_{s\over \epsilon}\right\}
 \\
=& \epsilon \sum_{j=1}^m \left(
 \E\left\{  \left(L_{Y_j}\tilde \rho^2(y^\epsilon_{s\over \epsilon}, y_{t\over \epsilon}^\epsilon)\right) 
\;\beta_j(z^\epsilon_{t\over \epsilon})\; \big |\; \F_{s\over \epsilon}\right\} 
-   \left(L_{Y_j}\tilde \rho^2(y^\epsilon_{s\over \epsilon}, \cdot)\right) (y_{s\over \epsilon}^\epsilon)
\;\beta_j(z^\epsilon_{s\over \epsilon})\right)\\
&-\epsilon \sum_{i,j=1}^m\int_{s\over \epsilon}^{t\over \epsilon} 
\E\left\{  \left(L_{Y_i}L_{Y_j} \tilde \rho^2( y^\epsilon_{s\over \epsilon}, y^\epsilon_r)\right)\;
\alpha_i(z^\epsilon_r)\;\beta_j(z^\epsilon_r)\; \big |\; \F_{s\over \epsilon}\right\} \;dr.
\end{split} } $$
In the above equation,  differentiation of $(\tilde \rho)^2 $ is w.r.t. to the second variable.
By construction $ \tilde \rho $ is bounded by $1$ and $|\nabla \tilde \rho |\le |\nabla \rho|\le 1$.  Furthermore  since $\alpha_j$ are $C^3$ functions on a compact manifold, so  $\beta_j$ and
$|\beta_j|$ are bounded. For any $y_0\in M$,
$L_{Y_j} \tilde \rho  ( y_0, \cdot)=\gamma'(\rho_{y_0} )L_{Y_j} \rho_{y_0} $.
Thus
$$ \left|\E\left\{  \left(L_{Y_j}\tilde \rho^2( y_{s\over \epsilon}^\epsilon, y^\epsilon_{t\over \epsilon})\right)
\;\beta_j(z^\epsilon_{t\over \epsilon}) \;\big |\; \F_{s\over \epsilon}\right\} \right|
\le |\beta_j|_\infty \E \left\{\tilde \rho(y_{s\over \epsilon}^\epsilon,{y^\epsilon_{t\over \epsilon}}) |Y_j({y^\epsilon_{t\over \epsilon}})|  \;\big |\; \F_{s\over \epsilon}\right\} .$$
Recall $\tilde \rho\le 1$ and  there are numbers $K_1$ and $p_1$ s.t. $|Y_j|\le K_1+K_1V^{p_1}$,
so
$$ \E \left\{ |Y_j({y^\epsilon_{t\over \epsilon}})|  \;\big |\; \F_{s\over \epsilon}\right\} 
\le K_1+K_1\E\left\{ V^{p_1}({y^\epsilon_{t\over \epsilon}})  \;\big |\; \F_{s\over \epsilon}\right\} 
\le K_1+K_1K(p_1,t) V^{p'(p_1,t)}({y^\epsilon_{s\over \epsilon}}).$$
 Let $g_1 =K_1+K_1K(p_1) V^{p'(p_1,t)}$, it is clear that $g_1\in B_{V,0}$.
We remark that,
$$L_{Y_i}L_{Y_j}(\tilde \rho^2) =(f^2)''(\rho) (L_{Y_i}\rho) (L_{Y_j}\rho)
 + (f^2)'(\rho) L_{Y_i}L_{Y_j}\rho.$$
By the assumption, there exists a function $g_2\in B_{V,0}$ s.t.
$$\E\left\{ \tilde \rho^2 \left(y^\epsilon_{s\over \epsilon}, y_{t\over \epsilon}^\epsilon\right)
\big | \F_{s\over \epsilon}\right\}
\le g_2(y_{s\over \epsilon})\epsilon+g_2(y_{s\over \epsilon})(t-s) .$$

For $\epsilon\ge \sqrt {t-s}$, it is better to estimate directly from (\ref{sde-3}):
\begin{equation*}
{\begin{split}
\E\left\{ \tilde \rho^2 \left(y^\epsilon_{s\over \epsilon}, y_{t\over \epsilon}^\epsilon\right)
\;\big |\; \F_{s\over \epsilon}\right\} 
&=\sum_{k=1}^m\int_{s\over \epsilon}^{t\over \epsilon} 
 \E\left\{ 2\tilde \rho \left(y^\epsilon_{s\over \epsilon}, y_{t\over \epsilon}^\epsilon\right) 
 L_{Y_k}\tilde \rho \left(y^\epsilon_{s\over \epsilon}, y_{t\over \epsilon}^\epsilon\right)
 \alpha_k(z_r^\epsilon)
 \;\big |\; \F_{s\over \epsilon}\right\} 
\\
&\le 2|\alpha_k|_\infty \sum_{k=1}^m  \int_{s\over \epsilon}^{t\over \epsilon}
 \E \left\{ |Y_k(y_r^\epsilon)|  \;\big |\; \F_{s\over \epsilon}\right\} \;dr
 \le g_3(y^\epsilon_{s\over \epsilon}) \left({t-s \over \epsilon}\right)
\end{split}}
\end{equation*}
where $g_3\in B_{V,0}$.
We interpolate these estimates and conclude that for some function $g_4\in B_{V,0}$ and a constant $c$
the following holds:
$\E\left\{  \tilde \rho^2\left(y_{t\over \epsilon}^{\epsilon}, y_{s\over \epsilon}^\epsilon\right)\;\big |\; \F_{s\over \epsilon}\right\}
\le (t-s)g_4(y^\epsilon_{s\over \epsilon})$. 
There is a function $g_5\in B_{V,0}$ s.t.
$$\E \tilde \rho^2\left(y_{t\over \epsilon}^{\epsilon}, y_{s\over \epsilon}^\epsilon\right)
\le \E g_5(y_0^\epsilon) (t-s)\le c(t-s).$$
In the last step we use Assumption (4) on the initial value.
By Kolmogorov's criterion, there exists $\alpha>0$ such that
$$\sup_\epsilon \E \sup_{s\not =t} \left({\tilde \rho^2 ( y_{s\over \epsilon}^\epsilon,  y_{t\over \epsilon}^\epsilon) \over
|t-s|^\alpha }\right)<\infty,$$
and the processes $(y_{s\over \epsilon}^\epsilon)$ are equi uniformly
H\"older continuous on any compact time interval.
Consequently  the family of stochastic processes
$\{ y_{t\over \epsilon}^\epsilon, 0<\epsilon \le 1 \}$ is tight.
\end{proof}

 If $\L_0$ is the Laplace-Beltrami operator on a compact Riemannian manifold and $\pi$ its invariant probability measure then for any Lipschitz 
continuous function $f:G\to \R$,
\begin{equation}
\label{lln1}
\sqrt{\E \left({1\over t} \int_0^t f(z_s) ds-\int f \d\pi \right)^2} \le C(\|f\|_{Osc}){1\over \sqrt t}.
\end{equation}
where $\|f\|_{Osc}$ denotes the oscillation of $f$. If $\L_0$ is not elliptic we suppose it satisfies 
H\"ormander's conditions and has index $0$. The dimension of the kernel of  $\L_0^*$ equals the dimension of the kernel of $\L_0$.  Let $\{u_i, i=1, \dots, n_0\}$ be a basis in $\ker( \L_0)$ 
and $\{\pi_i\, i=1, \dots, n_0\}$ the dual basis for the null space of $\L_0^*$.  For  $f\in L^2(G;\R)$ we define
 $\bar f= \sum_{i=1}^{n_0} u_i\<f, \pi_i\>$  where the
bracket denotes the dual pairing between $L^2$ and $(L^2)^*$. 
   \begin{lemma}
   \label{lln}
 Suppose that $(z_t)$ is a Markov process on a compact manifold $G$ with 
generator $ \L_0$  satisfying H\"ormander's condition and having Fredholm index $0$. Then for any function $f\in C^r(G; \R)$, where $r\ge \max{\{3, {n\over 2}+1\}}$, there is a constant
$C$ depending on $|f|_{{n\over 2}+1}$, s.t.
\begin{equation}
\label{llln-1}
\sqrt{\E \left({1\over t-s} \int_s^t f(z_r) dr- \bar f \right)^2} \le C(\|f-\bar f\|_{{n\over 2}+1}) {1\over \sqrt{t-s}}.
\end{equation}
 \end{lemma}
 \begin{proof}
 Since $\<f, \pi_j\>=\<f, \pi_j\>$, $f-\bar f\in N^\perp$.
By working with $f-\bar f$ 
we may assume that $f\in N^\perp$ and let $g$ be a solution to $\L_0 g=f$.
 By H\"ormander's theorem, \cite{Hormander-hypo-acta}, there is a positive number $\delta$, such that for all $u\in C^\infty(M)$,
   $$\|u\|_{s+\delta}  \le C( \|\L_0 u\|_s+\|u\|_{L_2} ).$$
   The number  $\delta=2^{1-k}$ where $k\in \N$ is related to the number of brackets needed
   to generate the tangent spaces.

 Furthermore every $u$ such that  $\|\L_0 u\|_s<\infty $ must be in $H^s$.
If $s>{n\over 2}+1$,  $H^s$ is embedded in $C^1$ and for some constant $c_i$, 
\begin{equation*}
{\begin{split}
|g|_{C^1(M)}\le c_1 \,\|g\|_{ {n\over 2}+1+\epsilon}\le c_2 \; ( \|f\|_{{n\over 2}+1}+|g|_{L_2})
\le c_3\, \|f\|_{{n\over 2}+1}.
\end{split}}
\end{equation*}
Recall that $\L_0=\sum_{i=1}^{m'} L_{X_i}L_{X_i}+L_{X_0}$.
 Let $\{W_t^j, j=1, \dots, m'\}$ be independent one dimensional Brownian motions. 
 Let $(z_t)$ be solutions of $dz_t=\sum_{j=1}^{m'} X_j(z_t)\circ dW_t^j$.
Since $f$ is $C^2$,
$${1\over t-s} \int_s^t f(z_r)dr
={1\over t-s} \left(g(z_t)-g(z_s)\right)
-{1\over t-s}\left(\sum_{j=1}^{m'}  \int_s^t (dg(X_j))(z_r) dW_r^j\right).$$
We apply the Sobolev estimates to $g$ and use Doob's $L^2$ inequality
to see that  for $t\ge1$ there is a constant $C$ such that,
$${ \begin{split}
\E \left({1\over t-s} \int_s^t f(z_r)dr\right)^2
&\le {4 \over t^2} |g|_\infty^2
+{8\over (t-s)^2} \sum_{j=1}^{m'} \int_s^t \left( \E|dg(z_r)|^2|X_j(z_r)|^2 \right) dr\\
&  \le {4 \over (t-s)^2} (|g|_\infty)^2
+{8m'\over t-s} (|dg|)^2_{\infty} \sum_{j=1}^{m'} |X_j|_\infty^2
\le C(\|f\|_{{n\over 2}+1})^2{1\over t-s}.
\end{split} } $$ 
  \end{proof}

We remark  that a self-adjoint operator satisfying H\"ormander's condition has index zero.
\begin{lemma}
\label{weak-convergence}
Suppose that $\L_0$ satisfies H\"ormander's condition. In addition it has Fredholm index $0$ or it has a unique invariant probability measure.
Let $r\ge \max{\{3, {n\over 2}+1\}}$.
 Let  $h : M \times G\to \R$ be  such that $h(y, \cdot)\in C^r$ for each $y$ and that  $|h|_\infty+ \sup_{z} |h(\cdot, z)|_{\Lip}
 +\sup_{y} |h(y, \cdot)|_{C^r}<\infty$. Let $s\le t$ be a pair of positive numbers, and 
 $F\in BC( C([0,s]; M)\to \R)$.
For any equi -uniformly continuous subsequence, $\tilde y^n_t:=(y^{\epsilon_n}_{t\over {\epsilon_n}})$, 
 of $(y^\epsilon_{t\over \epsilon})$ that converges weakly to a continuous process  $\bar y_\cdot$  as $n\to \infty$, the following convergence holds weakly:
  $$F(y^{\epsilon_n}_{\cdot\over \epsilon_n})
   \int_s^t  h(y^{\epsilon_n}_{u\over \epsilon_n}, z^{\epsilon_n}_{u\over \epsilon_n}) du
   \to F( \bar y_\cdot) \int_s^t  \overline{ h(\bar y_u, \cdot) }du$$
where $\overline{ h(y, \cdot)}=\sum_{i=1}^{n_0} u_i \< h(y, \cdot ),\pi_i\>$.
\end{lemma}
 \begin{proof}
 For simplicity we omit the subscript $n$.
 The required convergence follows from  Lemma 4.3 in \cite{Li-geodesic}
where it was assumed that (\ref{lln1}) holds and $\L_0$ has a unique invariant measure for $\mu$.
 It is easy to check that the proof there is valid. We take care to replace $\int_G h(y,z) d\mu(z)$ in Lemma 4.3 there
  by $\sum_{i=1}^{n_0} u_i \< h(y, \cdot ),\pi_i\>$.
We remark that by the regularity improving property each $u_i$ is smooth and therefore bounded.
In the first part of the proof, we divide $[s,t]$ into sub-intervals of size $\epsilon$, 
freeze the slow variable $(y^\epsilon_{u\over \epsilon})$ on $[t_k, t_{k+1}]$,
and approximate $h(y_{u\over\epsilon}^\epsilon, z_{u\over\epsilon}^\epsilon)$ by 
 $h(y_{t_k\over\epsilon}^\epsilon, z_{u\over\epsilon}^\epsilon)$ on each sub-interval $[t_k, t_{k+1}]$.
This approximation is clear: the computation is exactly as in Lemma 4.3 of \cite{Li-geodesic}
and we use the  uniform continuity of $(y_t^\epsilon)$, the fact that $|h|_\infty$ and $\sup_z |h( \cdot, z)|_{\Lip}$ are finite.
The convergence of
$$\int_{t_{k-1}\over \epsilon}^{t_{k-1}\over \epsilon} h(y_{t_k\over\epsilon}^\epsilon, z_{u\over\epsilon}^\epsilon)du
\to \Delta t_k \sum_{i=1}^{n_0}  u_i \< h(y_{t_{k-1}\over\epsilon}^\epsilon, \cdot), \pi_i\> $$
follows from the law of large numbers in Lemma \ref{lln}. The convergence of
$$ \sum_k\Delta t_k \sum_{i=1}^{n_0}  u_i \< h(y_{t_{k-1}\over\epsilon}^\epsilon, \cdot), \pi_i\> \to
\sum_{i=1}^{n_0} u_i  \int_s^t \< h(y_{u\over\epsilon}^\epsilon, \cdot), \pi_i\>du$$
is also clear and follows from the Lipschitz continuity of $h$ in the first variable and the equi continuity of the $y^\epsilon$ path.
Finally denote by $y^\epsilon_{[0,s]}$ the restriction of the path $y^\epsilon_\cdot$ to the interval $[0,s]$, 
the weak convergence of $ \sum_{i=1}^{n_0} u_i F(y^\epsilon_{[0,s]}) \int_s^t  \< h(y_{u\over\epsilon}^\epsilon, \cdot), \pi_i\>du$
to the required limit  is trivial, as explained in   Lemma 4.3, \cite{Li-geodesic}.
 \end{proof}
 
 \begin{assumption}
\label{assumption-Hormander}
The generator $\L_0$ satisfies H\"ormander's condition and has Fredholm index $0$ (or has a unique invariant probability measure). For $k=1, \dots, m$,  $\alpha_k\in C^r( G; \R)\cap N^\perp$ for some $r\ge \max\{3,{n\over 2}+1\}$.
\end{assumption}
If $\L_0$ is elliptic, it is sufficient to assume  $\alpha_k\in \B_b(G;\R)$, instead of $\alpha_k\in C^r$.

\begin{theorem}
\label{thm-weak}
If $\L_0$, $\alpha_k$,
$(y_0^\epsilon)$ and $|Y_j|$ satisfy the conditions of Proposition \ref{tightness}
and  Assumption \ref{assumption-Hormander}, then 
$(y_{t\over\epsilon}^\epsilon)$ converge weakly to the Markov process determined by the Markov
generator 
$$\bar \L =-\sum_{i,j=1}^m \overline{ \alpha_i \beta_j }
L_{Y_i}L_{Y_j}, \quad  \overline{ \alpha_i \beta_j }=\sum_{b=1}^{n_0} u_b \< \alpha_i \beta_j ,\pi_b\>.$$
\end{theorem}
\begin{proof}
By  Proposition \ref{tightness}, $\{(y_{t\over\epsilon}^\epsilon, t\ge 0)\}$  is tight.
We prove that any convergent sub-sequence 
 converges to the same limit.
 Let $\epsilon_n\to 0$ be a a monotone sequence converging to zero such that the probability distributions of
 $(y_{t\over \epsilon_n}^{\epsilon_n})$ converge weakly, on $[0,T]$, to a measure $\bar \mu$.
For notational simplicity we may assume that $\{(y_{t\over\epsilon}^\epsilon, t\ge 0)\}$ converges to $\bar\mu$.

Let $s<t$, $\{\B_s\}$  the canonical filtration, $(Y_s)$ the canonical process, and $Y_{[0,s]}$  its restriction  to $[0,s]$. 
By the Stroock-Varadhan martingale method, it is sufficient
to prove    
$f(Y_t)-f(Y_s)-\int_s^t \bar \L f(Y_r)\d r$
is a local martingale for any $f\in C_K^\infty(M)$.
By (\ref{Ito-tight}), the following is a local martingale,
\begin{equation*}
{\begin{split} 
&f(y^\epsilon_{t\over \epsilon}) -f(y^\epsilon_{s\over \epsilon}) 
- \epsilon \sum_{j=1}^m \left( df(Y_j(y^\epsilon_{t\over \epsilon} ) )\beta_j(z^\epsilon_{t\over \epsilon})
 +df(Y_j(y^\epsilon_{s\over \epsilon} ))\beta_j( z^\epsilon_{s\over \epsilon})\right)\\
&+\epsilon \sum_{i,j=1}^m\int_{s\over \epsilon}^{t\over \epsilon} L_{Y_i}L_{Y_j} f(y^\epsilon_r))
\alpha_i(z^\epsilon_r)\;\beta_j(z^\epsilon_r) \;dr.
\end{split}}\end{equation*}
 Since the third term converges to zero as $\epsilon $ tends to zero,  it is sufficient to prove
$$\lim_{\epsilon 
\to 0} \E\left\{\epsilon \sum_{i,j=1}^m\int_{s\over \epsilon}^{t\over \epsilon} L_{Y_i}L_{Y_j} f(y^\epsilon_r))
\alpha_i(z^\epsilon_r)\;\beta_j(z^\epsilon_r) \;dr-\int_s^t \bar \L f(y_{r\over \epsilon}^\epsilon)\d r \; \big| \F_{s\over \epsilon}\right\}=0.$$
This follows from Lemma \ref{weak-convergence},  completing the proof.\end{proof}

\begin{corollary}
\label{convergence-wasserstein-p}
Let $p\ge 1$ be a number and suppose that $\rho^p\in B_{V, 0}$. Then,
under the conditions of Theorem \ref{thm-weak} and Assumption \ref{assumption2-Y},  $(y_{\cdot\over \epsilon}^\epsilon)$ converges
  in the Wasserstein $p$-distance on $C([0,t];M)$. 
\end{corollary}
\begin{proof}
By Theorem \ref{uniform-estimates}, $\sup_{\epsilon\le \epsilon_0} \E\sup_{s\le t} \rho^p(o, y_{s\over \epsilon}^\epsilon)<\infty$. Let $W_p$ denote the Wasserstein $p$ distance:
$$W_p(\mu_1, \mu_2)=\left( \inf \int_{M\times M} \sup_{s\le t} \rho(\sigma_1(s), \sigma_2(s))  d\mu(\sigma_1, \sigma_2)\right)^{1\over p}.$$
Here the infimum is taken over all probability measures on the path spaces $C([0,t];M)$ with marginals $\mu_1$ and $\mu_2$.
Note that $C([0,t];M)$ is a Banach space, a family of probability measures $\mu_n$ converges to $\mu$ in $W_p$,   if and only if the following holds: (1) it converges weakly and (2)  $\sup_n \int  \sup_{s\le t} \rho^p(o, \sigma_2(s))  d\mu_n( \sigma_2)<\infty$. The conclusion follows.
\end{proof}

\section{A study of the semigroups}
\label{section-sde}
The primary aim of the section is to study the properties of $P_t f$ for $f\in B_{V,r}$ where $P_t$ is the semigroup for a generic stochastic differential equation.  These results will be applied to the limit equation, to provide the necessary 
a priori estimates. Theorem \ref{limit-thm} should be of independent interest, it also lead to Lemma \ref{semigroup-estimate-2}, which will be used in Section \ref{section-rate}.

Throughout this section $M$ is a complete smooth Riemannian manifold. 
 Let  $Y_0$ be $C^5$ and $\{Y_k, k=1,\dots, m\}$ be $C^6$ smooth vector fields on $M$, $\{B_t^k\}$ independent real valued
 Brownian motions.  Let $(\Phi_t(y), t<\zeta(y)) $ be the maximal solution to the following equation
 \begin{equation}
\label{limit.sde}
dy_t=\sum_{k=1}^m Y_k(y_t)\circ dB_t^k+Y_0(y_t)dt
\end{equation}
 with initial value $y$. Its Markov generator is
 $\L f={1\over 2} \sum_{k=1}^m L_{Y_k}L_{Y_k}f+L_{Y_0}f$.
     Let $Z={1\over 2}\sum_{k=1}^m \nabla _{Y_k}Y_k+Y_0$ be the drift vector field, so
\begin{equation}\label{generator-0}
\L f={1\over 2}\sum_{k=1}^m \nabla df(Y_k,Y_k)+df(Z).
\end{equation}
 If there exists a $C^3$ pre-Lyapunov function $V$, constants $c$ and  $K$  such that $\L V\le c+KV$ 
then (\ref{limit.sde}) is complete. However we do not limit ourselves to Lyapunov test for the completeness of the SDE.
Let us denote $|f|_r=\sum_{k=1}^{r}|\nabla^{(k-1)} df|$ and $|f|_{r,\infty}=\sum_{k=1}^{r}|\nabla^{(k-1)} df|_\infty$. The following observation is useful.
 \begin{lemma}
 \label{Lf-L2f}
   Let $V\in \B(M;\R)$ be locally bounded.
   \begin{itemize}
   \item  Suppose that $\sum_{j=1}^m |Y_j| \in B_{V,0}$ and $ |Z|\in B_{V,0}$. Then if $f\in B_{V, 2}$,  
   $\L f\in B_{V,0}$. If $f \in BC^2$, $|\L f|\le |f|_{2, \infty} F_1$ where $F_1\in B_{V,0}$, not depending on $f$.
 \item  Suppose that 
   $$\sum_{j=1}^m( |Y_j| +|\nabla Y_j|+|\nabla^{(2)} Y_j|) \in B_{V,0}, \quad  |Z|+|\nabla Z|+|\nabla^{(2)} Z|\in B_{V,0}.$$
   If $f\in B_{V,4}$, $\L^2 f\in B_{V,0}$. If $f \in BC^4$, $|\L^2 f|\le |f|_{4,\infty} F_2$ where $F_2$ is a function in  $B_{V,0}$,
   not dependent of $f$.
   \end{itemize} 
   \end{lemma}
    \begin{proof}
That $\L f$ belongs to $B_{V,0}$ follows from (\ref{generator-0}). If $f\in BC^2$, $|\L f|\le (|f|_2)_\infty (\sum_{k=1}^m|Y_k|^2+|Z|)$.
 For the second part we observe that  $\L^2 f$ involves  at most four derivatives of $f$ and two derivatives of $Y_j$ and $Z$
 where $j=1, \dots, m$.
    \end{proof}

    Let  $d\Phi_t(v)$ denote the derivative flow  in the direction of $v\in T_yM$. It is the derivative of the function $y\mapsto \Phi_t(y, \omega )$,
    in probability. Moreover, it solves the following stochastic covariant differential equation along the solutions $y_t:=\Phi_t(y_0)$,
$$Dv_t=\sum_{k=1}^m \nabla_{v_t} Y_k\circ dB_t^k+\nabla_{v_t} Y_0dt.$$
Here $D V_t:=\paral_t(y_\cdot)   d (\parals_t^{-1}(y_\cdot) V_t)$ where $\parals_t(y_\cdot) :T_{y_0}M\to T_{y_t}M$ is
 the stochastic parallel transport map along the path $y_\cdot$.
Denote $|d\Phi_t|_{y_0}$ the norm of $d\Phi_t(y_0): T_{y_0}M\to T_{y_t}M$. 
For $p>0$, $y\in M$ and  $v\in T_yM$,   we define $H_p(y)\in {\mathbb L}(T_yM\times T_yM; \R)$ by
 \begin{equation*}
 H_p(y)(v,v)= \sum_{k=1}^{m} |\nabla Y_k(v)|^2
+(p-2)\sum_{k=1}^{m} {\<\nabla Y_k(v), v\>^2\over |v|^2} + 2\<\nabla Z(v), v\>.
\end{equation*}
Let $\underline h_p(y)=\sup_{ |v|=1\}} H_p(y)(v,v)$. Its upper bound will be used to control
$|d\Phi_t|_y$.

 \begin{assumption}
  \label{Y-condition}
 The equation (\ref{limit.sde}) is complete.  Conditions (i) and (ii), or (i') and (ii), below hold.
 \begin{itemize}
 \item[(i)]  There exists a  locally bounded function $V\in \B(M;\R_+)$, s.t. for all $q\ge 1$ and $t\le T$, there exists a number $C_q(t)$ 
 and a polynomial $\lambda_q$ such that
\begin{equation}
\label{moment-assumption-2}
\sup_{s\le t}\E(|V(\Phi_s(y))|^q)\le C_q(t)+C_q(t) \lambda_q(  V(y)).
\end{equation}
\item[(i')] There exists $V\in C^3(M; \R_+)$ and constants $c$ and  $K$ such that 
$$\L V\le c+KV, \quad |L_{Y_j}V|\le c+KV, \quad 
j= 1, \dots, m,$$

\item [(ii)] Let $\tilde V=1+ \ln(1+|V|)$. For some constant $c$,
\begin{equation}\label{5.2}
 \sum_{k=1}^{m} |\nabla Y_k|^2\le c\tilde V, \quad  \sup_{|v|=1}{\<\nabla Z(v),v\>} \le c\tilde V.
\end{equation}
 \end{itemize}
  \end{assumption}
  
{\it Remark.} 
Suppose that (\ref{limit.sde}) is complete.  Since $\L V^q =qV^{q-1} \L V+q(q-1)V^{q-2} |L_{Y_j}V|^2$,  (i')  implies (i). 
In fact,
 $ \E \sup_{s\le t} \left(V(y_s)\right)^q\le \left( \E V(y_0)^q+cq^2t \right) e^{(c+K)q^2t}$.
 
 Recall that (\ref{limit.sde}) is strongly complete if $(t,y)\mapsto \Phi_t(y)$ is continuous almost surely
 on $[0, t]\times M$ for ant $t>0$. %  Definition\ref{definition-complete}.
\begin{theorem}
\label{limit-thm}
Under Assumption \ref{Y-condition}, the following statements hold.
 \begin{enumerate}
 \item
The  SDE (\ref{limit.sde}) is strongly complete and   for every $t\le T$,  $ \Phi_t(\cdot)$ is  $C^4$. Furthermore  for all $p\ge 1$, there exists a positive number $C(t,p)$ such that
 \begin{equation}
\label{derivative}
 \E\left(\sup_{s\le t}|d\Phi_s(y)|^p\right)\le C(t,p)+C(t,p)V^{C(t,p)}(y) .
\end{equation}

 \item Let $f \in B_{V,1}$. Define  $\delta P_t (df))=\E df(d\Phi_t(\cdot))$.
 Then $d(P_tf)=\delta P_t(df)$ and $|d(P_tf)|\in B_{V,0}$.
 Furthermore for a constant $C(t,p)$ independent of $f$,
$$ |d(P_tf)|\le \sqrt{ \E \left(|df|_{\Phi_t^\epsilon(y)}\right)^2} 
 \sqrt{ C(t,p)(1+ V ^{C(t,p)}(y))}.$$

\item Suppose furthermore that
$$\sum_{j=1}^m\sum_{\alpha=0}^3|\nabla^{(\alpha)}Y_j|\in B_{V,0}, \qquad 
\sum_{\alpha=0}^2|\nabla^{(\alpha)} Y_0| \in B_{V, 0}.$$
 Then, (a)  $\E\sup_{s\le t}|\nabla d\Phi_s|^2(y)\in B_{V,0}$;
 (b)  If $f\in B_{V, 2}$, then
  $P_tf\in B_{V,2}$, and
  $$(\nabla dP_tf)(u_1,u_2)=\E \nabla df (d\Phi_t(u_1), d\Phi_t(u_2))+\E df (\nabla _{u_1} d\Phi_t(u_2)).$$
Furthermore,  (c)
   ${d P_t f\over dt }=P_t \L f$, and $\L (P_t f)=P_t (\L f)$.

 \item 
 Let $r\ge 2$. Suppose  furthermore  that
 $$ {\begin{split} \sum_{\alpha=0}^{r}|\nabla ^{(\alpha)}Y_0| \in B_{V, 0}, \quad
 \sum_{\alpha=0}^{r+1}   \sum_{k=1}^m|\nabla^{(\alpha)} Y_k|\in B_{ V,0}.
  \end{split}}$$
 Then $\E\sup_{s\le t}(|\nabla^{(r-1)} d\Phi_s|_y)^2$ belongs to $B_{V,0}$.  If $f\in B_{V, r}$, then
  $P_tf\in B_{V,r}$. 
\end{enumerate}

\end{theorem}

\begin{proof}
The statement on strong completeness  follows from the following theorem, see Thm. 5.1 in \cite{Li-flow}.
Suppose that (\ref{limit.sde}) is complete. If $\tilde V$ is a function and $c_0$ a number such that 
for all $t>0$, $K$ compact, and all constants $\lambda$,
\begin{equation}
\label{assumption-strong-compl}
\sup_{y\in K}\E \exp{\left( \lambda \int_0^t \tilde V(\Phi_s(y))ds\right)}<\infty, \quad \sum_{k=1}^{m} |\nabla Y_k|^2\le c_0\tilde V, \quad
 \underline h_p\le  6pc_0 \tilde V,
\end{equation}
then (\ref{limit.sde}) is strongly complete. Furthermore for every $p\ge 1$  there exists a constant $c(p)$ such that 
\begin{equation}
\label{derivative}
 \E\left(\sup_{s\le t}|d\Phi_s(y)|^p\right)\le c(p)\E\left(\exp{\left(6p^2 \int_0^{t}\tilde V(\Phi_s(y)) ds\right)}\right).
\end{equation}
 Since $Y_j$ are $C^6$,  then
 for every $t$,   $ \Phi_t(\cdot)$ is  $C^4$. 
It is easy to verify that condition (\ref{assumption-strong-compl}) is  satisfied. In fact, by the assumption $\underline h_p\le 6p c \tilde V$. 
Take $\tilde V=1+ \ln(1+|V|)$ then for  $p\ge 1$,
  $${\begin{split}
  \E\left(\exp{\left(6p^2 \int_0^{t}\tilde V(\Phi_s(y)) ds\right)}\right)\le C(t,p) +C(t,p)  \left( V ^{C(t,p)}(y)\right)<\infty.
  \end{split}}
$$
This proves part (1).

For part (2) let $f\in C^1$. Then $y\mapsto f(\Phi_t(y,\omega))$ is differentiable for almost every $\omega$. Let $\sigma: [0,t_0]\to M$ be a geodesic segment with $\sigma(0)=y$. Then
$${ f(\Phi_t(\sigma_s, \omega))-f(\Phi_t(y, \omega))\over s}
={1\over s}\int_0^s {d\over dr} f\left(\Phi_t(\sigma_r, \omega)\right)dr.$$
Since $\E|d\Phi_t(y)|^2$ is locally bounded in $y$,
$r\mapsto \E|d\Phi_t(\sigma_r, \omega)|$ is continuous  and the expectation of the right hand side
converges to $\E df(d\Phi_t(\dot \sigma(0))$.
The left hand side clearly converges almost surely. Since $\E |df (d\Phi_t(y))|^2$ is locally bounded the convergence is in $L^1$.
We proved that $d(P_tf)=\delta P_t(df)$. Furthermore, suppose that $|df|\le K+K V^q$,
\begin{equation*}
{\begin{split}
|d(P_tf)|_y&\le \sqrt{ \E \left(|df|_{\Phi_t^\epsilon(y)}\right)^2} \sqrt{\E |d\Phi_t^\epsilon|_y^2} \\
&\le   \sqrt{  2K^2+2K^2 \E V^{2q}(\Phi_t^\epsilon(y))} 
 \sqrt{ c(p) C(t,p) +c(p) C(t,p) \left( V ^{C(t,p)}(y)\right)}.
\end{split}}
\end{equation*}
The latter, as a function of $y$,  belongs to $B_{V,0}$.

We proceed to part (3a).   Let $v,w \in T_yM$ and $U_t:=\nabla d\Phi_t(w, v)$. Then $U_t$ satisfies the following equation:
$${\begin{split}
DU_t=&\sum_{k=1}^m\nabla^{(2)} Y_k (d\Phi_t(v), d\Phi_t(w))\circ dB_t^k+\sum_{k=1}^m\nabla Y_k(U_t)\circ dB_t^k\\
&+\nabla^{(2)} Y_0 (d\Phi_t(v), d\Phi_t(w))dt+\nabla Y_0(U_t)dt.
\end{split}}$$
It follows that,
$${\begin{split}
{d} |U_t|^2=& 2\sum_{k=1}^m
\left< \nabla^{(2)} Y_k (d\Phi_t(v), d\Phi_t(w))\circ dB_t^k+\nabla^{(2)} Y_0 (d\Phi_t(v), d\Phi_t(w))dt, U_t\right\>\\
& +\left\<\sum_{k=1}^m\nabla Y_k(U_t)\circ dB_t^k+\nabla Y_0(U_t)dt, U_t\right\>.
\end{split}}$$
To the first term on the right hand side we apply Cauchy Schwartz inequality to split the first term in the inner product and the second term in the inner product. This gives: $C|U_t|^2$ and other terms that does not involve $U_t$.
 The Stratonovich corrections will throw out the extra derivative $\nabla^{(3)} Y_k$ which does not involve $U_t$.
The second  term on the right hand side is a sum of the form $\sum_{k=1}^m \< \nabla Y_k(U_t), U_t\>dB_t^k$
for which only bound on $|\nabla Y_k|$ is required, and 
$$\left \<  \sum_{k=1}^m \nabla^{(2)} Y_k(Y_k, U_t)+\nabla Y_0(U_t), U_t\right\>
=\< \nabla Z(U_t), U_t\>  -\left \< \sum_{k=1}^m \nabla Y_k(\nabla _{U_t} Y_k), U_t\right\>.$$
The second term is bounded by
$$\left| \sum_{k=1}^m\left \< \nabla Y_k(\nabla _{U_t} Y_k), U_t\right\>\right|
\le \sum_{k=1}^m  |\nabla Y_k|^2 |U_t|^2.$$
By the assumption, there exist $c>0,q\ge 1$ such that,
for every  $k=1, \dots, m$, 
 $$|\nabla Y_k| \le \tilde V , |\nabla ^2Y_j|\le c+cV^q, |\nabla^{(3)} Y_k|\le c+cV^q, \<\nabla _u Z, u\>\le (c+KV)|u|^2.$$
 There is a stochastic process $I_s$, which does not involve $U_t$, and constants $C,q$ such that
 $$\E|U_t|^2 \le \E|U_0|^2+\int_0^t \E I_r dr+\int_0^t C\E\tilde V^q(y_r^\epsilon)|U_r|^2dr.$$
 By induction $I_r$ has moments of all order which are bounded on compact intervals. 
 By Gronwall's inequality, for $t\le T$,
 $$\E|U_t|^2\le \left( \E|U_0|^2+\int_0^T \E I_r dr\right)\exp{\left(C \int_0^t \tilde V^q(y_r^\epsilon)dr\right)}.$$
To obtain the supremum inside the expectation, we simply use Doob's $L^p$ inequality  before taking expectations.
With the argument in the proof of part (1) we conclude that
$\E\sup_{s\le t}  |\nabla d\Phi_s|^2(y)$ is finite and belongs to $B_{V,0}$. 

{\it Part (3b).}
Let $f\in B_{V, 2}$.
By part (1),  $d(P_tf)=\E df(d\Phi_t(y))$. Let  $u_1, u_2\in T_yM$. 
By an argument analogous to part (3), we may differentiate the right hand side under the expectation to obtain that
$$(\nabla dP_tf)(u_1,u_2)=\E \nabla df (d\Phi_t(u_1), d\Phi_t(u_2))+\E df (\nabla _{u_1} d\Phi_t(u_2)).$$
 Hence $P_tf\in B_{V,2}$.
This procedure can be iterated. 

{\it Part (3c).}  By It\^o's formula,
$$f(y_t)=f(y_s)+\sum_{k=1}^m\int_s^t df(Y_k(y_r))dB_r^k+\int_s^t \L f(y_r)dr.$$
Since $df(Y_k)\in B_{V,0}$, the expectations of the stochastic integrals with respect to the Brownian motions vanish. 
Since $ \L f\in B_{V,0}$ by part (3), $\L f(y_r)$ is bounded in $L^2$.
 It follows that the function $r\mapsto  \E\L f (y_r)$ is  continuous,
$$\lim_{t\to s} {\E f(y_t)-\E f(y_s)\over t-s}=\E \L f(y_s)$$
and we obtain Kolmogorov's backward equation,
${\partial \over \partial s}P_sf =P_s (\L f)$.
 Since $P_sf \in B_{V,2}$, we apply the above argument to
$P_sf$, and take $t$ to zero in ${P_t(P_sf)-P_sf\over t}$ and obtain that
${\partial \over \partial s} P_sf = \L (P_sf)$.
This  leads to the required statement
$\L P_sf =P_s \L f$.

{\it Part (4).} For higher order derivatives of $\Phi_t$ we simply iterate the above procedure and note
that the linear terms in the equation for ${d\over dt} |\nabla^{k-1}d\Phi_t(u_1, \dots, u_k)|^2$
are always of the same form. \end{proof}

\begin{remark}
With the assumption of part (3), we can show that for all integer $p$, $\E\sup_{s\le t}|\nabla d\Phi_s|_{y}^p\in B_{V,0}$.
\end{remark}
If we assume the additional conditions that 
$$ |\nabla Y_0| \le c\tilde V, \quad
 \sum_{k=1}^m |\nabla^{(2)} Y_k||Y_k|\le c\tilde V,$$
the conclusion of the remark follows more easily. With the assumptions of part (5) we need to work a bit more which we illustrate below.
Let $U_t=\nabla d\Phi_t(w,v)$. Instead of writing down all term in  $|U_t|^p$ we classify the terms in $|U_t|^p$ into two classes: 
those involving $U_t$ and those  not.
For the first class we must assume that they are bounded by $c\tilde V$ for some $c$. For the second class we
may use induction and hence it is sufficient to assume that they 
belong to $B_{V,0}$.
The terms that involving $U_t$ are:
$$\nabla Y_k(U_t), \quad   \sum_{k=1}^m\nabla^{(2)} Y_k(Y_k, U_t)+\nabla Y_0(U_t).$$

The essential identity to use is:
$$\sum_{k=1}^m\nabla^{(2)} Y_k(Y_k, U_t)+\nabla Y_0(U_t)=\nabla Z(U_t)-\sum_{k=1}^m \nabla Y_k(\nabla Y_k (U_t)).$$
We do not need to assume that the second order derivatives $|\nabla^{(2)} Y_k||Y_k|\le c\tilde V$,  
it is sufficient to assume that for $|\nabla Y_k|^2$ and $\nabla Z$ for all $k=1,\dots, m$.
With a bit of care, we check that only one sided derivatives of $Z$ are involved.

For example we can convert it to the $p=2$ case,
$$d|U_t|^p={p\over 2} (|U_t|^{p-2}) \circ d|U_t|^2={p\over 2} |U_t|^{p-2} d|U_t|^p + {1 \over 4} p(p-1)|U_t|^{p-4} \< d|U_t|^2\>.$$
By the first term ${p\over 2} |U_t|^{p-2} d|U_t|^p $ we mean that in place of $d|U_t|^p$
plug in all terms on the right hand side of the equation for $d|U_t|^2$, after formally converting the integrals to It\^o form.
By $ \< d|U_t|^2\>$ we mean the bracket of the martingale term on the right hand side of $d|U_t|^2$. It is now easy to check that
in all the terms that involving $U_t$, higher order derivatives of $Y_k$ does not appear, except in the form of $|U_t|^{p-2}\<\nabla_{U_t} Z, U_t\>$.

\begin{remark}
Assume the SDE is complete. Suppose that for some positive number $C$,
$$\sum_{k=1}^m \sum_{k=0}^{5} |\nabla ^{(k)}Y_k|\le C, \quad  \sum_{k'=0}^{4} |\nabla^{(k')} Y_0|\le C.$$
Then for all $p\ge 1$, there exists a constant  $C(t,p)$ such that $$ \E\left(\sup_{s\le t}|d\Phi_s(x)|^p\right)\le C(t,p).$$
Furthermore the statements in Theorem \ref{limit-thm} hold for $r\le 4$.\end{remark}

Recall that $|f|_r=\sum_{k=1}^{r} |\nabla^{(k-1)}  d f|$ and $|f|_{r,\infty}=\sum_{k=1}^{r} |\nabla^{(k-1)}df|_\infty$.

\begin{lemma}
\label{semigroup-estimate-2}
Assume Assumption \ref{Y-condition} and
$$  \sum_{\alpha=0}^{4}|\nabla ^{(\alpha)}Y_0| \in B_{V, 0}, \quad
 \sum_{\alpha=0}^{5}   \sum_{k=1}^m|\nabla^{(\alpha)} Y_k|\in B_{ V,0}.$$
Then there exist constants $q_1, q_2\ge 1$, $c_1$ and $c_2$ depending on $t$ and $f$ and locally bounded in $t$,
also functions $\gamma_i \in B_{V,0}$, $\lambda_{q_i}$ polynomials, such that for $s\le t$,
 \begin{equation*}
 {\begin{split}
 |P_tf(y_0)-P_sf (y_0)|
 & \le
 (t-s) c_1\left( 1+\lambda_{q_1}(V(y_0))\right), \quad f\in B_{V,2}\\
 \left|P_tf(y_0)-P_sf(y_0)-(t-s) P_s( \L f) (y_0)\right| 
 &\le  (t-s)^2 c_2 \left( 1+\lambda_{q_2}(V(y_0))\right), \quad f\in B_{V,4}\\
 |P_tf(y_0)-P_sf (y_0)| &\le
 (t-s) \left(1+|f|_{2, \infty}\right)\gamma_1(y_0), \quad \forall f \in BC^2 \\
 \left|P_tf(y_0)-P_sf(y_0)-(t-s) P_s( \L f) (y_0)\right| &\le
 (t-s)^2\left(1+|f|_{4, \infty}\right)\gamma_2(y_0), \quad \forall f \in BC^4.
 \end{split}}
 \end{equation*}
\end{lemma}

\begin{proof}
Denote  $y_t=\Phi_t(y_0)$, the solution to (\ref{limit.sde}). Then for $f\in C^2$,
$$P_tf(y_0)=P_sf (y_0)+\int_s^t P_r (\L f)(y_0 )dr+\sum_{k=1}^{m} \E\left( \int_s^t df(Y_k(y_r))
  dB_r^k\right).$$
 Since $|L_{Y_k} f| \le |df|_\infty |Y_k|$ and $|df|$, $Y_k$ belong to $ B_{V,0}$, by Assumption \ref{Y-condition}(i), 
 $ \int_0^t  \E |L_{Y_k}f|_{y_r}^2 dr $ is finite and the last term vanishes.  Hence 
 $ |P_tf(y_0)-P_sf (y_0)|\le \int_s^t P_{s_2} ( \L f)(y_0)ds_2$. 
By Lemma \ref{Lf-L2f}, $\L f\in B_{V,0}$ if $f\in B_{V,2}$. Let $K, q_1$ be s.t. $|\L f|\le K+KV^{q_1}$.
  $$ \int_s^r |P_{s_2} ( \L f)(y_0)|ds_2 \le \int_0^r \left(K+K\E V^{q_1}(\Phi_{s_2}(y_0)) \right)ds_2.$$
By the assumption, we see easily that 
$\sum_{k=0}^3 |\nabla^{(\alpha)} Z|\in B_{V,0}$. 
   By Assumption \ref{Y-condition}, $\sup_{s\le t}\E(|V(\Phi_s(y_0))|^{q_1})\le C_{q_1}(t)+C_{q_1}(t) \lambda_{q_1}(  V(y_0))$ 
  and the first conclusion holds.
We repeat this procedure for $f\in C^4$ to obtain:
  $${\begin{split}
&P_tf(y_0)-P_sf (y_0)\\
&=\int_s^t  \left( P_s  (\L f)(y_0 ) +\int_s^r P_{s_2} ( \L^2 f)(y_0)ds_2
+\sum_{k=1}^{m} \E \int_s^t  \left( L_{Y_k} (\L f) \right)(y_{s_2}))
  dB_{s_2}^k \right)ds_1.\end{split}}$$
 The last term also vanishes, as every term in $L_{Y_k} \L f $ belongs to $B_{V,0}$. Indeed
   $${\begin{split} L_{Y_k} \L f &=\sum_i   \nabla^{(2)}  df(Y_k, Y_i, Y_i)+
2\sum_i \nabla df\left( \nabla_{Y_k}Y_i, Y_i\right)+ \nabla df \left(Y_k, Z\right)\\
&+\sum_i df(\nabla^{(2)}Y_i(Y_k,Y_i)
+\nabla Y_i\left(\nabla_{Y_k}Y_i+\nabla _{Y_k}Y_0)\right). \end{split}}
$$

This gives, for all $f \in B_{V,4}$, 
 \begin{equation}\label{type1}
\left|P_tf(y_0)-P_sf(y_0)-(t-s) P_s( \L f) (y_0)\right|
\le \left| \int_s^t \int_s^{s_1} P_{s_2} (\L^2 f)(y_0)ds_2 ds_1
\right|.
\end{equation}
Let $q_2, K$ be numbers such that $|\L^2 f| \le K+K V^{q_2}$. Then,
 $$
 \sup_{s\le t} P_{s} (\L^2 f)(y_0)\le K+K \E\left( V(y_s)\right)^{q_2} \le K+C_{q_2}(t)+KC_{q_2}(t) \tilde \lambda_{q_2}(V(y_0)). $$
 Consequently, there exist a constant $c_2(t)$ s.t.
 $$  \left|P_tf(y_0)-P_sf(y_0)-(t-s) P_s( \L f) (y_0)\right|
\le (t-s)^2 c_2(t, K, q_2) (1+  \lambda_{q_2}(V(y_0))). $$
completing the proof for $f\in B_{V,2}$ and $B_{V,4}$.
Next suppose that $f\in BC^2$. By Lemma \ref{Lf-L2f}, $| \L f|\le  |f|_{2, \infty} F_1$, and $|\L^2 f| \le |f|_{4, \infty} F_2$ if $f\in BC^4$.
Here $F_1, F_2\in B_{V,0}$ and do not depend on $f$. We iterate the argument above
 to complete the proof for $f\in BC^4$.
\end{proof}

\section{Rate of Convergence}
\label{section-rate}
If  $\L_0$ has a unique invariant probability measure $\pi$ and $f\in L^1(G, d\pi)$ denote $\bar f=\int_G fd\pi$.
Let $\bar \L=-\sum_{i,j=1}^m \overline{\alpha_i \beta_j} L_{Y_i}L_{Y_j} $.
Let $\{\sigma_k^i, i,k=1,\dots,  m\}$ be the entries in a square root of the matrix $(\overline{-\alpha_i \beta_j})$. 
They satisfy  $\sum_{k=1}^m \sigma_k^i\sigma_k^j=(\overline{-\alpha_i \beta_j})$ and are constants.  Let us consider the SDE:
\begin{equation}\label{limit.sde-2}
dy_t=\sum_{k=1}^{m} \left(  \sum_{i=1}^{m} \sigma_k^i Y_i(y_t) \right)\circ dB_t^k,
\end{equation}
where $\{B_t^k\}$ are independent one dimensional Brownian motions. Let
$$\tilde Y_k= \sum_{i=1}^{m} \sigma_k^i Y_i(y_t), \quad \tilde Z=\sum_{i,j=1}^m\overline{ -\alpha_i\beta_j}\nabla_{Y_i}Y_j.$$
The results from section \ref{section-sde} apply.
  Recall that $\L_0={1\over 2}\sum_{i=1}^p L_{X_i}L_{X_i}+L_{X_0}$ and $(z_t^\epsilon)$ are
$\L^\epsilon={1\over \epsilon} \L_0$ diffusions.
Let $\Phi^\epsilon_t(y)$ be the solution to the SDE (\ref{1}):
$\dot y_t^\epsilon =\sum_{k=1}^m \alpha_k(z_t^\epsilon)Y_k(y_t^\epsilon)$ with initial value $y$.

\begin{assumption}
\label{assumption-on-rate-result}
$G$ is compact, $Y_0\in C^5(\Gamma TM)$, and $Y_k\in C^6(\Gamma TM)$ for $k=1, \dots, m$.
Conditions (1)-(5) below hold or Conditions (1), (2') and (3-5) hold.
 \begin{enumerate}
\item [(1)]
The SDEs (\ref{limit.sde-2}) and (\ref{sde-3}) are complete.

\item [(2)] 
$V\in \B(M;\R_+)$ is a locally bounded function and  $\epsilon_0$ a positive number 
s.t. for all $q\ge 1$ and $T>0$,  there exists a locally bounded function $C_q: \R_+\to \R_+$, a 
real valued polynomial $\lambda_q$ such that for $0\le s\le t\le T$ and for all $\epsilon\le\epsilon_0$
\begin{equation}
\label{moment-assumption-2}
 \sup_{s\le u \le t}  \E\left\{ V^q(\Phi_{u\over \epsilon}^\epsilon(y)) \; \big| \F_{s\over \epsilon} \right\} 
 \le C_q(t)+C_q(t) \lambda_q\left(  V(\Phi_{s\over \epsilon}^\epsilon(y)\right).
\end{equation}

\item [(2')] 
There exists a function $V\in C^3(M; \R_+) $ s.t. for all $i,j\in \{1, \dots, m\}$,
$|L_{Y_i}L_{Y_j} V|\le c+KV$ and $|L_{Y_j} V| \le c+KV$.

\item[(3)]
For $V$ defined above,  let $\tilde V=1+ \ln(1+|V|)$. Suppose that 
 $$ {\begin{split} &\sum_{\alpha=0}^{4}|\nabla ^{(\alpha)}Y_0| \in B_{V, 0}, \quad
 \sum_{\alpha=0}^{5}   \sum_{k=1}^m|\nabla^{(\alpha)} Y_k|\in B_{ V,0}, \\
 &  \sum_{j=1}^m|\nabla Y_j| ^2\le c\tilde V, \quad
\sup_{|u|=1} \<\nabla \tilde Z(u), u\>\le c\tilde V   \end{split}}$$
\item [(4)]
  $\L_0$ satisfies H\"ormander's conditions and has a unique invariant measure $\pi$ satisfying Assumption \ref{assumption1}.
 \item [(5)] $\alpha_k\in C^3(G; \R)\cap N^\perp$. 
\end{enumerate}
\end{assumption}

We emphasize the following:
\begin{remark}
\label{rate-remark-1}
\begin{itemize}
\item [(a)] If  $V$ in (2') is a pre-Lyapunov function,
 then (\ref{sde-3}) is complete. Furthermore  $|\bar \L V|\le c+KV$ and so (\ref{limit.sde-2}) is complete. 

\item [(b)] Under conditions (1), (2') and (4-5), (2) holds. See  Theorem \ref{uniform-estimates}. Also Corollary \ref{corollary-to-lemma3-2} holds. Conditions (1-5) implies the conclusions of Theorem \ref{limit-thm}.
\item[(c)] If $\L_0$ satisfies strong H\"ormander's condition, condition (4) is satisfied.
\end{itemize}
\end{remark}

Let $P_t^\epsilon$ be the probability semigroup associated with $(y_{t}^\epsilon)$ and $P_t$ the  Markov semigroup for $\bar \L$. Recall that $|f|_{r,\infty}=\sum_{j=1}^{r} |\nabla^{(j-1)}df|_\infty$.  We recall that operator $\L_0$ on a compact manifold $G$ satisfying strong H\"ormander's condition has an exponential mixing rate,  so $\L_0$ satisfy Assumption  \ref{assumption1}.

\begin{theorem}
\label{rate}
Assume that $Y_k, \alpha_k$ and $ \L_0$ satisfy Assumption \ref{assumption-on-rate-result}.
%Let $y_0\in M$.
% Then for any $T>0$,  the stochastic processes $(\Phi^\epsilon_{t\over \epsilon}(y_0), 0\le t\le T)$ converge weakly, as $\epsilon\to 0$, to  the Markov process determined by $\bar \L$. 
 For every $f \in B_{V,4}$,
$$\left|\E f\left(\Phi^\epsilon_{T\over \epsilon}(y_0)\right) -P_Tf(y_0)\right|\le \epsilon |\log \epsilon|^{1\over 2}C(T)\gamma_1(y_0),$$
where  $\gamma_1\in B_{V,0}$ and $C(T)$ are constant increasing in $T$. Similarly,  if $f\in BC^4$,
 $$\left|\E f\left(\Phi^\epsilon_{T\over \epsilon}(y_0)\right) -P_Tf(y_0)\right|
 \le \epsilon  |\log \epsilon|^{1\over 2}\,C(T)\gamma_2(y_0) \left(1+ |f|_{4, \infty}\right). $$
where $\gamma_2$ is a function in $B_{V,0}$ that does not depend on $f$ and $C(T)$ are constants
increasing in $T$.   \end{theorem}

\begin{proof}
{\it Step 1.} To obtain optimal estimates we work on intervals  of  order $\epsilon$, c.f. Lemma \ref{lemma3}.
Let $t_0=0<t_1<\dots<t_N=T$ be a partition of $[0,T]$ with $\Delta t_k=t_k-t_{k-1}=\epsilon$ for $k<N$
and $t_1\le \epsilon$.  Write $y_t^\epsilon=\Phi_t^\epsilon(y_0)$. Then,
\begin{equation*}
{\begin{split}
& f\left(y^\epsilon_{T\over \epsilon}\right)-P_T f(y_0)
= \sum_{k=1}^N\left(   P_{T-t_k}f(y_{t_k\over \epsilon}^\epsilon)
-   P_{T-t_{k-1}}f(y_{t_{k-1}\over \epsilon}^\epsilon)  \right)\\
&= \sum_{k=1}^N \left(  P_{T-t_k}f(y_{t_k \over \epsilon}^\epsilon)
-   P_{T-t_{k}} f(y_{t_{k-1 \over \epsilon}}^\epsilon) 
+ \Delta t_k  \left( P_{T-t_{k-1}} \bar \L  f (y_{t_{k-1\over \epsilon}}^\epsilon)\right) \right)\\
&+ \sum_{k=1}^N  \left(  P_{T-t_k}f (y_{t_{k-1\over \epsilon}}^\epsilon)
-  P_{T-t_{k-1}}f(y_{t_{k-1}\over \epsilon}^\epsilon )
- \Delta t_k   \left(P_{T-t_{k-1}} \bar \L  f \right)(y_{t_{k-1\over \epsilon}}^\epsilon)\right).
\end{split}}
\end{equation*}
Define \begin{equation*}
{\begin{split}
I_k^\epsilon&=P_{T-t_k}f(y_{t_k \over \epsilon}^\epsilon)
-   P_{T-t_{k}} f(y_{t_{k-1 \over \epsilon}}^\epsilon) 
+ \Delta t_k  \left( P_{T-t_{k-1}} \bar \L  f (y_{t_{k-1\over \epsilon}}^\epsilon)\right),\\
J_k^\epsilon&= P_{T-t_k}f
-    P_{T-t_{k-1}}f -\Delta t_k   P_{T-t_{k-1}} 
\bar \L  f. 
\end{split}}
\end{equation*}
  Since $f\in B_{V,4}$,  Lemma \ref{semigroup-estimate-2} applies and obtain the desired estimate 
  on the second term:
\begin{equation*}
{\begin{split}
&\left| J_k^\epsilon(y_{t_{k-1\over \epsilon}}^\epsilon) \right|
&\le(\Delta t_k)^2  \tilde c_2(T,f)\left( 1+ \left(\lambda_{q_2} (V(y_{t_{k-1\over \epsilon}}^\epsilon)\right)\right)
%\le \epsilon  \tilde c_{2}(T,f)\left( 1+ \left(\lambda_{q_2} (V(y_{t_{k-1\over \epsilon}}^\epsilon)\right)\right).
\end{split}}
\end{equation*}
where $\tilde c_2(T,f)$ is a constant and $\lambda_{q_2}$ a polynomial.

Let $K, q$ be constants such that
$\lambda_{q_2}(V)\le K+KV^q$.
We apply  (\ref{moment-assumption-2}) from  Assumption \ref{assumption-on-rate-result} to see that for some constant $C_q(T)$ depending on $\lambda_{q_2}(V)$,
$$  \E\left(\lambda_{q_2} (V(y_{t_{k-1\over \epsilon}}^\epsilon)\right)\le K+KC_{q}(T)+KC_{q}(T) \lambda_{q} (V(y_0)).$$
Since $\Delta t_k\le \epsilon$ and $N\sim {1\over \epsilon}$,
\begin{equation}\label{rate-1}
\sum_{k=1}^N \E \left| J_k^\epsilon (y_{t_{k-1\over \epsilon}}^\epsilon) \right|
\le  \epsilon  \tilde c_2(T,f)(K+1)\left( 1+C_{q}(T)+C_{q}(T) \lambda_{q} (V(y_0))\right).
\end{equation}

If $f$ belongs to $ BC^4$, we apply Lemma \ref{semigroup-estimate-2} to see that
there exists a function $F\in B_{V,0}$, independent of $f$  s.t.
$$\left| J_k^\epsilon(y_{t_{k-1\over \epsilon}}^\epsilon) \right| 
 \le (\Delta t_k)^2\left(1+|f|_{4,\infty}\right) \left(F(y_{t_{k-1\over \epsilon}}^\epsilon)\right).$$
 Hence
\begin{equation}
\label{rate1-2}
\sum_{k=1}^N \E \left| J_k^\epsilon(y_{t_{k-1\over \epsilon}}^\epsilon) \right| 
 \le  \epsilon\left(1+|f|_{4,\infty}\right) \E\left(F(y_{t_{k-1\over \epsilon}}^\epsilon)\right).
\end{equation}
%Let $K', q'$ be constants such that $F\le K'+K'V^{q'}$.
%We apply Assumption (\ref{moment-assumption})  again to see that 
%$$  \E\left(F (y_{t_{k-1\over \epsilon}}^\epsilon)\right)
%\le K'+K'C_{q}(t)+K'C_{q'}(t) \lambda_{q'} (V(y_0)).$$
The rest of the proof is just as for the case of $f\in B_{V,4}$.

{\it Step 2.} 
Let $0\le s <t$.  By part (3) of Theorem \ref{limit-thm}, $\bar \L P_tf=P_t \bar\L f$ for any $t>0$ and 
  $P_{T-t_{k}} \bar \L  f =\bar \L P_{T-t_k}f$.
We will approximate $P_{T-t_{k-1}} \bar \L  f $ by $P_{T-t_{k}} \bar \L  f$ and estimate the error
$$\sum_{k=1}^N \Delta t_k  \left(   P_{T-t_{k}} \bar \L  f -P_{T-t_{k-1}} \bar \L  f \right)(y_{t_{k-1\over \epsilon}}^\epsilon).$$
By Lemma \ref{Lf-L2f}, $\L f\in B_{V,2}$, and we may apply 
 Lemma \ref{semigroup-estimate-2} to $\bar \L f$. We have,
\begin{equation*}
|P_{T-t_{k}}\bar \L f(y_0)-P_{T-t_{k-1}} \bar \L  f  (y_0)| \le
\Delta t_k \tilde c_1(T)\left( 1+\lambda_{q_1}(V(y_0))\right).
\end{equation*}
Recall that $\lambda_{q_1} (V)\in B_{V,0}$.
Summing over $k$ and take the expectation of the above inequality we obtain that
\begin{equation}
\label{rate-2}
\sum_{k=1}^N \Delta t_k  \left|  P_{T-t_{k}} \bar \L  f (y_{t_{k-1\over \epsilon}}^\epsilon)-P_{T-t_{k-1}} \bar \L  f(y_{t_{k-1\over \epsilon}}^\epsilon)\right|
\le \epsilon c_1(T)\left( 1+\lambda_{q_1}(V(y_0))\right).
\end{equation} 

If $f\in BC^2$, $\L f\in BC^2$. By Lemma \ref{semigroup-estimate-2} ,
 $$|P_{T-t_{k}}\bar \L f(y_0)-P_{T-t_{k-1}} \bar \L  f  (y_0)| \le
\Delta t_k \tilde c_1(T)\left( 1+\lambda_{q_1}(V(y_0))\right).$$
 there exist constant $C(T)$ and a function $\gamma_1 \in B_{V,0}$, independent  of $f$, s.t.
$$|P_tf(y_0)-P_sf (y_0)| \le (t-s) \left(1+|f|_{2,\infty}\right)\gamma_1(y_0).$$
Here $\gamma_1\in B_{V,0}$. Thus for $f\in BC^2$, 
\begin{equation}
\label{rate-2-2}
\sum_{k=1}^N \Delta t_k  \left|  P_{T-t_{k}} \bar \L  f (y_{t_{k-1\over \epsilon}}^\epsilon)-P_{T-t_{k-1}} \bar \L  f(y_{t_{k-1\over \epsilon}}^\epsilon)\right|
\le 2\epsilon |f|_{2,\infty}( 1+\gamma_1(y_0)).
\end{equation}
Finally instead of estimating  $I_k^\epsilon$, we estimate
\begin{equation*}
\label{estimate-eq-10}
D_k^\epsilon:=P_{T-t_k}f(y_{t_k \over \epsilon}^\epsilon)
-   P_{T-t_{k}} f(y_{t_{k-1 \over \epsilon}}^\epsilon) +
\Delta t_k P_{T-t_{k}} \bar \L  f (y_{t_{k-1\over \epsilon}}^\epsilon).
\end{equation*}

{\it Step 3.} 
If $f \in B_{V,4}$, by Theorem \ref{limit-thm}, $P_tf\in B_{V,4}$ for any $t$.
Since $\alpha_k\in N^\perp\cap C^3$, we may apply Lemma \ref{lemma5} to $P_{T-t_k}f$
and obtain the following formula for $D_k^\epsilon$.
\begin{equation*}
{\begin{split} 
& D_k^\epsilon=P_{T-t_k}f(y^\epsilon_{{t_k}\over \epsilon})
- P_{T-t_k}f(y^\epsilon_{t_{k-1}\over \epsilon}) + \Delta t_k P_{T-t_{k}}\bar \L f(y_{t_{k-1\over \epsilon}}^\epsilon)\\
&=\epsilon \sum_{j=1}^m \left(  dP_{T-t_k}f(Y_j(y^\epsilon_{{t_k}\over \epsilon} ) )\beta_j(z^\epsilon_{{t_k}\over \epsilon})
 - dP_{T-t_k}f(Y_j(y^\epsilon_{t_{k-1}\over \epsilon} ))\beta_j( z^\epsilon_{t_{k-1}\over \epsilon})\right)\\
&+ \Delta t_k P_{T-t_{k}}\bar \L f(y_{t_{k-1\over \epsilon}}^\epsilon) 
-\epsilon \sum_{i,j=1}^m\int_{t_{k-1}\over \epsilon}^{{t_k}\over \epsilon}
 \left( L_{Y_i}L_{Y_j} P_{T-t_k}f(y^\epsilon_r)\right)
\alpha_i(z^\epsilon_r)\;\beta_j(z^\epsilon_r)\; dr\\
&-  \sqrt \epsilon  \sum_{j=1}^m\sum_{k=1}^{m'}
\int_{s\over \epsilon}^{t\over \epsilon}  d P_{T-t_k}f( Y_j(y^\epsilon_r)) \;
d\beta_j( X_k(z^\epsilon_r)) \;dW_r^k.
\end{split}}\end{equation*}

Since $Y_0, Y_k\in B_{V,0}$, $L_{Y_i} L_{Y_j} P_{T-t_k}f\in B_{V,0}$, which follows the same argument as for
Lemma \ref{Lf-L2f}. In particular,  for each $0<\epsilon\le \epsilon_0$,
$$\int_0^{t\over \epsilon}\E \left(\left| L_{Y_i} L_{Y_j} P_{T-t_k}f(y_r^\epsilon)\right|\right)^2 dr <\infty.$$
The expectation of the martingale term in the above formula vanishes.
For $j=1,\dots, m$ and $k=1, \dots, N$, let
\begin{equation*}
{\begin{split}
A_{jk}^\epsilon &= dP_{T-t_k}f \left(  Y_j(y^\epsilon_{t_k\over \epsilon} )\right)
\beta_j(z^\epsilon_{ t_k \over \epsilon})
- dP_{T-t_k}f\left(Y_j(y^\epsilon_{t_{k-1}\over \epsilon} )\right)
 \beta_j( z^\epsilon_{t_{k-1}\over \epsilon}),\\
 B_{k}^\epsilon&=\Delta t_k (P_{T-t_{k}}\bar \L f)(y_{t_{k-1\over \epsilon}}^\epsilon)-\epsilon \sum_{i,j=1}^m\int_{t_{k-1}\over \epsilon}^{{t_k}\over \epsilon}
\left( L_{Y_i}L_{Y_j} P_{T-t_k}f\right)(y^\epsilon_r)
\alpha_i(z^\epsilon_r)\;\beta_j(z^\epsilon_r) \; dr.
\end{split}}
\end{equation*}

{\it Step 4. }
We recall that $\bar \L P_{T-t_k}f=\sum_{i,j=1}^m \overline{\alpha_i\beta_j} L_{Y_i}L_{Y_j}P_{T-t_k}f$.
By Theorem  \ref{limit-thm}, $L_{Y_i} L_{Y_j} P_{T-t_k}f$ is $C^2$.
 Furthermore by Assumption \ref{assumption1}, the $(z_t^\epsilon)$ diffusion has exponential mixing rate.
  We apply Corollary  \ref{corollary-to-lemma3-2} to each function of the form $L_{Y_i} L_{Y_j} P_{T-t_k}f$ and take $h=\alpha_i\beta_j$ There exist a constant $\tilde c$ and 
  a function $\gamma_{i,j,,k,\epsilon}\in B_{V,0}$
 such that 
 $${\begin{split} 
 & |B_k^\epsilon|
  \le {\Delta t_k}\sum_{i,j=1}^m
 \left|\overline{\alpha_i\beta_j}\;   L_{Y_i} L_{Y_j} P_{T-t_k}f \left(y_{t_{k-1}\over \epsilon}^\epsilon\right)- {\epsilon \over \Delta t_k} \int_{t_{k-1}\over \epsilon }^{t_k\over \epsilon} \E \left\{ L_{Y_i} L_{Y_j} P_{T-t_k}f (y_r^\epsilon) (\alpha_i\beta_j)(z_{r}^\epsilon) \big| \F_{t_{k-1}\over \epsilon}\right\} dr
 \right| \\
& \le  \sum_{i,j=1}^m\tilde c  |\alpha_i\beta_j|_\infty\gamma_{i,j, k,\epsilon}(y_{t_{k-1}\over \epsilon}^\epsilon) \left ({\epsilon^2}+ (\Delta t_k)^2\right),
\end{split}}
$$
where denoting $G^k_{i,j}:=L_{Y_i} L_{Y_j} P_{T-t_k}f$,
$$\gamma_{i,j,k, \epsilon}=|G^k_{i,j}| + \sum_{l'=1}^m |L_{Y_{l'}}G^k_{i,j}|+\sum_{l,l'=1}^m
{\epsilon\over \Delta t_k}  \int_{t_{k-1}\over \epsilon}^{t_k\over \epsilon} 
  \E  \left\{  \left|L_{Y_l} L_{Y_{l'}}G^k_{i,j}(y^\epsilon_r)\right | \; \big|\;\F_{s\over \epsilon} \right\}  dr.$$
By Theorem \ref{limit-thm}, $G^k_{i,j}=L_{Y_i} L_{Y_j} P_{T-t_k}f$ belong to $B_{V,2}$. Furthermore
$G^k_{i,j}$ and its first two derivatives are bounded by a function in $B_{V,0}$ which depends on $f$ only through 
$\sum_{k=0}^4 P_{T-t_k} ( |\nabla^{(k)}d f|^p)$, for some $p$.
Thus there are numbers $c ,q$ such that 
for all $k$, $\max_{i,j}|\gamma_{i,j,k, \epsilon}|\le c+cV^q$, for some $c,q$.
 Since $\Delta t_k\le \epsilon\le 1$,  $N\sim O({1\over \epsilon})$, we summing over $k$, 
\begin{equation}
\label{rate-3}
\sum_{k=1}^N \E |B_k^\epsilon|
 \le 2\epsilon \cdot c\cdot  \tilde c \sum_{i,j=1}^m  |\alpha_i\beta_j|_\infty 
   C_q(T) \sup_k  \E \left (1+V^q (y_{t_{k-1}\over \epsilon}^\epsilon) \right)
   \le \epsilon C(T)\tilde  \gamma(y_0),
\end{equation}
for some constant $C(T)$ and some function $\tilde \gamma$ in $B_{V,0}$.
If $f\in BC^4$, it is easy to see that there is a function $g\in B_{V,0}$, not depending on $f$, s.t.
$\max_{i,j,k}\E \gamma_{i,j, k,\epsilon}(y_{t_{k-1}\over \epsilon}^\epsilon) 
\le C(T)g(y_0)|f|_{4,\infty}$.

{\it Step 5.} 
Finally, by Lemma \ref{gap} below, for  $ \epsilon\le s\le t\le T$ and $f\in  B_{V,3}$, there is a constant $C$ and function $\tilde \gamma\in B_{V,0}$, depending on $T,f$
   s.t. for $0\le s <t \le T$,
\begin{equation}
\label{gap-0}
\left|  \sum_{j=1}^m \E df(Y_j(y^\epsilon_{t\over \epsilon} ) )\beta_j(z^\epsilon_{t\over \epsilon})
 -\E df(Y_j(y^\epsilon_{s\over \epsilon} ))   \beta_j( z^\epsilon_{s\over \epsilon}) \right|
 \le C\gamma(y_0)\epsilon \sqrt{|\log \epsilon|}+C\gamma(y_0) (t-s).
\end{equation}
For the partition $t_0<t_1<\dots<t_N$, we assumed that $t_1-t_0\le \epsilon$ and $\Delta t_k=\epsilon$ for $k\ge 1$.
Let $k\ge 2$. Since $dP_{T-t_k}f(Y_j)\in B_{V,3}$,  estimate (\ref{gap-0}) holds also with
$f$ replaced by  $dP_{T-t_k}f(Y_j)$, and we have:
\begin{equation}\label{final-estimate}
\left| \sum_{j=1}^m  \epsilon \E A_{jk}^\epsilon \right|\le C\tilde \gamma(y_0)\epsilon^2  \sqrt{|\log \epsilon|}, \quad k\ge 2
\end{equation}
Since $\beta_j$ are bounded and by Theorem \ref{limit-thm} $dP_{T-t_k}f$ is bounded by a function in $B_{V,0}$ that 
does not depend on $k$, for $\epsilon\le \epsilon_0$, each term 
$\E|A_{jk}^\epsilon|$ is bounded by a function in $B_{V,0}$ and $ \sup_{0<\epsilon \le \epsilon_0}|\E A_{jk}^\epsilon|$ is of order $\epsilon \tilde \gamma(y_0)$ for some function $\tilde \gamma\in B_{V,0}$. We ignore a finite number of terms in the summation. In particular we will not need to worry about  the terms with $k=1$. Since the sum over $k$ involves $O({1\over \epsilon})$ terms the following bound follows from (\ref{final-estimate}):
\begin{equation}
\label{rate-4}
\sum_{k=1}^N \left|\sum_{j=1}^m \epsilon \E A_{jk}^\epsilon\right|\le  C\tilde \gamma(y_0)\epsilon \sqrt{|\log \epsilon|}.
\end{equation}
Here $\tilde \gamma\in B_{V,0}$ and may depend on $f$.
The case of $f\in BC^4$ can be treated similarly. The estimate is of the form $\tilde \gamma(\epsilon)=(1+|f|_{4,\infty})\gamma_0$ where $ \gamma_0\in B_{V,0}$ does not depend on $f$.
We putting together (\ref{rate-1}), (\ref{rate-2}), (\ref{rate-3}) and (\ref{rate-4})to see that if $f \in B_{V,4}$, 
$$\left|\E f\left(\Phi^\epsilon_{t\over \epsilon}(y_0)\right) -P_tf(y_0)\right|\le C(T) \gamma(y_0)\epsilon  \sqrt{|\log \epsilon|},$$
where  $\gamma\in B_{V,0}$. If $f\in BC^4$, collecting the estimates together, we see that there is a constant $C(T)$ s.t.
 $$\left|\E f\left(\Phi^\epsilon_{t\over \epsilon}(y_0)\right) -P_tf(y_0)\right|
 \le\epsilon  \sqrt{|\log \epsilon|}\, C(T)\left(1+ \sum_{k=1}^4 |\nabla^{(k-1)}df|_\infty\right) \tilde\gamma(y_0)$$
where $\tilde \gamma$ is a function in $B_{V,0}$ that does not depend on $f$. 
By induction the finite dimensional distributions converge and hence the required weak convergence.
The proof is complete.
\end{proof}

\begin{lemma} 
\label{lemma4.2}
Assume that (\ref{sde-3}) are complete for all $\epsilon \in (0, \epsilon_0)$, some $\epsilon_0>0$.
\begin{enumerate}
\item[(1)] $\L_0$ is a regularity improving Fredholm operator on a compact manifold $G$,  $\alpha_k\in C^3\cap N^\perp$. 
\item [(2)] There exists $V\in C^2(M; \R_+)$, constants $c, K$, s.t. 
$$\sum_{j=1}^m |L_{Y_j} V| \le c+KV, \quad \sum_{j=1}^m |L_{Y_i}L_{Y_j}V| \le c+KV.$$

\item [(2')] There exists a locally bounded $V:M\to \R_+$ such that 
for all $q\ge 2$ and $t>0$ there are constants $C(t)$ and $q'$,  with the property that
\begin{equation}
\sup_{s\le u\le t}\E \left\{ \left(V(y_u^\epsilon)\right)^q \; \big|\; \F_{s\over \epsilon}\right \}\le C V^{q'}(y^\epsilon_{s\over \epsilon})+C.
\end{equation}

\item [(3)] For $V$ in part (2) or in part (2'), $ \sup_{\epsilon} \E V^q(y_0^\epsilon)<\infty$ for all $q\ge 2$.
\end{enumerate}
For $f\in C^2$ with the property that $L_{Y_j} f, L_{Y_i}L_{Y_j}f\in B_{V,0}$ for all $i,j$, there exists a number 
$\epsilon_0>0$ s.t.
for every $0<\epsilon\le\epsilon_0$,
$$\left|\E\left \{ f(y_{t\over \epsilon}^\epsilon) \; \big|\; \F_{s\over \epsilon}\right \}- f(y_{s\over \epsilon}^\epsilon)\right|
\le \gamma_1(y_{s\over \epsilon}^\epsilon)\max_j |\beta_j|_\infty \; \epsilon+(t-s) \gamma_2(y_{s\over \epsilon}^\epsilon)\max_i|\alpha_i|_\infty\max_j |\beta_j|_\infty.$$
Here $\gamma_1, \gamma_2\in B_{V,0}$ and depend on $|f|$ only through  $|L_{Y_j}f|$ and $|L_{Y_j}L_{Y_i}f|$. In particular there exists $\gamma \in B_{V,0}$ s.t. for all $0<\epsilon\le\epsilon_0$,
$$\left| \E f(y_{t\over \epsilon}^\epsilon)-\E f(y_{s\over \epsilon}^\epsilon)\right|
\le \sup_{0<\epsilon\le\epsilon_0} \E\gamma(y_0^\epsilon)(t-s+\epsilon).$$
Furthermore, $\sup_{0<\epsilon\le\epsilon_0}\E\left|f(y_{t\over \epsilon}^\epsilon)-f(y_{s\over \epsilon}^\epsilon)\right|\le  (\epsilon +\sqrt{t-s}))\E \gamma(y_0^\epsilon)$.
\end{lemma}

\begin{proof}
Since the hypothesis of Theorem \ref{uniform-estimates} holds, if $V$ is as defined in (2), it satisfies (2').
Since $L_{Y_j}f\in B_{V,0}$,  $\sup_{s\le t}\E|L_{Y_j}f(y_{s\over \epsilon}^\epsilon)|  ^2$ is finite. We apply Lemma \ref{Ito-tight}:
 \begin{equation*}
{\begin{split} 
\E\left\{ f(y^\epsilon_{t\over \epsilon}) \; \big|\; \F_{s\over \epsilon}\right \}
=&\ f(y^\epsilon_{s\over \epsilon}) 
+ \epsilon \sum_{j=1}^m \E\left\{ \left( df(Y_j(y^\epsilon_{t\over \epsilon} ) )\beta_j(z^\epsilon_{t\over \epsilon})
 -df(Y_j(y^\epsilon_{s\over \epsilon} ))\beta_j( z^\epsilon_{s\over \epsilon})\right) \; \big|\; \F_{s\over \epsilon}\right \}\\
&-\epsilon \sum_{i,j=1}^m \E\left\{ \int_{s\over \epsilon}^{t\over \epsilon} L_{Y_i}L_{Y_j} f(y^\epsilon_r))
\alpha_i(z^\epsilon_r)\;\beta_j(z^\epsilon_r) \;dr  \; \big|\; \F_{s\over \epsilon}\right \}.
\end{split}}\end{equation*}
Let $$\gamma_1(y_{s\over \epsilon}^\epsilon)=2 \sup_{s\le r\le t} \sum_{j=1}^m\E\left\{ |L_{Y_j}f(y_{r\over \epsilon}^\epsilon)| \; \big|\; \F_{s\over \epsilon}\right \}, \quad
 \gamma_2(y_{s\over \epsilon}^\epsilon)
= \sup_{s\le r\le t} \sum_{i,j=1}^m\E\left\{ |L_{Y_i}L_{Y_j} f(y^\epsilon_{s\over \epsilon}))| \; \big|\; \F_{s\over \epsilon}\right \}.$$ Since  $L_{Y_j}f$ and $ L_{Y_i}L_{Y_j} f\in B_{V,0}$,  $\gamma_1, \gamma_2\in B_{V,0}$.
 The required conclusion follows for there conditioned inequality, and hence the estimate for
 $\left| \E f(y_{t\over \epsilon}^\epsilon)-\E f(y_{s\over \epsilon}^\epsilon)\right|$. 
To estimate $ \E\left| f(y_{t\over \epsilon}^\epsilon)-f(y_{s\over \epsilon}^\epsilon)\right|$,
 we need to involve the diffusion term in (\ref{Ito-tight}) and hence $\sqrt{t-s}$ appears.

%  $$\sup_{0\le s\le t}\sup_{\epsilon \le \epsilon_0} \E \gamma_1^p(y_{s\over \epsilon}^\epsilon)<\infty
%  \quad \sup_{0\le s\le t}\sup_{\epsilon \le \epsilon_0} \E \gamma_2^p(y_{s\over \epsilon}^\epsilon)<\infty.$$ 

%  from the following:
%$${\begin{split}
%&\epsilon\sum_{j=1}^m\E \left\{  \left| \left( df(Y_j(y^\epsilon_{t\over \epsilon} ) )\beta_j(z^\epsilon_{t\over \epsilon})
% -df(Y_j(y^\epsilon_{s\over \epsilon} ))\beta_j( z^\epsilon_{s\over \epsilon})\right)\right| \; \big|\; \F_{s\over \epsilon}\right \}
% \le \epsilon C_1\max_j |\beta_j|_\infty;\\
%&\sum_{i,j=1}^m \epsilon  \E \left\{ \left|  \int_{s\over \epsilon}^{t\over \epsilon} L_{Y_i}L_{Y_j} f(y^\epsilon_r))
%\alpha_i(z^\epsilon_r)\;\beta_j(z^\epsilon_r) \;dr \right| \; \big|\; \F_{s\over \epsilon}\right \}
%\le C_2\cdot (t-s) \cdot \max_{i,j} (|\alpha_i|_\infty |\beta_j|_\infty).
%\end{split}} $$
%
\end{proof}

\begin{lemma}
\label{gap}

Assume the conditions of  Lemma \ref{lemma4.2} and Assumption \ref{assumption1}.
 Let $y_0^\epsilon=y_0$.
If $f\in C^3$ is s.t. $|L_{Y_j}f|$, $|L_{Y_i}L_{Y_j}f|$, $|L_{Y_l}L_{Y_i}L_{Y_j}f|$ belong to $B_{V,0}$ for all $i,j,k$,  then for some $\epsilon_0$ and all $0<\epsilon \le \epsilon_0$ and for all $0\le \epsilon \le s<t\le T$ where $T>0$,
 $$\left|  \sum_{l=1}^m \E df(Y_l(y^\epsilon_{t\over \epsilon} ) )\beta_l(z^\epsilon_{t\over \epsilon})
 -\E df(Y_l(y^\epsilon_{s\over \epsilon} ))   \beta_l( z^\epsilon_{s\over \epsilon}) \right|
 \le C(T) \gamma(y_0)\epsilon \sqrt{|\log \epsilon|}+C(T) \gamma(y_0) (t-s),$$
 where $\gamma\in B_{V,0}$ and $C(T)$ is a constant. 
If the assumptions of Theorem \ref{rate} holds, the above estimate holds for any
 $f\in B_{V,3}$; if $f\in BC^3$, we may take $\gamma=(|f|_{3,\infty}+1)\tilde \gamma $ where $\tilde \gamma\in B_{V,0}$.
\end{lemma}
\begin{proof}
Let $t\le T$.  Since $\beta_l(z^\epsilon_{t\over \epsilon})$ is the highly oscillating term, we expect that averaging in the oscillation in $\beta_l$  gains an $\epsilon$
in the estimation.
We first split the sums:
\begin{equation}
\label{7.4-1}{ \begin{split}
& \left(  df(Y_l(y^\epsilon_{t\over \epsilon} ) )\beta_l(z^\epsilon_{t\over \epsilon})\right)
 -\left(  df(Y_l(y^\epsilon_{s\over \epsilon} ))\beta_l( z^\epsilon_{s\over \epsilon}) \right)\\
 &=    df(Y_l(y^\epsilon_{s\over \epsilon} ) ) 
\left( \beta_l(z^\epsilon_{t\over \epsilon})- \beta_l(z^\epsilon_{s\over \epsilon})\right) 
+  \left(  df(Y_l(y^\epsilon_{t\over \epsilon} ) )- df(Y_l(y^\epsilon_{s\over \epsilon} ) )\right) \beta_l(z^\epsilon_{t\over \epsilon})=I_l+II_l.
 \end{split} } 
 \end{equation}
By Assumption \ref{assumption1}, $\L_0$ has mixing rate $\psi(r)=ae^{-{\delta r}}$.
 Let $s'<s\le t$,
   $${ \begin{split} &\left|\E  df(Y_l(y^\epsilon_{s'\over \epsilon} ) ) 
\left( \beta_l(z^\epsilon_{t\over \epsilon})- \beta_l(z^\epsilon_{s\over \epsilon})\right)\right|
\le \E  \left(\left|  df\left(Y_l(y^\epsilon_{s'\over \epsilon} )\right ) \right| \cdot 
\left|  {1\over \epsilon} \int_{s\over \epsilon}^{t\over \epsilon}  \E \left\{\alpha_l (z_r^\epsilon)
 \big | \F_{s'\over \epsilon} \right\}\d r \right|\right)\\
&\le  \E  \left|   df\left(Y_l(y^\epsilon_{s'\over \epsilon} )\right)\right |
{1\over \epsilon}  \int_{0}^{t-s\over \epsilon} \psi\left({r+{s-s'\over \epsilon} \over \epsilon}\right) dr\\
&\le  {a^2\over \delta} e^{-{\delta(s-s')\over \epsilon^2}} \; \E  \left|   df\left(Y_l(y^\epsilon_{s'\over \epsilon} )\right)\right|.
 \end{split} } $$
 If $s-s'=\delta_0 \epsilon^2|\log\epsilon|$,
 $\exp\left(-{\delta(s-s')\over \epsilon^2}\right)=\epsilon^{\delta\delta_0}$.
We apply Theorem \ref{uniform-estimates} to the functions $L_{Y_l}f\in B_{V,0}$. For a constant $\epsilon_0>0$,
$${a^2\over \delta}\sup_{0<\epsilon\le \epsilon_0}\sup_{0\le s'\le t}
\E\left| \left (df(Y_l(y_{s'\over \epsilon}^\epsilon)) \right) \right|
\le \tilde \gamma_l(y_0)$$
where $\tilde \gamma_l$ is a function in $B_{V,0}$, depending on $T$. Thus for $s'<s<t$,
\begin{equation}
\label{7.4-3}
\left|\E\left(   df(Y_l(y^\epsilon_{s'\over \epsilon} ) ) 
\left( \beta_l(z^\epsilon_{t\over \epsilon})- \beta_l(z^\epsilon_{s\over \epsilon})\right) \right)\right|
\le \tilde \gamma_l(y_0){a^2\over \delta} \exp{\left(-{\delta(s-s')\over \epsilon^2}\right)}.
\end{equation}
Let us split the first term on the right hand side of (\ref{7.4-1}). Denoting $s'=s-{1\over \delta}\epsilon^2|\log\epsilon|$,
 $${ \begin{split} &I_l= \E df(Y_l(y^\epsilon_{s\over \epsilon} ) ) 
\left( \beta_l(z^\epsilon_{t\over \epsilon})- \beta_l(z^\epsilon_{s\over \epsilon})\right)\\
 &=  \E  df(Y_l(y^\epsilon_{s'\over \epsilon} ) ) 
\left( \beta_l(z^\epsilon_{t\over \epsilon})- \beta_l(z^\epsilon_{s\over \epsilon})\right)
+\E \left( \left(  df(Y_l(y^\epsilon_{s\over \epsilon} ) ) -df(Y_l(y^\epsilon_{s'\over \epsilon} )) \right) 
\left( \beta_l(z^\epsilon_{t\over \epsilon})- \beta_l(z^\epsilon_{s\over \epsilon})\right) \right). \end{split} } $$
 The first term on the right hand side is estimated by (\ref{7.4-3}).
To the second term we take the supremum norm of $\beta_l$ and use  Lemma \ref{lemma4.2}.
For some $\tilde C(T)$ and $ \gamma\in B_{V,0}$,
\begin{equation}
\label{horder-rate}
\E \left|   df(Y_l(y^\epsilon_{s\over \epsilon} ) ) -df(Y_l(y^\epsilon_{s'\over \epsilon} )) \right| 
\le \tilde C(T)\gamma(y_0) \left(\epsilon+
{1\over \sqrt \delta}\epsilon |\log\epsilon|^{1\over 2}\right).
\end{equation}
Then for some number $C(T)$,
\begin{equation}
\sum_l I_l\le {1\over \sqrt \delta} \epsilon\sqrt {|\log \epsilon|}C(T)\gamma(y_0)
\end{equation} where $\gamma\in B_{V,0}$.
 Let us treat the second term on the right hand side of (\ref{7.4-1}).
  Let $t'=t-{1\over \delta}\epsilon^2|\log \epsilon|$. Then
  $${\begin{split} II_l&= \E \left(  df(Y_l(y^\epsilon_{t\over \epsilon} ) )- df(Y_l(y^\epsilon_{s\over \epsilon} ) )\right) \beta_l(z^\epsilon_{t\over \epsilon}) \\
&=\E\left(  df(Y_l(y^\epsilon_{t\over \epsilon} ) )- df(Y_l(y^\epsilon_{t'\over \epsilon} ) )\right) \beta_l(z^\epsilon_{t\over \epsilon})
+ \E \left(  df(Y_l(y^\epsilon_{t'\over \epsilon} ) )- df(Y_l(y^\epsilon_{s\over \epsilon} ) )\right) \beta_l(z^\epsilon_{t\over \epsilon}).\end{split}}$$
To the first term we apply (\ref{horder-rate}) and obtain a rate ${1\over \sqrt \delta}\epsilon\sqrt {|\log \epsilon|}$. 
We could assume that $\beta_l$ averages to zero. Subtracting the term $\bar \beta_l$ does not change $I_l$. Alternatively Lemma \ref{lemma4.2} provides an estimate of order $\epsilon$ for 
$\left| \E \left(  df(Y_l(y^\epsilon_{t\over \epsilon} ) )- df(Y_l(y^\epsilon_{s\over \epsilon} ) )\right) \right|$.
Finally, since $\int \beta d\pi=0$,
\begin{equation*}
{\begin{split}
& \left| \E \left(  df(Y_l(y^\epsilon_{t'\over \epsilon} ) )- df(Y_l(y^\epsilon_{s\over \epsilon} ) )\right) \beta_l(z^\epsilon_{t\over \epsilon}) \right|
 = \left|   \E  \left(  df(Y_l(y^\epsilon_{t'\over \epsilon} ) )- df(Y_l(y^\epsilon_{s\over \epsilon} ) )\right) \E\left\{ \beta_l(z^\epsilon_{t\over \epsilon}) \; \big | \F_{t'\over \epsilon}\right\}  \right|
\\
&\le \E  \left|  df(Y_l(y^\epsilon_{t'\over \epsilon} ) )- df(Y_l(y^\epsilon_{s\over \epsilon} ) ) \right| 
|\beta_l|_\infty  ae^{ -\delta{t-t'\over \epsilon^2}}\le \gamma_{l} (y_0)
|\beta_l|_\infty a\epsilon.
\end{split}}
\end{equation*}
In the last step we used condition (2') and  $\gamma_l$ is a function in $B_{V,0}$.
We have proved the first assertion.
  
 If the assumptions of Theorem \ref{rate} holds, for
any $f\in B_{V,3}$, the following functions belong to $B_{V,0}$:
$|L_{Y_j}f|$, $|L_{Y_i}L_{Y_j}f|$, and $|L_{Y_l}L_{Y_i}L_{Y_j}f|$. %See the proof of Lemma \ref{Lf-L2f}.
If $f\in BC^3$, the above mentioned functions can be obviously controlled by  $|f|_{3,\infty}$ multiplied by a function in $B_{V,0}$, thus completing the proof.
\end{proof}

\section{Rate of Convergence in Wasserstein Distance}
\label{Wasserstein}

Let $\B(M)$ denotes the collection of Borel sets in a $C^k$ smooth  Riemannian manifold $M$ with the Riemannian distance function  $\rho$;
let $\p(M)$ be the space of  probability measures on $M$. Let $\epsilon \in (0, \epsilon_0)$ 
where $\epsilon_0$ is a positive number. If $P_\epsilon\to P$ weakly, we may use either the total variation distance
or the Wasserstein distance, both  imply weak convergence,  to measure the rate of the convergence
of $P_\epsilon$ to $P$. Let $\rho$ denotes the Riemannian distance function.
The  Wasserstein 1-distance is
$$d_W(P,Q)=\inf_{ (\pi_1)^*\mu=P, (\pi_2)^*\mu=Q} \int_{M\times M} \rho(x,y)d\mu(x,y).$$
Here $\pi_i:M\times M\to M$  are projections to the first and the second factors respectively, and  the infimum are taken over 
probability measures on $M\times M$ that couples $Q$ and $P$. If the diameter, $\diam(M)$,   of $M$ is finite,
then the Wasserstein distance is controlled by the total variation distance,  $d_W(P,Q)\le \diam(M)\|P-Q\|_{TV}$. 
See C. Villani \cite{Villani-Optimal-transport}.

 Let us assume that the manifold has bounded geometry;
i.e. it has positive injectivity radius, $\inj(M)$, the curvature tensor and the covariant derivatives of the curvature
tensor are bounded. The exponential map from a ball of radius $r$, $r<\inj(M)$, at a point
$x$ defines a chart, through a fixed orthonormal frame at $x$.  Coordinates that consists of  
the above mentioned exponential charts are said to be canonical.
In  canonical coordinates, all transitions functions have bounded derivatives of all order. That $f$ is bounded in $C^k$ 
can be formulated as below: for any canonical coordinates and for any integer $k$, $|\partial^\lambda f| $ is bounded for 
any multi-index $\lambda$ up to order $k$. The following types of manifolds have bounded geometry:
Lie groups, homogeneous spaces with invariant metrics, Riemannian covering spaces of compact manifolds.

In the lemma below we deduce  from the convergence rate of $P_\epsilon$ to $P$ in the $(C^k)^*$ norm 
a rate in the Wasserstein distance.  Let $\rho$ be the Riemannian distance with reference to which we
speak of Lipschitz continuity of a real valued function on $M$ and the Wasserstein distance on $\p(M)$.
If $\xi$ is a random variable we denote by $\hat P_\xi$ its probability distribution.
Denote by $|f|_{\Lip}$ the Lipschitz constant of the function $f$. Let $p\in M$.  Let
$|f|_{C^k}=|f|_\infty+ \sum_{j=0}^{k-1} |\nabla^jdf|_\infty$.

\begin{lemma}\label{lemma-Wasserstein}
Let  $\xi_1$ and $\xi_2$ be  random variables on a $C^k$ manifold $M$, where $k\ge 1$, 
 with bounded geometry.  
  Suppose that for a reference point $p\in M$, $c_0:=\sum_{i=1}^2\E \rho^2(\xi_i, p) $ is finite.
 Suppose that there exist  numbers $c\ge 0, \alpha\in (0,1), \epsilon\in (0, 1]$  s.t. for $g\in BC^k$,
 $$|\E g(\xi_1)-\E g(\xi_2)|\le c\epsilon^\alpha (1+ |g|_{C^k}).$$
Then there is  a constant $C$,  depending only on the geometry of the manifold, s.t.
 $$d_W(\hat P_{\xi_1}, \hat P_{\xi_2})\le C(c_0+c)\epsilon^{\alpha\over k}.$$
\end{lemma}
\begin{proof}
If $k=1$, this is clear. Let us take $k\ge 2$ and let $f:M\to \R$ be a Lipschitz continuous function with Lipschitz constant $1$. Since we are concerned only with the difference of the values of $f$ at two points, $\left|\E f(\xi_1)-\E f(\xi_2)\right|$, we first shift $f$ so that its value at the reference point is zero. By the Lipschitz continuity of $f$, $|f(x) | \le |f|_{\Lip}\;\rho(x,p)$.  We may also assume that $f$ is bounded; if not we define a family of functions
 $f_n=(f\wedge n )\vee (-n)$. Then $f_n$ is Lipschitz continuous with its Lipschitz constant
  bounded by $|f|_{\Lip}$.  Let $i=1,2$. The correction term $(f-f_n)(\xi_i)$ can be easily controlled
by the second moment of $\rho(p, \xi_i)$:
$$\E |(f-f_n)(\xi_i)| \le \E |f(\xi_i)|\1_{\{|f(\xi_i)|>n\}} \le {1\over n} \E f(\xi_i)^2 \le {1\over n} \E\rho^2(p, \xi_i). $$

 Let $\eta: \R^n\to \R$ be a function supported in the ball  $B(x_0, 1)$ with $|\eta|_{L_1}=1$ and 
 $\eta_\delta=\delta^{-n}\eta({x\over \delta})$, where $\delta$ is a positive number and $n$ is the dimension of the manifold.  If $M=\R^n$, \begin{equation*}
{\begin{split}
& \left|\E f(\xi_1)-\E f(\xi_2)\right|\\
&\le  \left|\E (f*\eta_\delta)(\xi_1)-\E (f*\eta_\delta)(\xi_2)\right|
+\sum_{i=1}^2 \left|\E (f*\eta_\delta)(\xi_i)-\E f(\xi_i) \right|\\
&\le c\epsilon^\alpha (1+ |f*\eta_\delta|_{C^k})+2\delta |f|_{\Lip}.
\end{split}}
\end{equation*}
In the last step we used the assumption on $\E |f*\eta_\delta(\xi_1)-f*\eta_\delta(\xi_2)|$  for the $BC^k$ function $f*\eta_\delta$.
 By distributing the derivatives to $\eta_\delta$ we see that the norm of the first $k$ derivatives of $f*\eta_\delta$
are  controlled by $|f|_{\Lip}$.
If $f$ is bounded, $$c\epsilon^\alpha(1+ |f*\eta_\delta|_{C^k})\le c \epsilon^\alpha (1+ |f|_\infty+c_1 \delta^{-k+1}|f|_{\Lip}),$$
where  $c_1$ is a combinatorial constant. To summarize, for all Lipschitz continuous $f$ with $|f|_{\Lip}=1$,
$${\begin{split} \left|\E f(\xi_1)-\E f(\xi_2)\right|
&\le 2\delta|f|_{\Lip}+ c\epsilon^\alpha(1+ |f_n*\eta_\delta|_{C^k})+{c_0\over n}\\
&\le 2 \delta +c\epsilon^\alpha+ c\epsilon^\alpha n+  c_1c\epsilon^\alpha \delta^{-k+1}
+{c_0\over n} .
\end{split}}$$

 Let $\delta=\epsilon^{\alpha\over k}$. Since $k\ge 2$, we choose  $n$ with the property 
 $\epsilon^{-{\alpha\over k}}\le n\le 2\epsilon^{-\alpha+{\alpha\over k}}$, then for $f$ with $|f|_{\Lip}=1$,
 $$ \left|\E f(\xi_1)-\E f(\xi_2)\right| \le (2+2c+c_1c+2c_0)\epsilon^{\alpha\over k}.$$
 
  Let $\delta$ be a positive number with $4\delta<\inj(M)$. Let $B_x(r)$ denotes the geodesic ball centred
  at $x$ with radius $r$, whose Riemannian volume is denoted by $V(x,r)$.
 There is a countable sequence $\{x_i\}$ in $M$ with the following property:
 (1)  $\{B_{x_i}(\delta)\}$ covers $M$; (2) There is a natural number $N$ such that any point $y$ 
 belongs to at most $N$ balls from  $\{\B_{x_i}(3 \delta)\} $; i.e. the cover $\{\B_{x_i}(3 \delta)\} $ has finite multiplicity. 
 Moreover this number $N$ is independent of $\delta$. See M. A. Shubin \cite{Shubin92}. To see the independence of $N$
 on $\delta$, let  us choose a sequence $\{x_i, i\ge 1\}$ in $M$
with the property that $\{B_{x_i}( \delta)\}$ covers $M$ and $\{B_{x_i}({\delta \over 2})\}$ are pairwise disjoint.
Since the curvature tensors and their derivatives are bounded,  there is a positive number $C$ such that
$${1\over C}\le {V(x,r)\over V(y,r)}\le C, \quad x,y \in M, r\in (0, 4\delta).$$
Let $y\in M$ be a fixed point that belongs to $N$  balls of the form $B_{x_i}({\delta\over 2})$.
Since  $B_{x_i}({\delta\over 2})\subset B(y, 4\delta)$, the sum of the volume satisfies:
$\sum V(x_i, {\delta \over 2}) \le V(y, 4\delta)$ and 
${N\over C} V(y, {\delta\over 2})\le  V(y, 4\delta)$. The ratio $\sup_y{V(y, 4\delta )\over V(y, {\delta\over 2})}$
 depends only on the dimension of the manifold.
 
 Let us take a $C^k$ smooth partition of unity $\{\alpha_i, i \in \Lambda\}$
that is subordinated to $\{B_{x_i}(2\delta)\} $:  $1=\sum_{i\in \Lambda}\phi_i$, $\phi_i\ge 0$, $\phi_i$ is supported in 
$B_{x_i}(2\delta)$, and for any point $x$ there are only a finite number of non-zero summands in $\sum_{i\in \Lambda} \alpha_i(x)$. 
The partition of unity satisfies the
additional property:  $\sup_i|\partial ^\lambda \alpha_i|\le C_\lambda$, $\alpha_i\ge 0$.

Let $(B_{x_i}(\inj(M)), \phi_i)$ be the geodesic charts. 
 Let $f_i=f\alpha_i$ and  let $\tilde g=g\circ  \phi_i$ denote the representation of a function $g$ in a chart.
 \begin{equation*}
 {\begin{split}
&  \left|\E f(\xi_1)-\E f(\xi_2)\right| =  \left|\sum_{i\in \Lambda} \E \tilde f_i \left( \phi^{-1}_i(\xi_1)\right)-\sum_{i\in \Lambda}\E \tilde f_i \left( \phi^{-1}_i(\xi_2)\right)\right|\\
 &\le  \left|\sum_{i\in \Lambda} \E \tilde f_i*\eta_{\delta} \left( \phi^{-1}_i(\xi_1)\right)-\sum_{i\in \Lambda}\E \tilde f_i*\eta_{\delta}\left( \phi^{-1}_i(\xi_2)\right)\right|\\
&+  \sum_{j=1}^2\left|\sum_{i\in \Lambda} \E\tilde f_i*\eta_{\delta} \left( \phi^{-1}_i(\xi_j)\right)-\sum_{i\in \Lambda}\E \tilde f_i \left( \phi^{-1}_i(\xi_j)\right)\right|.
 \end{split}}
 \end{equation*}
It is crucial to note that there are at most $N$ non-zero terms  in the summation.  By the assumption, for each $i$,
$$\left| \E \tilde f_i*\eta_{\delta} \left( \phi^{-1}_i(\xi_1)\right)-\E \tilde f_i*\eta_{\delta}\left( \phi^{-1}_i(\xi_2)\right)\right|
\le c\epsilon^\alpha | \tilde f_i*\eta_{\delta}\circ \phi_i^{-1}|_{C^k}.$$
By construction, $\sup_i|\alpha_i|_{C^k}$  is bounded. There is a constant $c'$ that depends only on the partition of unity, such that
  $$| \tilde f_i*\eta_{\delta}\circ \phi_i^{-1}|_{C^k} \le c'| \tilde f_i*\eta_{\delta}|_{C^k}
\le  c' |\tilde f|_\infty+c'c_1\delta^{1-k}|\tilde f|_{\Lip}$$
Similarly for the second summation, we work with the representatives of $f_i$,
 $$ \left| \tilde f_i*\eta_{\delta} \left( \phi^{-1}_i(y)\right)-\tilde f_i\left( \phi^{-1}_i(y)\right)\right|
\le   \delta |\tilde f_i|_{\Lip}\le c'\delta.$$
Since we work in the geodesic charts the Lipschitz constant of $\tilde f_i$ are comparable to that of  $|f|_{\Lip}$. 
Let  $|f|_{\Lip}=1$. If $f$ is bounded,
\begin{equation*}
{\begin{split}
 \left|\E f(\xi_1)-\E f(\xi_2)\right|\le N c\epsilon^\alpha(1+c' | f|_\infty+c'\delta^{1-k})+2 c'\delta N
\end{split}}
\end{equation*}
 Let $\delta=\epsilon^{\alpha\over k}$,
$$ \left|\E f(\xi_1)-\E f(\xi_2)\right|\le Nc\epsilon^\alpha (c' | f|_\infty+1)+ Nc' \epsilon^{\alpha\over k}+2c'N\epsilon^{\alpha\over k}.$$
On a compact manifold, $|f|_\infty$ can be controlled by $|f|_{\Lip}$; otherwise we use the cut off function $f_n$ 
in place of  $f$ and the estimate $\E |(f-f_n)(\xi_i)| \le {c_0\over n} $. Choose  $n$ sufficiently large, as before,
to see that
$ \left|\E f(\xi_1)-\E f(\xi_2)\right|\le C\epsilon^{\alpha\over k}$.
Finally we apply the Kantorovich-Rubinstein  duality  theorem,
$$d_W(\hat P_{\xi_1},\hat P_{\xi_2})=\sup_{f: |f|_{\Lip\le 1} }\left\{  |\E f(\xi_1) -\E f(\xi_2)| \right\}\le  C\epsilon^{\alpha\over k},$$
to obtain the required estimate on the Wasserstein 1-distance and concluding the proof. 
\end{proof}

Let $\ev_t: C([0,T]; M)\to M$  denote the evaluation map at time $t$ : $\ev(\sigma)=\sigma(t)$.
Let $\hat P_{\xi}$ denote the probability distribution of a random variable $\xi$. Let $o\in M$.
\begin{proposition}
\label{proposition-rate}
Assume the conditions and notations of Theorem \ref{rate}. Suppose that $M$ has bounded geometry and
 $\rho_o^2 \in B_{V,0}$.
Let $\bar \mu$ be the limit measure and  $\bar \mu_t=(ev_t)_*\bar \mu$.  Then for every $r<{1\over 4}$ 
there exists $C(T)\in B_{V,0}$ and $\epsilon_0>0$ s.t. for all $\epsilon\le\epsilon_0$ and $t\le T$,
$$d_W(\hat P_{y^\epsilon_{t\over \epsilon}}, \bar \mu_t)\le C(T)\epsilon^{r}.$$
\end{proposition}
 \begin{proof}
 By Theorem \ref{rate}, for $f\in BC^4$,
 $$\left|\E f(\Phi^\epsilon_{t\over \epsilon}(y_0)) -P_tf(y_0)\right|\le C(T)(y_0)\epsilon  \sqrt{|\log \epsilon|},$$
where $C(T) (y_0)\le\tilde C(T)(y_0) (1+|f|_{C^4})$ 
for some function $\tilde C(T)\in B_{V,0}$. Since by Theorem \ref{uniform-estimates}, there exists $\epsilon_0>0$ such that 
$\sup_{\epsilon\le \epsilon_0} \E\rho^2_o (\Phi_t^\epsilon(y_0))$ is finite,
we take $\alpha$  in Lemma \ref{lemma-Wasserstein} to be any number less than $1$
to conclude the proposition.
 \end{proof}

\section{Appendix}
We began with the proof of Lemma \ref{lemma2}, follow it with a discussion on conditional inequalities without assuming conditions on the $\sigma$-algebra concerned.

\subsection*{Proof of Lemma \ref{lemma2}}
Step 1. Denote $\psi(t)=ae^{-\delta t}$. 
Firstly, if $f \in \B_b(G;\R)$ and $z\in G$, $$|Q_tf(z)-\pi f| 
\le \|f\|_W \cdot \psi(t)\cdot W(z) .$$
Next, by the Markov property of $(z_t)$ and the assumption that $\int g d\pi=0$:
\begin{equation*}
{\begin{split}
&\left| \E \{f(z_{s_2}) g(z_{s_1})|\F_{s}\}
-\int_G fQ_{s_1-s_2}g d\pi \right|\\
&=\left| \E\left\{   \left(fQ_{s_1-s_2} g\right)(z_{s_2})\Big| \F_s\right\}
-\int_G fQ_{s_1-s_2}g d\pi \right|\\
&\le \psi(s_2-s) \;\|fQ_{s_1-s_2}g\|_W \;W(z_s) 
\le \psi(s_2-s)  \sup_{z\in G}\left( { |f(z)|  |Q_{s_1-s_2}g(z) |\over W(z)}\right) W(z_s) \\
&\le  \psi(s_2-s)\psi(s_1-s_2)  |f|_\infty\, \|g\|_W W(z_s) \le a\psi(s_1-s) |f|_\infty \|g\|_W W(z_s) . \end{split}}
\end{equation*}
From this we see that,
$${\begin{split} &\left|{1\over t-s} \int_s^{t} \int_s^{s_1} \left( \E\left\{ f(z_{s_2}) g(z_{s_1}) \Big| \F_s\right\}
-\int_G fQ_{s_1-s_2}g d\pi\right) ds_2ds_1\right|\\
&\le a |f|_\infty\, \|g\|_W W(z_s) {1\over t-s} \int_s^{t} \int_s^{s_1} \psi\left(s_1-s\right)\d s_2 \d s_1\\
 &  \le {a^2\over \delta^2 (t-s)}|f|_\infty\, \|g\|_W W(z_s)\int_0^{(t-s)\delta} re^{-r} \d r
 \le {a^2\over \delta^2 (t-s)}|f|_\infty\, \|g\|_W W(z_s).\end{split}}$$
 This concludes (1). 
  Step 2. For (2), we compute the following:
$${ \begin{split}
& {1\over t-s} \int_s^{t} \int_s^{s_1}\int_G fQ_{s_1-s_2}g \d\pi \d s_2 \d s_1
=\int_G  {1\over t-s} \int_0^{t-s}fQ_rg(t-s-r)  \d r d \pi\\
&=\int_G \int_0^\infty \left(f Q_r g\right) \d r\d\pi-\int_G \int_{t-s}^\infty f Q_r g \d r \d \pi
-{1\over t-s}\int_G \int_0^{t-s} r f Q_r g \d r d \pi.
\end{split} } $$
We estimate the last two terms. Firstly,
$${ \begin{split}
&\left|\int_G \int_{t-s}^\infty f (z)Q_r g(z) \d r \d \pi(z)\right|
\le |f|_\infty \left| \int_G \int_{t-s}^\infty |Q_rg(z)| \d r  \d\pi(z)\right|_\infty\\
&  \le  |f|_\infty \|g\|_W \int_G W(z) \pi(dz) \int_{t-s}^\infty \psi(r)dr
  \le  {1\over \delta}  |f|_\infty \|g\|_W \bar W\int_{(t-s)\delta}^\infty   ae^{-r}dr\\
  &\le {a\over \delta} |f|_\infty \|g\|_W \bar W.
\end{split} } $$
It remains to calculate the following: 
$${ \begin{split}
\left|{1\over t-s} \int_G \int_0^{t-s} r f Q_r g \d r d \pi \right|
&\le {1\over t-s}  |f|_\infty \|g\|_W  \bar W \int_0^{t-s} r\psi(r)\d r\\
&\le   {a\over (t-s)\delta^2} |f|_\infty\|g\|_W \bar W.
\end{split} }$$
Gathering the estimates together we obtain the bound:
\begin{equation*}
{ \begin{split}
&\left| {1\over t-s} \int_s^{t} \int_s^{s_1}\int_G fQ_{s_1-s_2}g \d\pi \d s_2 \d s_1
-\int_G \int_0^\infty \left(f Q_r g\right) \d r\d\pi\right| \\
&\le{a\over \delta}|f|_\infty\|g\|_W \bar W+ {a\over (t-s)\delta^2} |f|_\infty\|g\|_W \;\bar W.
\end{split} }
\end{equation*}
By adding this estimate to that in part (1), we conclude part (2):
\begin{equation}
\label{estimate-1-1}
{ \begin{split}
&\left|{1\over t-s} \int_s^{t} \int_s^{s_1}  \E\left\{ f(z_{s_2}) g(z_{s_1}) \Big| \F_s\right\}
-\int_G \int_0^\infty \left(f Q_r g\right) \d r\d\pi\right| \\
&\le{a\over \delta}|f|_\infty\|g\|_W \bar W+ {a\over (t-s)\delta^2} |f|_\infty\|g\|_W \;\bar W
+{a^2\over \delta^2 (t-s)}|f|_\infty\|g\|_W W(z_s).
\end{split} }
\end{equation}
We conclude part (2).
Step 3.
We first assume that $\bar g =0$, then,
 $${ \begin{split}
&\left|{\epsilon\over t-s} \int_{s\over \epsilon}^{t\over \epsilon} \int_{s\over \epsilon}^{s_1} 
  \E\left\{ f(z^\epsilon_{s_2}) g(z^\epsilon_{s_1}) \Big| \F_{s\over \epsilon}\right\}  \d s_2 \d s_1\right|\\
 & \le\left|{\epsilon\over t-s}  \int_{s\over \epsilon}^{t\over \epsilon} \int_{s\over \epsilon}^{s_1} 
 \E\left\{ f(z^\epsilon_{s_2}) g(z^\epsilon_{s_1}) \Big| \F_{s\over \epsilon}\right\}\d s_2 \d s_1 
-\int_G  \int_0^\infty f Q_r^\epsilon g \d r \d \pi
  \right| \\
  & +\left|\int_G  \int_0^\infty f Q_r^\epsilon g \d r \d \pi\right|.
\end{split} } $$
We note that for every $x\in G$, 
$\|Q_r^\epsilon(x, \cdot)-\pi\|_{TV,W}\le \psi({ r\over \epsilon}) W(x)$.
In line (\ref{estimate-1-1}) we replace $s$, $t$, $\delta $ by ${s\over \epsilon}$, ${t\over \epsilon}$, and $ {\delta \over \epsilon}$ respectively to see the first term on the right hand side is bounded by
$${a\epsilon^3\over \delta^2 (t-s)}(  a W(z^\epsilon_{s\over \epsilon})+ \bar W) |f|_\infty\|g\|_W
+ {a\epsilon\over \delta} |f|_\infty \|g\|_W \bar W.
$$
Next we observe that
$${\begin{split} 
\int_0^\infty f(z) Q_s^\epsilon g(z) \d s&=\int_0^\infty f (z)Q_{s\over \epsilon}(z) \d s
=\epsilon \int_0^\infty f (z)Q_s g(z) \d s\\
\left|\int_G \int_0^\infty f(z) Q_s^\epsilon g(z) \d s \d \pi(z)\right| 
&\le \epsilon\,|f|_\infty  \|g\|_W \bar W  \int_0^\infty \psi(s) \d s=  {a\epsilon\over \delta}|f|_\infty \|g\|_W \bar W .
\end{split}}$$
This gives the estimate for 
the case of $\bar g=0$:
$$\left|{\epsilon\over t-s} \int_{s\over \epsilon}^{t\over \epsilon} \int_{s\over \epsilon}^{s_1} 
  \E\left\{ f(z^\epsilon_{s_2}) g(z^\epsilon_{s_1}) \Big| \F_{s\over \epsilon}\right\}  \d s_2 \d s_1\right| 
  \le  C_1(z_{s\over \epsilon}^\epsilon) {\epsilon^3 \over t-s}+C_2'(z_{s\over \epsilon}^\epsilon)\epsilon.
$$
where   $$C_1={a\over \delta^2}
(aW(\cdot)+\bar W)|f|_\infty\|g\|_W, \quad C_2'={2a\over \delta}|f|_\infty\|g\|_W\bar W.$$
If $\int g \d\pi\not =0$, we split $g=g-\bar g+\bar g$ and estimate the remaining term. We use the fact that $\pi f=0$,
 $${ \begin{split}
&\left|{\epsilon\over t-s} \int_{s\over \epsilon}^{t\over \epsilon} \int_{s\over \epsilon}^{s_1} 
  \E\left\{  f(z^\epsilon_{s_2}) \bar g\big| \F_{s\over \epsilon}\right\}  \d s_2 \d s_1\right|
 \le |\bar g|  \left|{\epsilon\over t-s} \int_0^{t-s\over \epsilon} \int_0^{s_1} 
  \left| Q^\epsilon _{s_2} f(z_{s\over \epsilon}) \right|\d s_2 \d s_1\right|\\
&  \le |\bar g\||f\|_W W(z_{s\over \epsilon}^\epsilon)
  \sup_{s_1>0} \left\{ \left|\int_0^{s_1} \psi({s_2\over \epsilon}) ds_2\right|\right\}
  \le  |\bar g| \; \||f\|_W W(z_{s\over \epsilon}^\epsilon)\epsilon \int_0^\infty \psi(r)dr\\
 & ={a\epsilon\over \delta} |\bar g| \; \|f\|_W W(z_{s\over \epsilon}^\epsilon).
  \end{split} } $$
Finally we obtain the required estimate in part (3):
$${ \begin{split}
&\left|{\epsilon\over t-s} \int_{s\over \epsilon}^{t\over \epsilon} \int_{s\over \epsilon}^{s_1} 
  \E\left\{ f(z^\epsilon_{s_2}) g(z^\epsilon_{s_1}) \Big| \F_{s\over \epsilon}\right\}  \d s_2 \d s_1\right| \\
 & \le C_1(z_{s\over \epsilon}^\epsilon)  \left({\epsilon^3\over t-s}\right)
 +C_2'(z_{s\over \epsilon}^\epsilon)\epsilon
  + \epsilon{a \over \delta}|\bar g| \; \|f\|_W W(z_{s\over \epsilon}^\epsilon), \end{split} } $$
thus concluding part (3).
\\[2em]
\noindent
{\bf Acknowledgement: } I would like to thank Michael R\"ockner for a helpful discussion.

\def\dbar{\leavevmode\hbox to 0pt{\hskip.2ex \accent"16\hss}d} \def\cprime{$'$}
  \def\cprime{$'$} \def\cprime{$'$} \def\cprime{$'$} \def\cprime{$'$}
  \def\cprime{$'$} \def\cprime{$'$} \def\cprime{$'$} \def\cprime{$'$}
  \def\cprime{$'$}

\end{document}